%% file: NS_Pressure_OC.tex
\documentclass[a4paper,dvipsnames]{m2an}
\usepackage[dvipsnames]{xcolor}
\usepackage{amsmath,amsfonts,amssymb,mathtools}
\usepackage{graphicx,subcaption}
\usepackage{braket,dsfont,url,enumerate}
\usepackage{tikz}
\usepackage{pdfpages} 
\usepackage{cmap}
\usepackage{hyperref}
\usepackage{pgfplotstable}

\pgfplotstableset{col sep=semicolon}
\usepackage{braket}
\usepackage{bbm}
\usepackage{fancyhdr}
\usepackage{times}
\usepackage{cmap}
\usepackage{pgf}
\usepackage{comment}
\usepackage{algorithm}
\usepackage{algorithmic}
\usepackage[textsize=small]{todonotes}
\usepackage[noabbrev]{cleveref}
\makeatletter
\@namedef{subjclassname@1991}{2010 Mathematics Subject Classification}
\makeatother
\definecolor{darkred}{rgb}{.7,0,0}
\definecolor{RED}{rgb}{1,0,0}

\definecolor{green}{rgb}{0,0.7,0}
\newcommand\revA[1]{\textcolor{black}{#1}}
\newcommand\revB[1]{\textcolor{black}{#1}}
\newcommand\revC[1]{\textcolor{black}{#1}}
\newcommand\revD[1]{\textcolor{black}{#1}}

\newtheorem{theorem}{Theorem}[section]
\newtheorem{lemma}[theorem]{Lemma}

\newtheorem{proposition}[theorem]{Proposition}

\newtheorem{remark}[theorem]{Remark}
\newtheorem{assumption}[theorem]{Assumption}

\input{definitions.tex}

\begin{document}
\author{Boris Vexler}\address{Chair of Optimal Control, Technical University of Munich,
  School of Computation Information and Technology,
Department of Mathematics , Boltzmannstra{\ss}e 3, 85748 Garching b. Munich, Germany
(vexler@tum.de).}
\author{Jakob Wagner}\address{Chair of Optimal Control, Technical University of Munich,
  School of Computation Information and Technology,
Department of Mathematics , Boltzmannstra{\ss}e 3, 85748 Garching b. Munich, Germany
(wagnerja@cit.tum.de) ORCID: \url{https://orcid.org/0000-0001-8510-9790}.}

\begin{abstract}
    In this work we study an optimal control problem subject to the instationary Navier-Stokes equations,
    where the control enters via an inhomogeneous Neumann/Do-Nothing boundary condition.
    Despite the Navier-Stokes equations with these boundary conditions not being well-posed for 
    large times and/or data, we obtain wellposedness of the optimal control problem by choosing a
    proper tracking type term.
    \revD{Moreover, we derive first and second order optimality conditions.}
    In order to discuss the regularity of the optimal control, state and adjoint state, 
    we present new results on $L^2(I;H^2(\Om))$ regularity of solutions to a Stokes problem
    with mixed inhomogeneous boundary conditions.
\end{abstract}

\keywords{Navier-Stokes, transient, instationary, finite elements, discontinuous Galerkin, error estimates,
best approximation, fully discrete}
\title{Optimal Control of the Navier-Stokes equations via Pressure Boundary Conditions}

\subjclass{%
  35Q30, 
  76D05,
  \revA{49J20,
  49K20}
}

\maketitle

\input{main_body.tex}
\section{Numerical Examples}\label{sect:num_exp}
\input{numericalResults_pt2.tex}

\bibliographystyle{siam}
\bibliography{Quellen.bib}
\appendix
\input{appendix.tex}
\section{Proof of \Cref{lemm:space_equivalence}}
\input{appendix_space_equivalence_proof.tex}
%
%

\end{document}

%% file: definitions.tex
\newcommand{\throwout}[1]{}

\usepackage{accents}

\newcommand{\qc}{\accentset{\circ}{q}}
\newcommand{\uuc}{\accentset{\circ}{\uu}}

\newcommand{\R}{\mathbb{R}}
\newcommand{\N}{\mathbb{N}}
\newcommand{\Om}{\Omega}
\newcommand{\IOprod}[1]{\left( #1 \right)_{I\times\Omega}}

\newcommand{\Iprod}[1]{\left( #1 \right)_{I}}
\newcommand{\IOpair}[1]{\left\langle #1 \right\rangle_{I\times\Omega}}

\newcommand{\Oprod}[1]{\left( #1 \right)_{\Omega}}
\newcommand{\Gprodj}[2]{\left( #1 \right)_{\Gamma_{#2}}}
\newcommand{\Gprod}[1]{\left( #1 \right)_{\partial \Omega}}
\newcommand{\dOprod}[1]{\left( #1 \right)_{\partial \Omega}}
\newcommand{\Opair}[1]{\left\langle #1 \right\rangle_{\Omega}}

\newcommand{\IGprodj}[2]{\left( #2 \right)_{I\times\Gamma_{#1}}}

\newcommand{\du}{\delta \uu}

\newcommand{\MathSet}[1]{\left\lbrace \, #1 \, \right\rbrace}

\newcommand{\dO}{\partial \Omega}
\newcommand{\Vt}{\widetilde \V}
\newcommand{\Ht}{\widetilde \Hspace}

\DeclareMathOperator{\ext}{ext}
\DeclareMathOperator{\supp}{supp}

\DeclareMathOperator{\tr}{tr}
\let\div\relax
\DeclareMathOperator{\div}{div}
\DeclareMathOperator{\interior}{int}
\newcommand{\uu}{\mathbf{u}}
\newcommand{\nn}{\mathbf{n}}
\newcommand{\xx}{\mathbf{x}}
\newcommand{\ttau}{\boldsymbol{\tau}}
\newcommand{\ppsi}{{\boldsymbol{\psi}}}
\newcommand{\zz}{\mathbf{z}}
\newcommand{\vv}{\mathbf{v}}
\newcommand{\ff}{\mathbf{f}}
\newcommand{\ww}{\mathbf{w}}
\newcommand{\yy}{\mathbf{y}}
\newcommand{\ee}{\mathbf{e}}
\newcommand{\EE}{\mathbf{E}}
\renewcommand{\gg}{\mathbf{g}}
\newcommand{\ubar}{\bar \uu}
\newcommand{\zbar}{\bar \zz}
\newcommand{\qbar}{\bar q}
\newcommand{\dq}{\delta q}
\newcommand{\su}{\textit{\S}\uu}

\renewcommand{\rq}{\varrho q}
\newcommand{\ru}{\varrho \uu}

\renewcommand{\textin}{\text{ in }}
\newcommand{\texton}{\text{ on }}
\newcommand{\dt}{\mathrm{d}t}
\newcommand{\rhq}{\varrho q}
\newcommand{\rhu}{\varrho \uu}

\newcommand{\ltwoltwonorm}[1]{\left\| #1 \right\|_{L^2(I;L^2(\Omega))}}
\newcommand{\W}{\mathbb{W}}
\newcommand{\VV}{\mathbb{V}}
\newcommand{\cimbed}{\overset{c}{\hookrightarrow}}
\newcommand{\En}{\mathrm{E}_{\partial_n}}

\newcommand{\V}{\mathbf{V}}
\newcommand{\Hspace}{\mathbf{H}}
\newcommand{\alev}{\text{ a.e. }}
\newcommand{\GNi}{\Gamma_{N,i}}
\newcommand{\oo}{\mathbf{0}}
\newcommand{\Qad}{\mathcal{Q}_{ad}}
\newcommand{\Qbox}{\underline{\overline{Q}}^b_a}
\newcommand{\Hhalfoodual}{(H^{\frac12}_{00}(\Gamma_{N,i}))^*}
\renewcommand{\textforall}{\text{for all }}

%% file: main_body.tex
\section{Introduction}
In this work, we aim to consider an optimal control problem for the Navier-Stokes equations of the form
\begin{align}
   \tag{$\mathcal{P}$}\label{eq:minProb_1}
   \min_{(\uu,p,q)} & \qquad  J(\uu,\revA{\uu_d},p,q)
   \\
   \nonumber
   \text{s.t.} & \qquad
   \begin{aligned}[t]
     \partial_t \uu -  \Delta \uu + (\uu\cdot \nabla) \uu + \nabla p &= \oo && \text{in } I\times \Omega,\\
     \nabla \cdot \uu &= 0 && \text{in } I \times \Omega,\\
     \uu &= \oo && \text{on } I \times \Gamma_D,\\
     g(\uu,p) &=q_i(t)\nn && \text{on } I \times \GNi, \quad i=1,...,L,\\
     \uu(0) &=\uu_0 && \text{in } \Omega.
    \end{aligned}
\end{align}
Our focus lies on finding formulations of the objective function $J$ and 
boundary conditions $g$ such that this becomes a well defined and meaningful problem.
In the argument of $J$, the term $\uu_d$ denotes a target velocity field, obtained, e.g.,
from measurements.
In the above formulation, we consider $\Omega \subset \R^2$ some given domain, where its boundary
$\partial \Omega =:\Gamma= \Gamma_D \cup \bigcup_{i=1}^L \GNi$ is partitioned into a wall part of the boundary
$\Gamma_D$, where we prescribe no-slip (Dirichlet) boundary conditions,
and $L$
disjoint boundary segments $\GNi$, \revB{on which} we prescribe some boundary conditions
\revD{involving} the velocity $\uu$ and the pressure $p$, \revD{and the outer normal vector $\nn$}.
As fluid may enter or leave $\Omega$ through the boundary segments $\GNi$, they are often 
referred to as \textit{open boundaries}.
\revB{The equations are considered on the time interval $I$ and the initial condition
$\uu_0$ is supposed to be given.}
We \revA{wish to} control the Navier-Stokes equations by the right hand sides $q_i$ of these 
boundary conditions.
It is quite common, that the boundary segments $\GNi$ stem from truncating a larger 
domain, in order to reduce the total number of degrees of freedom of a full model, and
to facilitate numerical computations.
One important application, in which such domains arise, are cardiovascular CFD simulations, 
where $\Om$ contains one or more bifurcations of a bloodvessel. 
These simulated data can be used by clinicians to diagnose the severity of 
cardiovascular diseases, which makes this field of research highly impactful. We refer the reader to 
the survey \cite{nolte_inverse_2022}, where a vast collection of relevant problems and solution 
approaches in this context is presented. In order to obtain CFD data of the bloodflow in a truncated domain,
proper values for the boundary conditions at the open boundaries have to 
be specified. As these are generally unknown, optimal control problems like \eqref{eq:minProb_1} 
facilitate detailled CFD simulations, provided that a patient specific scan of the bloodvessel 
geometry $\Om$ and space time velocity data $\uu_d$ are available. 
Both can be obtained non-invasively by the use of 4D MRI imaging, 
see, e.g., \cite{stankovic_4d_2014} for an introduction to this method.
There, the authors give references on how other flow quantities like pressure and wall sheer stress 
can be indirectly deduced from the measured data.
Solving an optimal control formulation like \eqref{eq:minProb_1} makes all these quantities available 
immediately without further need of postprocessing the MRI data.
Such time dependent boundary control problems are among the \revD{methods} presented
in \cite{nolte_inverse_2022}
and in Section 6.5 therein, the authors remark, that a wellposedness analysis of such formulations
is an open question. The techniques and results presented in this work contribute 
towards answering this question.

\revD{Throughout this work, we shall assume, that the controls $q_i$ are 
  only time dependent and constant in space.
This structure is partially motivated by the Poiseuille flow in a rectangular channel, 
where the inhomogeneous Do-Nothing boundary data are indeed constant in space. 
At the same time, we will see, that this allows us to circumvent some tedious technicalities
in the analysis of the problem under consideration, while still keeping the main difficulties
to solve. Lastly, from a numerical point of view, the pure time dependence greatly reduces the
number of degrees of freedom of the discretized control. Let us also specifically mention
\cite[Section 2]{nolte_inverse_2022} and \cite{quarteroni_coupling_2001}, where 
the time dependent Navier-Stokes equations are bi-directionally coupled to an ODE model on each
open boundary segment. These ODE models, also called Windkessel models, describe the lumped 
impact, e.g., resistance and capacity, of the truncated smaller bloodvessels, on the main flow 
in $\Omega$. As each ODE component only carries time-dependent information, it is necessary to 
prescribe a spatial profile, e.g., constant, to couple the models. In turn, only the parameters 
of the Windkessel models have to be optimized, even further reducing the number of degrees of 
freedom. We hope that the analysis of \eqref{eq:minProb_1} with purely time dependent controls,
may in the future contribute towards the analysis of optimal control problems subject to such
coupled systems.}

Not only the question of specifying the correct data at the open boundaries, but also the choice 
of boundary condition itself is a very active reasearch topic, 
and a large number of possible boundary conditions is being discussed.
At the present point in time, the following conditions seem to be the most used in the literature
%
\begin{align}
   & \text{Do-Nothing Boundary Condition:} & 
    p\nn -  \nabla \uu \cdot \nn &= q_i \nn \label{eq:DN}\tag{DN}\\
   & \text{Total Pressure Boundary Condition:} & 
    p\nn -  \nabla \uu \cdot \nn + \frac{1}{2} |\uu|^2 \nn &= q_i \nn
   \label{eq:TP}\tag{TP}\\
   & \text{Advective Boundary Condition:} & 
    p\nn -  \nabla \uu \cdot \nn + \frac{1}{2} (\uu\cdot \nn) \uu &= q_i \nn \label{eq:AD}\tag{AD}\\
   & \text{Directional Do-Nothing Boundary Condition:} &
    p \nn - \nabla \uu \cdot \nn + \frac{1}{2} [\uu \cdot \nn]^- \uu &= q_i \nn \label{eq:DDN}\tag{DDN}
\end{align}
\revB{where in the last condition, $[\, \cdot \, ]^- = \min\{\cdot,0\}$ denotes the negative part.}
The \eqref{eq:DN} and \eqref{eq:TP} conditions have been discussed in \cite{heywood_artificial_1996},
the \eqref{eq:DDN} conditions were
introduced in \cite{braack_directional_2014}, and for the \eqref{eq:AD} boundary conditions,
we refer the reader to
\cite{simon_convective_2022}.
It is well known, that while the instationary Stokes equations with the \eqref{eq:DN}
boundary condition are well posed,
the same is not true anymore for the Navier-Stokes equations. Only for short times or small data, the existence
of a weak solution is known. The reason for this is the fact, that the term $\Oprod{\uu\cdot \nabla \uu,\uu}$
in the tested weak formulation collapses to a boundary integral term, which a priori cannot be bounded.
Thus, it is possible that an infinite amount of inflow through $\Gamma_N$ occurs, making the total energy 
of the system unbounded.
The modifications present in the \eqref{eq:TP}, \eqref{eq:AD} and \eqref{eq:DDN} boundary conditions aim to cancel out this boundary term,
in order to retrieve existence of weak solutions globally in time for data of any size.
In the literature, there are also other strategies to overcome the difficulty in applying the \eqref{eq:DN} boundary conditions.
In \cite{kracmar_modeling_2018} the authors search for solutions to the Navier-Stokes equations with
\eqref{eq:DN} boundary conditions in a subset of the solution space, such that the inflow remains bounded.
This allows them to consider the equations in form of a variational inequality. 
In \cite{beirao_da_veiga_navierstokes_2024} the authors explore regularized \eqref{eq:DDN} conditions for the stationary 
Navier-Stokes equations.
While none of the above boundary conditions carry a physical meaning at parts of the boundary, that are
introduced by truncating a larger domain, the \eqref{eq:DN} boundary conditions are at least valid 
for Poiseuille flow in a rectangular channel. 
Moreover, while the additional terms in the \eqref{eq:TP}, \eqref{eq:AD} or \eqref{eq:DDN} boundary condition allow for a canceling of terms 
in the state equation, the same is not true anymore in the adjoint equation that arises as part of the 
discussion of the optimal control problem.
Lastly, choosing the same type of boundary condition at all \revB{open boundaries} eliminates the need 
to distinguish inflow and outflow boundaries a priori, as all parts of the boundary in question
are treated the same.
Thus, our aim is to work with the \eqref{eq:DN} boundary conditions and to devise a strategy to overcome the 
difficulties associated with that choice mentioned above.
Since the Navier-Stokes equations are embedded into the optimal control problem, we do not 
require solvability of the equations for arbitrary data, but get existence of solutions along 
a minimizing sequence.
Additionally, we have the added freedom of choosing the norm in the tracking type term in the cost functional,
and thus the a priori known
smoothness of the solutions in the minimizing sequence.
This strategy of choosing a stronger norm in the tracking term was recently employed in
\cite{casas_analysis_2016}, in order to tackle a distributed control
problem of the 3D Navier-Stokes equations, \revD{with homogeneous Dirichlet boundary conditions}.
There the stronger norm guaranteed uniqueness
of solutions, whereas in our case, we will use it to obtain existence of solutions.
The origin of such approaches can be traced back to as early as \cite{gunzburger_analysis_1991},
where a Dirichlet control problem of the stationary Navier-Stokes was considered.
\revB{While there is a lot of literature on distributed control 
  (e.g., \cite{casas_discontinuous_2012,casas_velocity_2024})
  or boundary control (e.g., \cite{hinze_second_2004}) of the Navier-Stokes equations with pure 
Dirichlet boundary conditions, much less is available for mixed boundary conditions.}
The authors of \cite{guerra_existence_2015,tiago_velocity_2017,mondaini_boundary_2018} 
use the standard $L^2$ tracking norm in a Dirichlet control problem for the stationary Navier-Stokes
equations, but additionally admit a homogeneous \eqref{eq:DN} condition on a subset of the boundary.
This forces them to impose a constraint on the size of the controls,
in order to get solvability of the state equation.
In \cite{guerra_optimal_2014,guerra_optimal_2024} the authors have extended these approaches to the
instationary Navier-Stokes equations.
\revB{In the very recent work \cite{rathore_projection_based_2024}, the authors study optimal control problems
  which are very similar to ours, but consider the 3D instationary case and admit fully space-time dependent 
  controls on the right hand side of the \eqref{eq:DN} conditions. While they do not present theoretical 
  results concerning the optimal control problems, they implement model order reduction techniques
  to overcome the large computational effort necessary to solve such problems.
  \\
}
The mixed boundary conditions in the state equation of \eqref{eq:minProb_1} require careful analysis
of the available regularity of the primal and adjoint states.
\revB{The work \cite{rosch_regularity_2005} highlights, that even for pure Dirichlet conditions,
the adjoint equation needs special attention.}
It is known that the available regularity on piecewise smooth domains depends crucially on the angle 
between boundary segments. Moreover, in certain special cases additional compatibility conditions 
between the Neumann and Dirichlet data need to be imposed.
While for the Poisson problem results for mixed inhomogeneous boundary conditions and a piecewise smooth
boundary can be found in \cite{banasiak_mixed_1989}, for the parabolic case there seems to be 
much less known.
In \cite{denk_optimal_2007} the authors discuss maximal parabolic regularity for parabolic problems 
with inhomogeneous data in $C^2$ domains  for either inhomogeneous Neumann or inhomogeneous Dirichlet 
boundary conditions. The line of research is continued in \cite{lindemulder_maximal_2020}, discussing the 
regularity in weighted spaces, indicating the relevance of this area of research.
In \cite{hieber_quasilinear_2008,hallerdintelmann_maximal_2009} parabolic systems with mixed boundary
conditions in nonsmooth domains are discussed, under the assumption of homogeneous Neumann data.
In \cite{savare_parabolic_1997} parabolic problems with mixed inhomogeneous data 
on $C^{1,1}$ domains are discussed, where an emphasis is put on the setting,
that the type of boundary condition at any point $\xx$ is allowed to change over time.
This line of research is carried on in \cite{choi_optimal_2022}, which serves as a great collection of 
references in the topic. In their paper, the authors work with homogeneous Dirichlet and very special Neumann boundary conditions, for which they show, that the standard $L^2(I;H^1(\Omega))$ and $L^\infty(I;L^2(\Omega))$
regularities hold.
The authors of \cite{kim_existence_2018} also fall in that line of research, studying mixed inhomogeneous
parabolic problems on time dependent domains. However, they consider $C^2$ spatial domains and work with a 
weaker notion of solutions to the PDE.
In \cite[Chapter 4, section 2.2 \& 2.5 ]{lions_nonhomogeneous_1972_vol2} instationary trace results 
for simultaneous Neumann and Dirichlet data on the whole boundary are presented,
which hold on $C^\infty$ domains.
Moreover, the authors discuss compatibility conditions with the initial condition.
For results on the Stokes and Navier-Stokes equations, the literature is even more scarce in the general case.
Results for the stationary problems in 3D with either pure \eqref{eq:DN} or mixed boundary conditions can be found in 
\cite{mitrea_stokes_2011,brown_mixed_2009,cialdea_mixed_2009}.
The well known monograph \cite{mazia_elliptic_2010} present results for the stationary Stokes
and Navier-Stokes problems
in a 3D polyhedron with mixed inhomogeneous boundary conditions, howewer exclude the case of right angles 
between faces, see their Lemma 9.2.4.
The authors of \cite{kozlov_neumann_2023} consider a pure Neumann instationary Stokes problem in a 
general cone in weighted spaces.
Local existence in time of solutions to the instationary Navier-Stokes equations with mixed homogeneous boundary 
conditions in 3D polyhedral domains is shown in \cite{benes_mixed_2011}.
In \cite{kucera_basic_2009} instationary
2D and 3D Stokes and Navier-Stokes problems are considered with mixed inhomogeneous boundary data.
A slightly different concept of weak solutions, possessing more time regularity, is introduced and 
unique solvability to perturbed equations in the vicinity of an (assumed to be existing) solution is shown.
In \cite{kim_non_stationary_2017}, for 2D and 3D domains and a variety of boundary 
conditions, existence of a solution to the instationary Navier-Stokes equations for small enough data is proven.
In order to work with their concept of more regular weak solutions, the authors need to assume improved 
regularity of the data.
For domains with piecewise smooth boundaries, where boundary segments meet in a right angle, 
\cite{benes_solutions_2012,benes_solutions_2016} show $H^2$ regularity results of the instationary 
Stokes and Navier-Stokes equations in 2D and 3D. Here the results for Navier-Stokes hold in a neighborhood
of an assumed to exist solution. 
Local existence in time is shown \revB{in} \cite{benes_local_2021} for the case that the data have slightly
more than
standard regularity, improving the results of \cite{kucera_local_1998} where much more regularity of the data 
was needed.
Let us lastly comment on literature where the Navier-Stokes equations with mixed boundary conditions arise
as a subproblem.
The authors of \cite{lasiecka_boundary_2018} consider a 2D stationary FSI problem in a
rectangle. The appendix contains a detailled explanation on how to treat corner 
singularities for the Stokes problem. The same line of research is continued in 
\cite{hintermuller_differentiability_2023}.
On 2D domains with piecewise smooth boundaries joined by right angles,
\cite{benes_strong_2011,arndt_existence_2020} show existence 
of solutions to a Boussinesq system either locally in time or for small data.
The authors of \cite{ceretani_boussinesq_2019} consider a stationary Boussinesq system
and augment it with either the ideas of \cite{kracmar_modeling_2018} or the \eqref{eq:DDN} conditions.
The article is a great survey \revB{over different approaches to treat open boundaries.}

The remainder of the paper is structured as follows. In \Cref{sect:preliminary} we introduce some notations and recollect preliminary results for the involved function spaces.
Following up, in \Cref{sect:state}, we discuss the instationary 
Navier-Stokes equations with mixed boundary conditions,
which pose the state equations of our optimal control problem.
We carefully examine the existence\revD{/non-existence}, uniqueness and regularity of solutions, as well as introduce 
a suitable reformulation of the state equation for the specific case of purely time-dependent data.
\Cref{sect:opt_cont} then covers the optimal control problem.
Our main result is \Cref{thm:wellposedness_ocp}, in which we show that, despite the state equation not 
being solvable for arbitrary data, our formulation leads to a well posed optimal control problem.
We will turn towards first order optimality conditions involving the adjoint state and answer the 
question of regularity of the primal and dual variables.
While the regularity of the state follows from the discussion in \Cref{sect:state}, for the discussion of the 
adjoint state, we prove a general result on $L^2(I;H^2(\Om)) \cap H^1(I;L^2(\Om))$ regularity
of solutions to the Stokes equations with mixed inhomogeneous data. This result is presented in 
\Cref{thm:stokes_inhom_neumann_data} and is of significance also outside the context of our specific optimal 
control problem. To the best of the authors knowledge, there are no results in the literature on 
compatibility conditions for the inhomogeneous mixed boundary data of instationary Stokes problems 
in nonsmooth domains.
We conclude the section with a discussion of second order conditions and proceed to 
\Cref{sect:num_exp} in which we show numerical results.

\section{Preliminary}\label{sect:preliminary}

Throughout this work, all vectorial quantities will be indicated by boldface letters.
For $1 \le p\le \infty$ and $k \in \mathbb N$,
we denote by $L^p(\Omega)$, $W^{k,p}(\Omega)$, $H^k(\Omega)$ and $H^1_0(\Omega)$
the usual Lebesgue and Sobolev spaces.
The inner product on $L^2(\Omega)$ will be denoted by $\Oprod{\cdot , \cdot}$.
For $s \in \R\backslash \N$, $s>0$ \revB{and $1 \le p < \infty$,} the fractional order Sobolev(-Slobodeckij)
space $W^{s,p}(\Omega)$ is defined, see, e.g., \cite{demengel_functional_2012}, as
\begin{equation*}
  W^{s,p}(\Omega) := \MathSet{ v \in W^{\lfloor s\rfloor,p}(\Omega) : 
  \sum_{|\alpha|=\lfloor s \rfloor} \iint_{\Omega \times \Omega} 
  \dfrac{|D^\alpha v(\xx) - D^\alpha v(\yy)|^p}{|\xx-\yy|^{(s-\lfloor s\rfloor)p+2}}
  dx dy < + \infty}.
\end{equation*}
In case $p=2$ we again use the notation $H^s(\Omega)$. Note that in this case $H^s(\Omega)$ 
can equivalently be obtained via real or complex
interpolation of the integer degree spaces $H^k(\Omega)$.
This is due to the fact, that in the Hilbert space setting, all resulting Bessel potential spaces
$H^s_2(\Omega)$, Besov spaces $B^s_{2,2}(\Omega)$ and Sobolev-(Slobodeckij) spaces $H^s(\Omega)$ coincide,
see \cite[pp. 12,39]{triebel_theory_1992}.
We denote for a Banach space $X$, the Bochner spaces of $p-$integrable, $X-$valued functions on $I$ by 
$L^p(I;X)$. Accordingly, we use the notations 
$W^{k,p}(I;X)$ and $H^k(I;X)$ for the \revD{Bochner type} spaces, where all time derivatives
up to order $k$ are also $p$- (respectively $2$-) integrable.
Let us recall some important relations between spaces, that can be obtained by interpolation.
  We denote the complex interpolation functor for two compatible Banach spaces $X$,$Y$ and parameter 
  $\theta \in [0,1]$ by $[X,Y]_\theta$. There holds an estimate
  \begin{equation*}
    \|x\|_{[X,Y]_\theta} \le \|x\|_X^{1-\theta} \|x\|_Y^\theta \quad \textforall x \in X \cap Y.
  \end{equation*}
  If interpolation between spaces is applied repeatedly, the reiteration theorem
  \cite[Chapter 1, Theorem 2.3]{lions_nonhomogeneous_1972_vol1} gives that   
  \begin{equation}\label{eq:reiteration_theorem}
     [[X,Y]_{\theta_1},[X,Y]_{\theta_2}]_\sigma 
     = [X,Y]_{\theta_1 \cdot (1-\sigma) + \theta_2 \sigma} \quad \text{ for } 
     \theta_1,\theta_2, \sigma \in (0,1).
  \end{equation}
  \revB{Following \cite{lions_nonhomogeneous_1972_vol1}, we define for any $s \in \R$, $s > 0$,
  and a Hilbert space $X$, the space $H^s(I;X) := [H^m(I;X),L^2(I;X)]_\theta$, where $(1-\theta)m = s$.}
  In order to deal with Bochner integrable functions, we recall the
  intermediate derivative theorem, which states that 
  \begin{equation}\label{eq:intermediate_derivative_theorem}
     u \in L^2(I;X) \cap H^s(I;Y) \Rightarrow
     u \in H^{s\theta}(I;[X,Y]_\theta), \text{ for all } \theta \in (0,1),
  \end{equation}
  and can be found for integer order derivatives in 
  \cite[Chapter 1, Theorem 2.3]{lions_nonhomogeneous_1972_vol1}, and for fractional derivatives in 
  \cite[Chapter 1, Theorem 4.1]{lions_nonhomogeneous_1972_vol1}.
  Secondly, as taking derivatives w.r.t. time and space commutes, due to 
  \cite[Chapter 1, Theorem 13.1]{lions_nonhomogeneous_1972_vol1}, there holds 
  for any $s,\sigma \ge 0$ and $\theta \in (0,1)$:
  \begin{equation}\label{eq:space_time_intersection_interpolation}
    [L^2(I;H^s(\Om)) \cap H^\sigma(I;L^2(\Om)),L^2(I;L^2(\Om))]_\theta
    = L^2(I;H^{s(1-\theta)}(\Om)) \cap H^{\sigma \theta}(I;L^2(\Om)).
  \end{equation}
  Lastly, when $X,Y$ are Hilbert spaces, from \cite[Chapter 1, Equation 9.24]{lions_nonhomogeneous_1972_vol1},
  we get the respresentation
  \begin{equation*}
     [H^{s_1}(I;X),H^{s_2}(I;Y)]_\theta = H^{(1-\theta) s_1 + \theta s_2 }(I;[X,Y]_\theta)
     \quad \text{ for } s_1,s_2 \ge 0, \theta \in (0,1),
  \end{equation*}
  \revB{which by setting $s_1=0$ can be seen as a generalization of \eqref{eq:intermediate_derivative_theorem}.}
For $X$ being any function space over $\Omega$, we denote by $X^*$ its topological dual space,
and abbreviate the duality pairing by $\Opair{\cdot , \cdot}$. We will also use the notation
$H^1_0(\Omega)^* = H^{-1}(\Omega)$.
We follow the setting of \cite{benes_solutions_2016} regarding the assumptions on the spatial domain
$\Omega$:
\begin{assumption}\label{ass:domain}
  Throughout this work, let $\Omega \subset \R^2$ be a Lipschitz domain, satisfying the following
  conditions:
  \begin{enumerate}
    \item the boundary $\partial \Omega$ is partitioned into Dirichlet and Neumann boundaries
      $\partial \Omega = \bar \Gamma_D \cup \bar \Gamma_N$ such that $\Gamma_D \cap \Gamma_N = \emptyset$.
    \item the connected components of $\Gamma_D$ are $C^2$,
    \item the connected components of $\Gamma_N$ are line segments,
    \item at the points of transition from $\Gamma_D$ to $\Gamma_N$, the boundary segments form an angle of 
      $\frac{\pi}{2}$.
  \end{enumerate}
\end{assumption}
\begin{figure}[H]
  \centering
  \begin{tikzpicture}
  \node (image) at (.4,-.3) {\includegraphics[width=2.2cm]{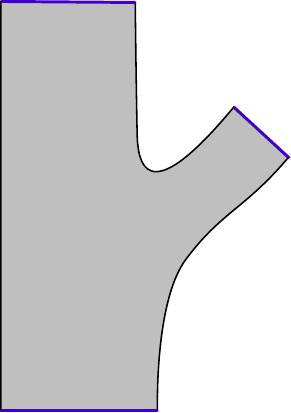}};
  \node at (-.1,-.2) {$\Omega$};
  \node[blue] at (-.1,-2.1) {\small $\Gamma_{N}$};
  \node[blue] at (-.1,1.5) {\small$\Gamma_{N}$};
  \node[blue] at (1.5,0.5) {\small $\Gamma_{N}$};
  \node at (-1.1,-.2) {\small $\Gamma_{D}$};
  \node at (1.1,-.9) {\small $\Gamma_{D}$};
  \node at (.75,.7) {\small $\Gamma_{D}$};
  \end{tikzpicture}
  \caption{Domain $\Omega$ satisfying the \Cref{ass:domain}}
  \label{fig:domain_assumption}
\end{figure}
Note, that in \cite{benes_solutions_2016}, $C^\infty$ of the connected components of $\Gamma_D$ is assumed,
while $C^2$ regularity is indeed enough, compare also the setting of \cite{benes_local_2021}.
Below, we use the spaces:
\begin{equation}\label{eq:spaces1}
  \V := \left\{ \uu \in H^1(\Omega)^2 : \nabla \cdot \uu = 0 \land \uu|_{\Gamma_D}=\oo \right\}
  \quad \text{and} \quad
  \Hspace := \left\{ \uu \in L^2(\Omega)^2 : \nabla \cdot \uu = 0 \land \uu\cdot \nn|_{\Gamma_D} = 0 \right\}.
\end{equation}
By $\nn$ we denote the outer normal vector to $\dO$, and understand $\uu \cdot \nn$ as the normal trace,
which is well defined for solenoidal $L^2(\Omega)$ vector fields,
see \cite[Lemma II.1.2.2]{sohr_navier_stokes_2001}.
In case $\revB{\Gamma_D = \dO}$, these spaces coincide with the ones used in standard literature,
e.g., \cite{Temam1977}.
   As we will also need the space of divergence free functions with zero trace on the whole boundary
   in our analysis, we will use the following notation
   \begin{equation}\label{eq:def_Vcirc}
      \mathring{\V} := \{ \uu \in H^1_0(\Omega)^2 : \nabla \cdot \uu = 0\}.
   \end{equation}
Note however, that in the pure Dirichlet case, there exist alternative, equivalent characterizations
of these spaces, defined via closures. For our mixed boundary condition setting, these spaces would 
be the following:
\begin{equation}\label{eq:spaces2}
  \Vt := \overline{\mathcal V}^{H^1(\Omega)}, \quad
  \Ht := \overline{\mathcal V}^{L^2(\Omega)} \quad \text{where} \quad
  \mathcal V := \left\lbrace \uu \in C^\infty(\overline{\Omega})^2 :
   \nabla \cdot \uu = 0 \land
   \supp \uu \cap \bar \Gamma_D = \emptyset \right\rbrace.
\end{equation}
While in this work, we choose to work with the spaces $\V$ and $\Hspace$, it is a natural question to ask, 
when these descriptions coincide. 
Most of the standard literature, \cite{Temam1977,sohr_navier_stokes_2001,john_finite_2016,
girault_finite_1986,boyer_mathematical_2013} deals with the pure Dirichlet case, where it is well 
established, that for Lipschitz domains the two characterizations are equivalent.
In \cite{mitrea_stokes_2011,mitrea_boundary_2012}
problems with pure Neumann conditions are studied, and it is shown, that $\Hspace = \Ht$.
Their argument uses extension results of solenoidal vector fields, derived in \cite{kato_extension_2000},
but is not directly applicable to the case where additional Dirichlet conditions are present.
Regarding mixed problems, \cite{lanzendorfer_pressure_2011,braack_directional_2014,nguyen_boundary_2015} 
work with spaces obtained by restricting the containing Sobolev spaces, while
\cite{benes_local_2021,benes_solutions_2016,benes_mixed_2011,benes_strong_2011,benes_solutions_2012,
kim_equations_2021,simon_convective_2022} define their spaces via closures of $C^\infty$ functions.
Let us lastly mention, that this question regarding the characterization of the proper solenoidal space
is closely related to the existence of a Helmholtz decomposition. The work \cite{amrouche_decomposition_1994}
contains results on various Helmholtz decompositions of spaces for the Stokes problem subject to Dirichlet
boundary conditions.
In \cite{amrouche_stokes_2011}, the authors give a constructive approach to the proper Helmholtz
decompositions. Such a Helmholtz decomposition, presented in \cite[Lemma II.2]{nguyen_boundary_2015},
gives us a characterization of the space $\Hspace^\perp$, it holds
\begin{equation}\label{eq:Hperp}
  L^2(\Omega)^2 = \Hspace \oplus \Hspace^\perp, \ \text{where} \ 
  \Hspace^\perp = \{\nabla p: p \in H^1(\Omega) \land p|_{\Gamma_N} = 0\}.
\end{equation}
With this characterization of the orthogonal space of $\Hspace$, we \revA{have} the following lemma.
\begin{lemma}\label{lemm:space_equivalence}
The two characterizations of solenoidal spaces, defined in \eqref{eq:spaces1} and \eqref{eq:spaces2}
are equivalent, i.e. it holds $\Hspace = \Ht$ and $\V = \Vt$.
\end{lemma}
A proof of this result is given in \Cref{sect:app_proof_space_equivalence}.
With these spaces, we can introduce the mixed boundary condition Stokes operator by
$A\colon \mathcal D(A) \to \Hspace$
\begin{equation}\label{eq:def_Stokes_operator}
   \Oprod{A \uu,\vv} = \Oprod{\nabla \uu, \nabla \vv} \quad \textforall \vv \in \V.
\end{equation}
Its domain is given by $\mathcal D(A) = \{\uu \in \V: \exists \ff \in \Hspace \text{ s.t. } A \uu = \ff\}$.
We introduce the following two spaces of space and time dependent functions,
\revB{which we will repeatedly use in our regularity discussions.}
\begin{align}
  \label{eq:space_VV_definition}
  \VV & := L^2(I;\V) \cap H^1(I;\V^*) && \hspace{-3.6cm} \hookrightarrow C(\bar I;\Hspace),\\
  \label{eq:space_W_definition}
  \W & := L^2(I;H^2(\Omega)^2 \cap \V) \cap H^1(I;\Hspace) && \hspace{-3.6cm} \hookrightarrow C(\bar I;\V).
\end{align}
From \cite[Corollary 8]{simon_compact_1986} we obtain the following compact imbedding:
\begin{equation}\label{eq:compact_imbedding}
  \W \cimbed L^2(I;\V) \quad \text{and} \quad \W \cimbed C(\bar I;\Hspace).
\end{equation}

The main difficulty in discussing the Navier-Stokes equations with \eqref{eq:DN} boundary conditions, 
is the fact, that boundary terms arise when testing the equations with its solution.
These do not vanish, contrary to the pure Dirichlet case.
In order to discuss them in the following, we shall recall some important trace results.
For a bounded Lipschitz domain $\Omega$ \revB{and} $1 < p < \infty$ it holds,
see \cite[Thm. 3]{marschall_trace_1987}:
There exists a linear, bounded trace operator
\begin{equation}\label{eq:trace_general}
    \tr \colon W^{s,p}(\Omega) \to W^{s- \frac{1}{p},p}(\partial \Omega),
   \qquad  \text{whenever } \frac{1}{p} < s < 1 + \frac{1}{p}.
\end{equation}
The norm of fractional Sobolev spaces can be estimated according to the following theorem,
compare \cite[Theorem 1]{brezis_gagliardo_2018}:
\begin{theorem}[Gagliardo Nirenberg]\label{thm:gagliardo_nirenberg}
  Let $\Omega$ be a Lipschitz domain and let $\revB{s_2 \ge s_1 \ge 0}$, $1 \le p_1,p_2 \le \infty$ and $\theta \in (0,1)$.
  Define $s = \theta s_1 + (1-\theta)s_2$ and $1/p = \theta/p_1 + (1-\theta)/p_2$. If the 
  condition
  \begin{equation*}
    s_2 \text{ is integer } \ge 1, p_2 = 1 \text{ and } s_2 - s_1 \le 1 - \frac{1}{p_1}
  \end{equation*}
  is violated, then the following inequality holds
  \begin{equation*}
    \|v\|_{W^{s,p}(\Omega)} \le C \|v\|_{W^{s_1,p_1}(\Omega)}^\theta \|v\|_{W^{s_2,p_2}(\Omega)}^{1-\theta}
    \quad \textforall v \in W^{s_1,p_1}(\Omega) \cap W^{s_2,p_2}(\Omega).
  \end{equation*}
\end{theorem}
\revA{As the boundary terms arising in our discussion of the problem do not involve any derivatives,
we are interested in precise estimates for $L^p$ norms on the boundary.}
From \cite[Thm. 1.6.6]{Brenner2008} there holds for $1\le p \le \infty$
   \begin{equation}\label{eq:trace_interpolation_1}
     \|v\|_{L^p(\partial \Omega)} \le C 
     \|v\|^{\frac{1}{p}}_{W^{1,p}(\Omega)} 
     \|v\|^{1- \frac{1}{p}}_{L^p(\Omega)} 
      \quad \textforall v \in W^{1,p}(\Omega).
   \end{equation}
The result there is stated without proof, a proof can be found, e.g., in \cite{grisvard_elliptic_2011}.
Intuitively, the scaling in this theorem means, that if $v$ possesses increasing regularity in $L^p(\Omega)$,
less additional regularity of its first derivatives is needed, in order to properly define a trace.
If $v$ does not possess any $W^{1,p}(\Omega)$ regularity for $p>2$, but is only in $H^1(\Om)$, this tradeoff
  can be switched in the sense that a higher contribution of the $H^1(\Om)$ norm allows for a bound of 
the trace in $L^p$ for larger values of $p$. There holds the following result, in which we want to emphasize
the switch in the exponents compared to \eqref{eq:trace_interpolation_1}.
\begin{theorem}[Trace $L^p$ estimate]\label{thm:trace_interpolation_2}~\\
   Let $\Omega$ be a Lipschitz domain in $\mathbb R^2$,
   and let $2 \le p < \infty$.
   Then there is a constant $C\revD{=C(p)}$, such that
   \begin{equation*}
      \|v\|_{L^p(\partial \Omega)} \le C \|v\|^{1-\frac{1}{p}}_{H^1(\Omega)} \|v\|^{\frac{1}{p}}_{L^2(\Omega)}
      \quad \textforall v \in H^1(\Omega).
   \end{equation*}
\end{theorem}
\begin{proof}
   Note, that the case $p=2$ is a special case of \eqref{eq:trace_interpolation_1}, which is why
   we only focus on the case $p>2$.
   Since $p>2$, it holds $1/p < 1/2$, which allows us to define $s:= 1- 1/p$ and we have $1/2 < s < 1$.
   By the Sobolev imbedding theorem, since the boundary $\partial \Omega$
   of the domain is one-dimensional, there holds the embedding  
   $H^{s-\frac{1}{2}}(\partial \Omega) \hookrightarrow L^q(\partial \Omega)$, for 
   $q \le \frac{1 \cdot 2}{1 - (s-1/2)\cdot 2} = \frac{2}{2-2s}=\frac{2}{2-2+2/p} = p$.
   Thus, together with \eqref{eq:trace_general} and \Cref{thm:gagliardo_nirenberg} there holds 
   \begin{equation*}
     \|v\|_{L^p(\dO)} \le C \|v\|_{H^{s-\frac{1}{2}}(\dO)}
     \le C \|v\|_{H^s(\Omega)}
     \le C \|v\|_{H^1(\Omega)}^s \|v\|_{L^2(\Omega)}^{1-s}
     = C \|v\|_{H^1(\Omega)}^{1-\frac{1}{p}} \|v\|_{L^2(\Omega)}^{\frac{1}{p}}.
   \end{equation*}
\end{proof}
Let us lastly summarize some identities and estimates concerning the nonlinear term in the 
Navier-Stokes equations.
It is straightforward to verify with integration by parts, that for $\uu,\vv,\ww \in \V$, it holds
\begin{equation*}
   \Oprod{(\uu \cdot \nabla)\vv,\ww} 
   = - \Oprod{(\uu \cdot \nabla)\ww,\vv} + \Gprod{\uu\cdot \nn,\ww\cdot \vv}.
\end{equation*}
Choosing $\ww=\vv$ and using the homogeneous Dirichlet data on $\Gamma_D$, there thus holds
\begin{equation}\label{eq:trilinear_boundary_term}
   \Oprod{(\uu \cdot \nabla)\vv,\vv} 
   = \frac12 \left(\uu\cdot \nn, |\vv|^2\right)_{\partial \Om}
   = \frac12 \Gprodj{\uu \cdot \nn,|\vv|^2}{N}.
\end{equation}
From \cite[Theorem 3.1]{adams_cone_1977}, we obtain that for any $2 \le p < \infty$, there exists $C=\revD{C(p)}> 0$,
such that 
\begin{equation}\label{eq:Lp_interpolation_estimate}
   \|\varphi\|_{L^p(\Om)} \le C \|\varphi\|_{L^2(\Om)}^{\frac2p} \|\nabla \varphi\|_{L^2(\Om)}^{1-\frac2p}
\end{equation}
for all $\varphi \in \{\psi \in H^1(\Om) : \psi|_{\Gamma_D} = 0\}$.
We shall mostly use this estimate for $p=4$, as together with Hölder's inequality, it allows us to estimate
\begin{equation}\label{eq:trilinear_l4_est}
   \Oprod{(\uu \cdot \nabla )\vv,\ww}
   \le \|\uu\|_{L^4(\Om)} \|\nabla \vv\|_{L^2(\Om)} \|\ww\|_{L^4(\Om)}
   \le C \|\uu\|_{L^2(\Om)}^\frac12  \|\nabla \uu\|_{L^2(\Om)}^\frac12
      \|\nabla \vv\|_{L^2(\Om)} 
      \|\ww\|_{L^2(\Om)}^\frac12 \|\nabla \ww\|_{L^2(\Om)}^\frac12.
\end{equation}
Moreover, for $p=4$ this estimate yields the imbeddings
\begin{equation}\label{eq:L4L4_imbeddings}
   \VV \hookrightarrow L^4(I\times \Om)^2 \qquad \text{and} \qquad
   \W \hookrightarrow W^{1,4}(I;L^4(\Om)^2).
\end{equation}

\section{Discussion of State Equation}\label{sect:state}
With the auxiliary results of the previous section established, we can now turn towards the discussion of 
the state equation.
In its general form for initial data $\uu_0 \in \V$, right hand side $\ff \in L^2(I;L^2(\Omega)^2)$ and 
boundary data $q_i \in L^2(I)$, $i=1,...,L$, it has the form
\begin{equation}\label{eq:NS_fullData}
  \left\{
  \begin{aligned}
    \revB{\Opair{\partial_t \uu,\vv}} + \Oprod{\nabla \uu,\nabla \vv} + \Oprod{(\uu\cdot \nabla)\uu,\vv}
    &= \Oprod{\ff,\vv} - \sum_{i=1}^L q_i(t)\Gprodj{\vv, \nn}{N,i}
    \quad \textforall \vv \in \V,  \alev \ t \in I, \\
    \uu(0) &= \uu_0.
  \end{aligned}
  \right.
\end{equation}
To analyze this problem, let us recall some results on related and auxiliary problems.
We first discuss the stationary Stokes equations with mixed homogeneous boundary data, i.e.,
for $\ff \in L^2(\Omega)^2$ we search $(\ww,p_\ww) \in H^1(\Omega)^2 \times L^2(\Omega)$ such that
\begin{equation}\label{eq:Stokes_stationary}
   \left\lbrace
   \begin{aligned}
      - \Delta \ww + \nabla p_\ww &= \ff &&\text{in } \Omega,\\
      \nabla \cdot \ww &= 0&&\text{in } \Omega,\\
         \ww &= \oo&& \text{on } \Gamma_D,\\
         \partial_\nn \ww - p_\ww \nn &= \oo && \text{on } \Gamma_N.
   \end{aligned}
   \right.
\end{equation}
Its weak formulation in terms of divergence free spaces reads: Find $\ww \in \V$ such that
\begin{equation*}
  \Oprod{\nabla \ww,\nabla \vv} = \Oprod{\ff,\vv} \quad \textforall \vv \in \V.
\end{equation*}
Solvability of this problem in terms of $\ww$ can be argued by the Lax-Milgram theorem, while the existence 
of a unique associated pressure $p_\ww \in L^2(\Omega)$ can be argued as in
\cite[Lemma 5.1, Theorem 5.2]{mazya_l_2007}.
\revD{The argument presented there for the 3D case, directly translates to 2D.
   Note that due to the presence of Do-Nothing boundary conditions, the pressure is unique in the whole
   $L^2(\Om)$ and not only up to constants, as it would be the case for pure Dirichlet conditions.
}
Moreover, due to \Cref{ass:domain}, there holds the following $H^2$ regularity result, 
see \cite[Theorem A.1]{benes_solutions_2016}:
\begin{theorem}\label{thm:mixed_Stokes_h2}
  Let $(\ww,p_\ww) \in \V \times L^2(\Omega)$ be weak solutions to \eqref{eq:Stokes_stationary}
  for $\ff \in L^2(\Omega)^2$. Then $(\ww,p_\ww) \in [\V \cap H^2(\Omega)^2] \times H^1(\Omega)$ and there holds the 
   bound 
   \begin{equation*}
      \|\ww\|_{H^2(\Omega)} + \|p_\ww\|_{H^1(\Omega)} \le C \|\ff\|_{L^2(\Omega)}.
   \end{equation*}
\end{theorem}
A similar result also holds for the instationary Stokes equations with mixed homogeneous boundary data
\begin{equation}\label{eq:Stokes_instationary}
  \left\{
  \begin{aligned}
    \revB{\Opair{\partial_t \uu,\vv}} + \Oprod{\nabla \uu,\nabla \vv} &= \Opair{\ff,\vv} 
    \quad \textforall \vv \in \V,  \alev t \in I,\\
    \uu(0) &= \uu_0.
  \end{aligned}
  \right.
\end{equation}
For this, we have the following, see \cite[Theorem 3.4]{benes_solutions_2016},
also \cite[Proposition 72.2]{ern_finite_2021}.
\begin{proposition}\label{prop:instat_stokes_solvability}
  \revD{Let } $\ff \in L^2(I;\V^*)$ and $\uu_0 \in \Hspace$. \revD{Then }
  \eqref{eq:Stokes_instationary} possesses a unique solution $\uu \in \VV$ with a bound
   \begin{equation*}\label{eq:LtwoV_Stokes_Bound}
      \|\partial_t \uu\|_{L^2(I;\V^*)} + \|\uu\|_{L^2(I;\V)} 
      \le C \revD{\left(\|\uu_0\|_{L^2(\Om)} + \|\ff\|_{L^2(I;\V^*)} \right)}.
  \end{equation*}
  Further, if $\ff \in L^2(I;L^2(\Omega)^2)$ and $\uu_0 \in \V$, then
  $\uu \in \revB{\W}$ \revB{and there exists a unique pressure $p_\uu \in L^2(I;H^1(\Omega))$, such that
  \begin{equation*}\label{eq:pressure_existence}
    \partial_t \uu - \Delta \uu + \nabla \revB{p_\uu} = \ff
  \end{equation*}
  }
  holds in the $L^2(I;L^2(\Omega)^2)$ sense, and there holds a bound
   \begin{equation*}
     \|\uu\|_{L^2(I;H^2(\Omega))} +\|\uu\|_{H^1(I;L^2(\Omega))} + \|\revB{p_\uu}\|_{L^2(I;H^1(\Omega))} 
      \le C (\|\uu_0\|_\V + \|\ff\|_{L^2(I;L^2(\Omega))}).
   \end{equation*}
\end{proposition}
Here, the statements for $\revB{p_\uu}$ were obtained by bootstrapping and applying the ones
of \Cref{thm:mixed_Stokes_h2} for almost every $t \in I$.
With this, we can consider the instationary Navier-Stokes equations, 
first with homogeneous \eqref{eq:DN} boundary conditions.
We consider the problem: Find $\uu \in \VV$ such that
\begin{equation}\label{eq:NS_DN}
   \left\lbrace
   \begin{aligned}
      \revB{\Opair{\partial_t \uu, \vv}} + \Oprod{\nabla \uu,\nabla \vv} +
   \Oprod{(\uu\cdot \nabla)\uu,\vv} &= \Oprod{\ff,\vv} \quad \textforall \vv \in \V,  \alev t \in I,\\
   \uu(0) &= \uu_0.
   \end{aligned}\right.
\end{equation}
Due to the choice of boundary conditions, existence of weak solutions to the 
Navier-Stokes problem is not a priori given. 
However, if a solution exists, \revD{we are able to show uniqueness and, if permitted by the data, improved regularity. Let us first discuss regularity and defer the discussion of uniqueness to \Cref{thm:uniqueness_stateeq}. If} the data are sufficiently regular, the solution has the 
same regularity as the one for the corresponding Stokes problem.
\begin{theorem}\label{thm:NS_H2}
  Let $\uu \in \revB{\VV}$ be a weak solution to \eqref{eq:NS_DN} for $\ff \in L^2(I;L^2(\Omega)^2)$ and 
   $\uu_0 \in \V$. Then $\uu \in \W$ and there holds a bound
      \begin{equation*}
         \|\uu\|_{L^2(I;H^2(\Omega))} + \|\uu\|_{H^1(I;L^2(\Omega))}
         \le C (\|\uu\|_{L^2(I;\V)} + \|\uu\|_{L^\infty(I;\Hspace)})\cdot (\|\ff\|_{L^2(I;L^2(\Omega))} + \|\uu_0\|_\V).
      \end{equation*}
      \revD{Moreover, there exists a unique pressure $p_\uu \in L^2(I;H^1(\Om))$ such that 
        \begin{equation*}
          \partial_t \uu - \Delta \uu + (\uu\cdot \nabla)\uu + \nabla p_\uu = \ff
        \end{equation*}
      holds in the $L^2(I;L^2(\Om)^2)$ sense, together with a bound}
        \begin{equation*}
          \revD{
          \|p_\uu\|_{L^2(I;H^1(\Om))} \le C \left( \|\ff\|_{L^2(I;L^2(\Om))} 
          + \|\uu\|_{L^2(I;H^2(\Om))} + \|\uu\|_{H^1(I;L^2(\Om))} 
          + \|\uu\|_{L^2(I;H^2(\Om))}\|\uu\|_{H^1(I;L^2(\Om))} \right).
        }
        \end{equation*}
\end{theorem}
\begin{proof}
   The result for $\uu$ can be shown following precisely the same steps as 
   \cite[Ch.3, Theorem 3.10]{Temam1977}, using the $H^2$-regularity result of 
   \Cref{thm:mixed_Stokes_h2}.
   \revD{The regularity of $p_\uu$ can be obtained by a bootstrapping argument similar to the one 
     in \Cref{prop:instat_stokes_solvability}, where the term $\|(\uu \cdot \nabla)\uu\|_{L^2(I;L^2(\Om))}$ can be estimated using \cite[Theorem 3.8]{vexler_error_2024}.
   }
\end{proof}
Note that the possible non-existence of solutions is due to the fact, that we do not have a bound
of $\|\uu\|_{L^2(I;\V)}$ and $\|\uu\|_{L^\infty(I;\Hspace)}$ purely depending on the data.
As these terms arise on the right hand side of the estimate shown in \Cref{thm:NS_H2},
in order to apply this result \revD{to obtain boundedness of a minimizing sequence,}
we need not only assume 
the existence of solutions, but also provide a bound for the norms in question.

While existence of weak solutions is in general only shown for small times and/or data, 
\cite[Theorem 4.1]{benes_solutions_2016} show that if a solution exists for some data,
then for small perturbations of the data, the equations remain solvable.
Note that, while the authors of \cite{benes_solutions_2016} assume the existence 
of a solution in $\W$,
\Cref{thm:NS_H2} shows that it suffices to 
assume existence of a weak solution in $\VV$,
as the $H^2$ regularity can then be deduced from the data.
\begin{proposition}\label{prop:open_dataset_H2}
  Let $\ff \in L^2(I;L^2(\Omega)^2)$ and $\uu_0 \in \V$ such that a weak solution 
  $\uu \in \VV$ to \eqref{eq:NS_DN} 
  exists. Then there exist $\varepsilon, \revD{\delta} > 0$, such that for all 
  $\tilde \ff \in L^2(I;L^2(\Omega)^2)$ and $\tilde \uu_0 \in \V$, satisfying
  \begin{equation*}
    \| \tilde \ff - \ff \|_{L^2(I;L^2(\Omega))} + \|\tilde \uu_0 - \uu_0\|_\V \le \varepsilon,
  \end{equation*}
  there exists a solution $\tilde \uu \in \W$ 
  to \eqref{eq:NS_DN} with data $\tilde \ff$ and $\tilde \uu_0$, \revD{satisfying}
  \begin{equation*}
    \revD{
    \|\tilde \uu - \uu\|_{L^2(I;H^2(\Om))} + 
    \|\tilde \uu - \uu\|_{H^1(I;L^2(\Om))} \le \delta.
  }
  \end{equation*}
\end{proposition}
\begin{proof}
  The result for $\ff,\tilde \ff \in L^2(I;\Hspace)$ is presented in \cite[Theorem 4.1]{benes_solutions_2016}.
  With the characterization of $L^2(\Omega)^2 = \Hspace \oplus \Hspace^\perp$ presented in 
  \eqref{eq:Hperp},
  and the orthogonal projection $\mathbb P\colon L^2(\Omega)^2 \to \Hspace$, we further know that
  a right hand side $\mathbb P \ff$ yields the same solution $\uu$ as $\ff$, which concludes the proof.
\end{proof}
This shows, that the set of data, for which \revD{a solution in $\W$} exists,
is an open subset of $L^2(I;L^2(\Omega)^2) \times \V$.
Note however, that again the above result does not contain a bound for $\|\tilde \uu\|_{L^2(I;\V)}$ and
$\|\tilde \uu\|_{H^1(I;\Hspace)}$ in terms of norms of the data.

Up to now, we have considered only problems with homogeneous boundary data.
To conclude this section, we show that all results derived previously, can be also extended to the 
type of inhomogeneous boundary data under consideration in this work.
We achieve this, by transforming \eqref{eq:NS_fullData} into a formulation where 
no boundary data terms arise.
Here we will \revC{use the fact,} that the controls of \eqref{eq:minProb_1} are purely time dependent.
\begin{lemma}\label{lemm:boundary_to_rhs}
  Let $\uu_0 \in \V$, $\ff \in L^2(I;L^2(\Omega)^2)$ and $q_i \in L^2(I)$, $i=1,...,L$. Then 
  there exist functions \revA{$\zeta_i \in C^\infty(\Omega)$},
  $i=1,...,L$, such that $\uu \in \W$ solves
  \eqref{eq:NS_fullData} if and only if $\uu$ solves the following
\begin{equation}\label{eq:NS_onlyRhs}
  \left\{
  \begin{aligned}
    \Oprod{\partial_t \uu,\vv} + \Oprod{\nabla \uu,\nabla \vv} + \Oprod{(\uu\cdot \nabla)\uu,\vv}
    &= \Oprod{\ff - \sum_{i=1}^L q_i \nabla \zeta_i,\vv}
    \quad \textforall \vv \in \V,  \alev \ t \in I, \\
    \uu(0) &= \uu_0.
  \end{aligned}
  \right.
\end{equation}
\end{lemma}
\begin{proof}
  Due to \Cref{ass:domain}, for every $i=1,...,L$, there exist two nested neighborhoods, 
  $\mathcal U_i, \widetilde{\mathcal U}_i \revB{\subset \R^2}$ such that
  $\Gamma_{N,i} \subset \widetilde{\mathcal U}_i \subset \overline{\widetilde{\mathcal U}}_i
  \subset \mathcal U_i$
  and
  $\Gamma_N \cap \mathcal U_i = \Gamma_{N,i}$. Define smooth functions $\chi_i\colon\R^2 \to \revB{[0,1]}$ such that 
  $\chi_i(\xx) = 0$ for all $\xx \in \R^2 \backslash \mathcal U_i$ and 
  $\chi_i(\xx) = 1$ for all $\xx \in \widetilde{\mathcal U}_i$.
  Note that this implies $\nabla \chi_i = 0$ in $\widetilde{\mathcal U}_i$.
  Denote the constant normal vector of $\Gamma_{N,i}$ by $\nn_i$ and let $\beta_i := \nn_i \cdot \xx_{0,i}$ for 
  some arbitrary $\xx_{0,i} \in \Gamma_{N,i}$.
  \revB{Then $\nn_i \cdot \xx = \beta_i$ for all $\xx \in \GNi$.}
  Consider the function 
  \begin{equation*}
    {\zeta}_i\colon \R^2 \to \R, \
    {\zeta}_i(\xx) := \chi_i(\xx)\cdot ((\nn_i\cdot \xx)-\beta_i + 1).
  \end{equation*}
  For all $j=1,...,L$ and $\xx \in \Gamma_{N,j}$, it holds
  $\revD{\zeta_i(\xx)|_{\Gamma_{N,j}}} = \delta_{i,j}$ and
  $\revD{\nabla \zeta_i(\xx)|_{\Gamma_{N,j}}} = \delta_{i,j} \nn_i$,
  where by $\delta_{i,j}$ we denote the Kronecker delta.
  With this, using $\vv|_{\Gamma_D} = \oo$ for all $\vv \in \V$, we can write
  \begin{equation*}
    \Gprodj{\vv,\nn}{N,i} 
    \revB{= (\vv,\zeta_i \nn)_{\Gamma_N}}
    = (\vv,\zeta_i \nn)_{\partial \Omega} 
    = \Oprod{\nabla \cdot \vv,\zeta_i} + \Oprod{\vv, \nabla \zeta_i}
    = \Oprod{\vv,\nabla \zeta_i} \quad \textforall \vv \in \V,
  \end{equation*}
  which concludes the proof.
\end{proof}
\begin{remark}
      Note that, despite the right hand side terms $q_i \nabla \zeta_i$ being gradients of smooth functions,
      as $\zeta_i|_{\GNi} = 1$, it holds $q_i \nabla \zeta_i \not \in \Hspace^\perp$, due to 
      the characterization of the latter space derived in \eqref{eq:Hperp}.
   \end{remark}
\begin{remark}
   While \Cref{lemm:boundary_to_rhs} shows, that the two formulations 
   \eqref{eq:NS_fullData} and \eqref{eq:NS_onlyRhs} yield the same velocity solution $\uu$, they lead to 
   different associated pressures.
Testing \eqref{eq:NS_onlyRhs} especially with any $\vv \in \mathring{\V} \subset \V$,
\revD{where $\mathring{\V}$ is defined in \eqref{eq:def_Vcirc},}
yields for a.e. $t \in I$
   \begin{equation*}
      \Oprod{\partial_t \uu - \Delta \uu + (\uu\cdot \nabla \uu),\vv} 
      = \Oprod{\ff - \sum_{i=1}^L q_i \nabla \zeta_i ,\vv}.
   \end{equation*}
   \revB{As functions from $\mathring{\V}$ vanish on the whole $\partial \Om$, by}
   \cite[Chapter I, Proposition 1.2]{Temam1977}, 
   there exists $\revB{\tilde p_\uu} \in L^2(I;H^1(\Omega))$ such that 
   \begin{equation*}
      \partial_t \uu - \Delta \uu + (\uu\cdot \nabla \uu) + \nabla \revB{\tilde p_\uu} 
      = \ff - \sum_{i=1}^L q_i \nabla \zeta_i,
   \end{equation*}
   which is understood as an equality in $L^2(I;L^2(\Omega)^2)$.
   Subtracting \eqref{eq:NS_onlyRhs} from this equation tested with any $\vv \in \V$ gives for a.e. $t \in I$
   \begin{equation*}
      0 = \Oprod{- \Delta \uu ,\vv} + \Oprod{\nabla \revB{\tilde p_\uu},\vv} - \Oprod{\nabla \uu,\nabla \vv}
      = \Gprod{\revB{\tilde p_\uu} \nn - \partial_\nn \uu,\vv} - \Oprod{\revB{\tilde p_\uu},\nabla \cdot \vv}
      = \Gprodj{\revB{\tilde p_\uu} \nn - \partial_\nn \uu,\vv}{N}.
   \end{equation*}
   Following the same approach, we obtain from \eqref{eq:NS_fullData}, that 
   there exists $\revB{p_\uu} \in L^2(I;H^1(\Omega))$ such that 
   \begin{equation*}
      \partial_t \uu - \Delta \uu + (\uu\cdot \nabla \uu) + \nabla \revB{p_\uu} = \ff.
   \end{equation*}
   Subtracting \eqref{eq:NS_fullData} from this equation tested with $\vv \in \V$ yields for a.e. $t \in I$
   \begin{equation*}
      \Gprodj{\sum_{i=1}^L q_i \nn,\vv}{N}
      = \Oprod{- \Delta \uu ,\vv} + \Oprod{\nabla \revB{p_\uu},\vv} - \Oprod{\nabla \uu,\nabla \vv}
      = \Gprod{ \revB{p_\uu} \nn - \partial_\nn \uu,\vv} - \Oprod{\revB{p_\uu},\nabla \cdot \vv}
      = \Gprodj{ \revB{p_\uu} \nn - \partial_\nn \uu,\vv}{N}.
   \end{equation*}
   As the solution $\uu$ is the same for both formulations, this shows that only the associated pressures
   differ by $\sum_{i=1}^L q_i \zeta_i$.
   In fact, equation \eqref{eq:NS_onlyRhs} is the weak formulation to
   \begin{equation*}
      \left\{
         \begin{aligned}
            \partial_t \uu - \Delta \uu + (\uu\cdot \nabla) \uu + \nabla \revB{\tilde p_\uu} 
            &= \ff && \text{in } I \times \Om,\\
            \nabla \cdot \uu &= 0 && \text{in } I \times \Om,\\
            \uu(0) &= \uu_0 && \text{in } \Om,\\
            \uu &= \oo && \text{on } I \times \Gamma_D,\\
            \revB{\tilde p_\uu} \nn - \partial_\nn \uu &= \oo && \text{on } I \times \Gamma_N,
         \end{aligned}
         \right.
   \end{equation*}
   where we have defined the modified pressure $\tilde p_\uu = p_\uu - \sum_{i=1}^L q_i \zeta_i$.
   \revD{As $\tilde p_\uu$ is the associated pressure to \eqref{eq:NS_onlyRhs}, which is formulated
   with homogeneous boundary conditions, the result of \Cref{thm:NS_H2} is applicable, and it holds
   \begin{equation*}
      \|\tilde p_\uu\|_{L^2(I;H^1(\Om))} \le 
      C \left(
        \|\uu\|_{L^2(I;H^2(\Om))} + \|\uu\|_{H^1(I;L^2(\Om))} 
        + \|\uu\|_{L^2(I;H^2(\Om))}\|\uu\|_{H^1(I;L^2(\Om))}
      + \|\ff\|_{L^2(I;L^2(\Om)^2)} + \|q\|_{L^2(I)}\right).
   \end{equation*}
   Due to the relationship between $p_\uu$ and $\tilde p_\uu$, a similar bound also holds for $p_\uu$.
   Note that this estimate cannot be stated purely in terms of data, as the norms of $\uu$ cannot a priori be 
   bounded in terms of the data $q,\ff$ and $\uu_0$.}
\end{remark}
Concluding the considerations on existence \revD{and regularity} of solutions, let us now turn toward the question of uniqueness.
We shall see that this question is much easier to discuss, as there holds the following.
\begin{theorem}\label{thm:uniqueness_stateeq}
   Let $\uu \in \VV$ solve \eqref{eq:NS_fullData} \revB{for given $q \in L^2(I)^L, \uu_0 \in \Hspace$.}
   Then \revB{$\uu$} is the unique solution of \eqref{eq:NS_fullData} in the space $\VV$.
\end{theorem}
\begin{proof}
   Assume there are $\uu_1 = \uu$ and $\uu_2$ two solutions in $\VV$
   to \eqref{eq:NS_fullData}.
   For almost every $t \in I$,
   we then test both equations with $\ee=\uu_1-\uu_2$ and arrive by \revB{standard} calculations at
   \begin{equation*}
      \frac{1}{2} \partial_t \|\ee\|_{L^2(\Omega)}^2 +  \|\nabla \ee\|_{L^2(\Omega)}^2
      + \Oprod{(\uu_1\cdot \nabla)\uu_1 - (\uu_2 \cdot \nabla)\uu_2,\ee} = 0.
   \end{equation*}
   \revB{By inserting an artificial zero and using \eqref{eq:trilinear_boundary_term},
      the trilinear term can be rewritten as
      \begin{align*}
     \Oprod{(\uu_1\cdot \nabla)\uu_1 - (\uu_2 \cdot \nabla)\uu_2,\ee}
     &= \Oprod{(\uu_1\cdot \nabla)\uu_1 - (\uu_1\cdot \nabla )\uu_2 
     + (\uu_1\cdot \nabla)\uu_2 - (\uu_2 \cdot \nabla)\uu_2,\ee}\\
     &= \Oprod{(\uu_1\cdot \nabla) \ee, \ee} + \Oprod{(\ee \cdot \nabla)\uu_2,\ee}
     = \frac12 (\uu_1 \cdot \nn,|\ee|^2)_{\partial \Om} + \Oprod{(\ee \cdot \nabla)\uu_2,\ee}.
      \end{align*}
      Using \eqref{eq:trilinear_l4_est} to estimate the domain integral term,
      and the trace $L^p$ estimate from \Cref{thm:trace_interpolation_2}, together with
      \revD{the extended Poincaré-Steklov inequality, see \cite[Lemma 3.30]{ern_finite_2021_partI},}
      we obtain
   }
   \begin{align*}
     \Oprod{(\uu_1\cdot \nabla)\uu_1 - (\uu_2 \cdot \nabla)\uu_2,\ee}
      &\le C\left(\|\ee\|_{L^2(\Omega)}\|\nabla \ee\|_{L^2(\Omega)} \|\nabla \uu_2\|_{L^2(\Omega)}
      + \|\uu_1\|_{L^3(\partial \Omega)} \|\ee\|_{L^3(\partial \Omega)}^2\right)\\
      &\le C\left(\|\ee\|_{L^2(\Omega)}\|\nabla \ee\|_{L^2(\Omega)} \|\nabla \uu_2\|_{L^2(\Omega)}
      + \|\uu_1\|_{H^1(\Omega)}^{\frac{2}{3}}\|\uu_1\|_{L^2(\Omega)}^{\frac{1}{3}} 
      \|\nabla \ee\|_{L^2(\Omega)}^{\frac{4}{3}}\|\ee\|_{L^2(\Omega)}^{\frac{2}{3}}\right)\\
      & \le  \frac{1}{2} \|\nabla \ee\|_{L^2(\Omega)}^2 
      + C\left(\|\ee\|_{L^2(\Omega)}^2\|\nabla \uu_2\|_{L^2(\Omega)}^2 
      + \|\uu_1\|_{H^1(\Omega)}^{2}\|\uu_1\|_{L^2(\Omega)}\|\ee\|_{L^2(\Omega)}^{2}\right).
   \end{align*}
   In the last step, we have used Young's inequality twice,
   with $\frac{1}{2} + \frac{1}{2} = 1$ and $\frac{2}{3} + \frac{1}{3} = 1$.
   Absorbing the term $\revB{\frac{1}{2}} \|\nabla \ee\|_{L^2(\Omega)}^2$ to the left,
   yields that for almost every $t \in I$, there holds
   \begin{equation*}
      \partial_t \|\ee\|_{L^2(\Omega)}^2 + \|\nabla \ee\|_{L^2(\Omega)}^2
      \le C\left(\|\nabla \uu_2\|_{L^2(\Omega)}^2 + \|\uu_1\|_{H^1(\Omega)}^{2}\|\uu_1\|_{L^2(\Omega)}\right)
      \|\ee\|_{L^2(\Omega)}^{2}.
   \end{equation*}
   \revB{As $\uu_1,\uu_2 \in \VV$, the term 
      $\|\nabla \uu_2\|_{L^2(\Omega)}^2 + \|\uu_1\|_{H^1(\Omega)}^{2}\|\uu_1\|_{L^2(\Omega)}$
   is integrable in time.}
   An application of Gronwall's Lemma, \revD{together with the fact, that $\ee(0)=\oo$,}
   then yields $\ee \equiv \revB{\oo}$.
\end{proof}
As already discussed, the possible non existence of solutions to the Navier-Stokes equations 
is due to the fact, that the \revA{nonlinear} term 
$\Oprod{(\uu \cdot \nabla) \uu,\uu}$ may be unbounded for the unknown quantity $\uu$.
\revA{In \Cref{prop:instat_stokes_solvability}, we observed that this difficulty is not present in the 
  linear Stokes equations.
  Moreover, in \Cref{thm:NS_H2} we have seen that for the Navier-Stokes equations we can obtain improved
  regularity, assuming that we already have some information on the solution.
  These considerations motivate the definition of the following linearized problem,
  where we assume some given regularity of the linearization point.}
For given $\uu \in L^4(I;L^4(\Omega)^2)$, $\ww_0 \in \Hspace$ and $q_i \in L^2(I)$, $i=1,...,L$,
we search for
$\ww \in \VV$ solving
\begin{equation}\label{eq:linearized_NS}
  \left\{
  \begin{aligned}
     \revB{\Opair{\partial_t \ww,\vv}} +  \Oprod{\nabla \ww,\nabla \vv} + \Oprod{(\uu\cdot \nabla \ww),\vv}
  &= - \sum_{i=1}^L q_i \Gprodj{\vv,\nn}{N,i} \quad \textforall \vv \in \V,  \alev t \in I,\\
    \ww(0) &= \ww_0.
\end{aligned}\right.
\end{equation}
\begin{theorem}\label{thm:linearized_NS}
  Let $\uu \in L^4(I;L^4(\Omega)^2)$, $\ww_0 \in \Hspace$ and $q_i \in L^2(I)$, $i=1,...,L$.
  Then problem \eqref{eq:linearized_NS}
  possesses a unique solution $\ww \in \VV$
  with an estimate
  \begin{equation*}
    \revB{\|\partial_t \ww\|_{L^2(I;\V^*)}+}
    \|\ww\|_{L^2(I;\V)} + \|\ww\|_{L^\infty(I;\Hspace)} \le C(\|\uu\|_{\revB{L^4(I\times \Om)}})
    \left[\|\ww_0\|_{L^2(\Omega)} + \|q\|_{L^2(I)}\right],
  \end{equation*}
  where the constant $C$ depends nonlinearly on $\revD{\|\uu\|_{L^4(I\times \Om)}}$.
\end{theorem}
\begin{proof}
   The proof uses standard arguments based on the Gronwall Lemma.
   For the solution $\ww$, one has to bound the term
   \begin{align*}
      ((\uu \cdot \nabla)\ww,\ww)_\Omega 
      & \le \|\uu\|_{L^4(\Omega)}\|\nabla \ww\|_{L^2(\Omega)} \|\ww\|_{L^4(\Omega)}
      \le C \|\uu\|_{L^4(\Omega)}\|\nabla \ww\|_{L^2(\Omega)}^{\frac{3}{2}} 
         \|\ww\|_{L^2(\Omega)}^{\frac{1}{2}}\\
      & \le \frac{1}{2}\|\nabla \ww\|_{L^2(\Omega)}^2 + C \|\uu\|_{L^4(\Omega)}^4\|\ww\|_{L^2(\Omega)}^2.
   \end{align*}
   \revB{where we have used \eqref{eq:Lp_interpolation_estimate} and Young's inequality.
      By assumption $\|\uu\|_{L^4(\Om)}^4$ is integrable in time. 
   The bound then follows by a standard Gronwall argument. Existence is obtained by a Galerkin procedure.
   Uniqueness is obtained from the stability estimate by replacing $\ww$ with $\ee=\ww_1-\ww_2$,
   the difference between two solutions. }
\end{proof}

\section{Optimal Control Problems}\label{sect:opt_cont}
With the considerations from the previous sections, we can now formulate our optimal control problem.
To overcome the problem of \revD{non-existence} of solutions to the Navier-Stokes equations subject to the 
\eqref{eq:DN} boundary conditions for general data, we restrict the admissible set of controls.
We \revA{will} search only in the set of controls that lead to a weak solution of the Navier-Stokes equations,
which by \Cref{prop:open_dataset_H2} is an open set 
of the space of controls. \revD{In the following, we denote the control space by $Q:= L^2(I)^L$.}
While this remedies the \revB{problem of solvability of the state equation, there remains 
   the lack of an energy estimate of the state in terms of the controls.
   For the optimal control problem, this means that for a minimizing sequence $(q_n,\uu_n)$ 
   with $q_n$ bounded in $Q$, the sequence of states $\uu_n$ could be unbounded in $\VV$.
   However, as we are able to obtain boundedness of $\uu_n$ in the norm of the tracking type term,
   and we have seen in \Cref{thm:linearized_NS}, that a linearized equation is stable if the 
   linearization point is in $L^4(I\times \Om)^2$, this lead us to the choice of a 
   non-standard tracking functional in our objective function, using the $L^4(I \times \Omega)^2$ norm,
   instead of the $L^2(I \times \Omega)^2$ norm.}
%
All in all, we consider the problem
\begin{align}
   \min_{q,\uu}\quad & J(\revB{q,\uu}) := \frac{1}{4} \|\uu-\uu_d\|_{L^4(I\times \Omega)}^4 + 
   \frac{\alpha}{2} \|q\|_{L^2(I)}^2 \label{eq:minProb}\tag{$P$}\\
   \text{s.t.} \quad & q \in \revA{Q} , \uu \in \VV,\notag\\
        &
        \left\lbrace
   \begin{aligned}
      \revB{\Opair{\partial_t \uu,\vv}} + \Oprod{\nabla \uu,\nabla \vv} + \Oprod{(\uu\cdot \nabla)\uu,\vv} &=
     - \sum_{i=1}^L q_i \Gprodj{\nn,\vv}{N,i} \quad \textforall \vv \in \V,  \alev t \in I,\\
      \uu(0) &= \uu_0,\\
   \end{aligned}
 \right. \notag\\
        & \textcolor{black}{q_{a,i} \le q_i \le q_{b,i} \qquad \text{a.e. in } I, \qquad i=1,...,L.} \notag
\end{align}
where $\uu_d \in L^4(I \times \Omega)^2$ is some desired data, 
$q_a,q_b \in [\R \cup \{ \pm \infty\}]^L$ \revB{represent the control constraints} 
with $q_a < q_b$ componentwise, $\alpha > 0$ is a regularization parameter,
\revD{and $\uu_0 \in \Hspace$ are fixed initial data.}
Using the construction presented in \Cref{lemm:boundary_to_rhs}, 
we obtain an equivalent optimization problem
\begin{align}
   \min_{q,\uu}\quad & J(\revB{q,\uu}) := \frac{1}{4} \|\uu-\uu_d\|_{L^4(I\times \Omega)}^4 + 
   \frac{\alpha}{2} \|q\|_{L^2(I)}^2 \label{eq:minProb_3}\tag{$\widehat{P}$}\\
   \text{s.t.} \quad & q \in \revA{Q}, \ \uu \in \VV, \label{eq:opt_prob_feasible_space}\\
        &
        \left\lbrace
   \begin{aligned}
      \revB{\Opair{\partial_t \uu,\vv}} + \Oprod{\nabla \uu,\nabla \vv} + \Oprod{(\uu\cdot \nabla)\uu,\vv} &=
     - \Oprod{\sum_{i=1}^L q_i \nabla \zeta_i,\vv} \quad \textforall \vv \in \V,  \alev t \in I,\\
   \uu(0) &= \uu_0,\\
   \end{aligned}
 \right. \label{eq:stateequation}\\
        & q_{a,i} \le q_i \le q_{b,i} \qquad \text{a.e. in } I, \qquad i=1,...,L. \label{eq:control_constraint}
\end{align}
\subsection{Wellposedness}
We first show, that this is a wellposed problem.
To deduce convergence of minimizing sequences, we need to recall the following,
see \cite[Ch. 3, Theorem 2.3]{Temam1977}
\begin{proposition}\label{thm:compact_imbedding}
  Let $X_0 \overset{c}{\hookrightarrow} X \hookrightarrow X_1$ be three Hilbert spaces with $X_0, \ X_1$ being
  reflexive, then the following embedding is compact:
  \begin{equation*}
    \left[L^2(I;X_0) \cap W^{1,1}(I;X_1)\right] \overset{c}{\hookrightarrow} L^2(I;X).
  \end{equation*}
\end{proposition}
In order to ensure, that the feasible set of the above problem is nonempty, we shall work under the
   following assumption
   \begin{assumption}\label{ass:feasible_point}
      There exists $\qc \in Q$ satisfying $q_{a,i}\le \qc \le q_{b,i}$ a.e. in $I$, 
      $i=1,...,L$, for which \eqref{eq:stateequation} possesses a solution $\uuc \in \VV$.
   \end{assumption}
\begin{remark}
   If \revC{it holds $\uu_0 = \oo$, and}
   $q_a, q_b$ admit the control $q = 0$, then it is easy to check, that in the above assumption, we can use 
   $\qc = 0$ and $\uuc = \oo$. Moreover, if there exists a constant $d$, such that $q_{a,i} \le d \le q_{b,i}$
   for all $i=1,...,L$, setting $\qc_i = d$, we observe
   \begin{equation*}
      \sum_{i=1}^L \Oprod{\qc_i \nabla \zeta_i,\vv}
      = d\sum_{i=1}^L  \Oprod{\nabla \zeta_i,\vv}
      = d\sum_{i=1}^L \dOprod{\zeta_i \nn, \vv} -d\sum_{i=1}^L  \Oprod{\zeta_i,\nabla \cdot \vv}
      = d\dOprod{\nn,\vv}
      = d\Oprod{\nabla \cdot \vv \revB{,1}} = 0,
   \end{equation*}
   for all $\vv \in \V$, where we have used that $\nabla \cdot \vv = 0$, $\zeta_i(x) = \delta_{i,j}$ for
   $x \in \Gamma_{N,j}$ and $\vv = 0$ on $\Gamma_D$. Thus also in this case, it is straightforward to check,
   that $\uuc = \oo$ solves \eqref{eq:stateequation}.
   \revC{If in the above settings $\uu_0 \neq \oo$, then a smallness condition on $\|\uu_0\|_{\Hspace}$ 
   still guarantees existence of $\uuc$.}
\end{remark}

\begin{theorem}\label{thm:wellposedness_ocp}
   Let $(\qc,\uuc)$ satisfy \Cref{ass:feasible_point}, \revB{and let $\uu_0 \in \Hspace$}.
   Then problem \eqref{eq:minProb_3} possesses a 
   global solution $(\qbar,\ubar)\in \revB{Q \times \VV}$, which satisfies the bounds
  \begin{equation*}
     \|\qbar\|_{L^2(I)} \le (2\alpha)^{-\frac12} \|\uuc - \uu_d\|_{L^4(I\times \Om)}^2 + \|\qc\|_{L^2(I)},
     \qquad
     \|\ubar\|_{L^4(I\times \Om)} \le \|\uuc - \uu_d\|_{L^4(I\times \Om)} +
     (2 \alpha)^{\frac14} \|\qc\|_{L^2(I)}^{\frac12} + \|\uu_d\|_{L^4(I\times \Om)},
  \end{equation*}
  \revB{as well as
  \begin{equation*}
     \|\ubar\|_{L^2(I;\V)} + \|\ubar\|_{H^1(I;\V^*)} \le C_1(\uu_d,\alpha,\qc,\uuc,\uu_0,\Omega).
  \end{equation*}
  }
  Moreover, \revB{if $\uu_0 \in \V$, then $\ubar \in \W$ and} there exists a bound 
  \begin{equation*}
    \|\ubar\|_{L^2(I;H^2(\Om))} + \|\ubar\|_{H^1(I;L^2(\Om))} 
    \le \revB{C_2(\uu_d,\alpha,\qc,\uuc,\uu_0,\Omega).}
  \end{equation*}
\end{theorem}
\begin{proof}
  \revD{Let us first introduce the set of feasible control-state pairs
    \begin{equation*}
      M := \{(q,\uu) \in Q \times \VV: (q,\uu) \text{ satisfy } 
    \eqref{eq:stateequation},\eqref{eq:control_constraint}\}.
    \end{equation*}
  By \Cref{ass:feasible_point}, it holds $(\qc,\uuc) \in M$ and thus $M \neq \emptyset$.
  }
   Since the objective functional is bounded from below by $0$, there exists a minimizing sequence
   $\{(q_n,\uu_n)\} \subset \revD{M}$ such that
   \begin{equation*}
      J(\revB{q_n,\uu_n}) \rightarrow \revC{\bar{J}} :=
      \revD{\inf \{J(\revB{q,\uu}) : (q,\uu) \in M\}.}
   \end{equation*}
   \revD{Especially, due to the equivalence of \eqref{eq:minProb} and \eqref{eq:minProb_3},
     for each $n \in \N$, $(q_n,\uu_n)$ satisfy the initial condition
     $\uu_n(0)=\uu_0$ and the state equation
      \begin{equation}\label{eq:stateequation_sequence}
           \Opair{\partial_t \uu_n,\vv} +  \Oprod{\nabla \uu_n,\nabla \vv} + \Oprod{(\uu_n\cdot \nabla \uu_n),\vv}
        = - \sum_{i=1}^L q_{n,i} \Gprodj{\vv,\nn}{N,i} \quad \textforall \vv \in \V,  \alev t \in I,
      \end{equation}
      where we have written the right hand side as in the original problem formulation 
      \eqref{eq:minProb}.}
   As $(\qc,\uuc)$ are feasible, \revD{it holds $J(\qc,\uuc)\ge \bar J$.
     In the case of equality, we immediately obtain, that $(\qbar,\ubar)=(\qc,\uuc)$ is a global
     solution. Hence we can w.l.o.g. assume that $J(\qc,\uuc) > \bar J$. 
    In this case we obtain,} that for $n$ large enough, there holds 
   \begin{equation*}
      \frac{\alpha}{2} \|q_n\|_{L^2(I)}^2 \le J(q_n,\uu_n) \le J(\qc,\uuc) 
      = \frac14 \|\uuc -\uu_d\|_{L^4(I\times \Om)}^4 + \frac{\alpha}{2} \|\qc\|_{L^2(I)}^2.
   \end{equation*}
   Dividing by $\frac\alpha2$ and taking the square root, yields
   \begin{equation}\label{eq:minseq_boundedness_q}
      \|q_n\|_{L^2(I)} \le (2\alpha)^{-\frac12} \|\uuc - \uu_d\|_{L^4(I\times \Om)}^2 + \|\qc\|_{L^2(I)}
     = \text{ const.}
   \end{equation}
   Similarly, using the same arguments plus the triangle inequality, 
   from $\frac14 \|\uu_n - \uu_d\|_{L^4(I\times \Om)}^4 \le J(\qc,\uuc)$, we obtain
   \begin{equation}\label{eq:minseq_boundedness_u}
     \|\uu_n\|_{L^4(I\times \Om)} \le \|\uuc - \uu_d\|_{L^4(I\times \Om)} +
     (2 \alpha)^{\frac14} \|\qc\|_{L^2(I)}^{\frac12} + \|\uu_d\|_{L^4(I\times \Om)} = \text{const.}
   \end{equation}
   \revD{As we have just shown a uniform bound for $\|\uu_n\|_{L^4(I\times \Om)}$, for each $n \in \N$
      the corresponding $\uu_n$ may be used as a linearization point in \eqref{eq:linearized_NS}. 
      We thus define $\ww_n \in \VV$ to be the unique solution to
      \begin{equation}\label{eq:linearized_sequence}
           \Opair{\partial_t \ww_n,\vv} +  \Oprod{\nabla \ww_n,\nabla \vv} + \Oprod{(\uu_n\cdot \nabla \ww_n),\vv}
        = - \sum_{i=1}^L q_{n,i} \Gprodj{\vv,\nn}{N,i} \quad \textforall \vv \in \V,  \alev t \in I,
      \end{equation}
      with initial condition $\ww_n(0) = \uu_0$.
      By \Cref{thm:linearized_NS} there holds
   \begin{equation}\label{eq:bound_wn}
     \|\partial_t \ww_n\|_{L^2(I;\V^*)}+ \|\ww_n\|_{L^2(I;\V)} + \|\ww_n\|_{L^\infty(I;\Hspace)} 
   \le C\left(\|\uu_n\|_{L^4(I\times \Om)}\right)\left(\|\uu_0\|_\Hspace + \|q_n\|_{L^2(I)}\right).
   \end{equation}
   Since for every $n \in \N$, it holds $\uu_n \in \VV$, we may also insert $\uu_n$ 
   in place of $\ww_n$ into \eqref{eq:linearized_sequence}. In this case we have on the left hand
   side exactly the left hand side of \eqref{eq:stateequation_sequence}. 
   As the right hand sides of \eqref{eq:stateequation_sequence} and \eqref{eq:linearized_sequence}
   coincide, the uniqueness of solutions to \eqref{eq:linearized_sequence} implies $\ww_n = \uu_n$.
   Hence the estimate \eqref{eq:bound_wn} translates to
   }
   \begin{equation}\label{eq:minimizing_sequence_bounded}
     \revB{\|\partial_t \uu_n\|_{L^2(I;\V^*)}+}
   \|\uu_n\|_{L^2(I;\V)} + \|\uu_n\|_{L^\infty(I;\Hspace)} 
   \le C\left(\|\uu_n\|_{L^4(\revB{I\times \Om})}\right)\left(\|\uu_0\|_\Hspace + \|q_n\|_{L^2(I)}\right) \le C,
   \end{equation}
   which due to \eqref{eq:minseq_boundedness_q} and \eqref{eq:minseq_boundedness_u} holds uniformly for all $n$.
   \revC{After} we have shown these boundedness statements,
   we can extract subsequences, which converge in some topologies.
   Due to \eqref{eq:minseq_boundedness_q}, for a subsequence of controls, again denoted by $\{q_n\}$, it holds
   $q_n \rightharpoonup \revA{\qbar}$ weakly in $\revB{Q}$.
      As the set of controls satisfying the constraints is closed and convex, it is 
    weakly closed, and thus $q_a \le \qbar \le q_b$.
   For a subsequence of the states, again denoted by $\{\uu_n\}$, we will show:
   \begin{align}
     \label{alig:weak_conv_L2H1}
     \uu_n \rightharpoonup & \ \revA{\ubar},  \text{ weakly in } L^2(I;\V),\\
     \label{alig:weakstar_conv_H1Vstar}
     \revB{\partial_t \uu_n \overset{*}{\rightharpoonup}} & \ \revB{\partial_t \ubar,  \text{ weak-star in } L^2(I;\V^*),}\\
     \label{alig:weakstar_conv_LinftyL2}
     \uu_n \overset{*}{\rightharpoonup} & \ \revA{\ubar},   \text{ weak-star in } L^\infty(I;\Hspace),\\
     \label{alig:strong_conv_L2L2}
     \uu_n \rightarrow & \ \revA{\ubar},  \text{ strongly in } L^2(I;\revB{\Hspace}).
   \end{align}
   The statements \eqref{alig:weak_conv_L2H1},\revB{\eqref{alig:weakstar_conv_H1Vstar}} and \eqref{alig:weakstar_conv_LinftyL2} 
   can immediately be deduced from the bound
   \eqref{eq:minimizing_sequence_bounded}. The strong convergence \eqref{alig:strong_conv_L2L2} is a 
   consequence of \eqref{eq:minimizing_sequence_bounded} and \Cref{thm:compact_imbedding} 
   \revB{with the choice $X_0 = \V$, $X_1 = \V^*$ and $X = \Hspace$}.
    We conclude our summary of convergence results by showing that 
   for any $\ww \in L^\infty(I;L^4(\Omega)^2)$, it holds
   \begin{equation}\label{eq:NonlinDomainConvergence}
     \IOprod{(\uu_n \cdot \nabla) \uu_n - (\revA{\ubar} \cdot \nabla)\revA{\ubar}, \ww} \rightarrow 0.
   \end{equation}
   \revB{By \eqref{eq:trilinear_l4_est} we obtain the estimate}
   \begin{align*}
     \left| \IOprod{(\revA{\ubar} \cdot \nabla) \yy,\ww} \right| 
      &\le C \int_0^T \|\revA{\ubar}\|_{L^4(\Omega)} \|\nabla \yy \|_{L^2(\Omega)} \|\ww\|_{L^4(\Omega)} \ dt
      \le C \int_0^T \|\nabla \revA{\ubar}\|_{L^2(\Omega)}^{\frac{1}{2}}\|\revA{\ubar}\|_{L^2(\Omega)}^{\frac{1}{2}} 
         \|\nabla \yy \|_{L^2(\Omega)} \|\ww\|_{L^4(\Omega)} \ dt\\
      &\le C \|\revA{\ubar}\|_{L^2(I;H^1(\Omega))}^{\frac{1}{2}} 
      \|\revA{\ubar}\|_{L^2(\revB{I\times \Om})}^{\frac{1}{2}}
      \|\yy\|_{L^2(I;H^1(\Omega))} \|\ww\|_{L^\infty(I;L^4(\Omega))},
   \end{align*}
   \revB{which shows that for fixed $\ww \in L^\infty(I;L^4(\Om)^2)$, the} 
   functional $L^2(I;\V) \ni \yy \mapsto \IOprod{(\revA{\ubar}\cdot \nabla)\yy,\ww} \in \R$
   is an element of $L^2(I;\V^*)$. 
   \revB{Thus by the weak convergence \eqref{alig:weak_conv_L2H1} on the one hand,
     and the strong convergence \eqref{alig:strong_conv_L2L2} and boundedness 
   \eqref{eq:minimizing_sequence_bounded}. on the other hand, we obtain
   \begin{equation*}
     |\IOprod{(\ubar\cdot \nabla)(\uu_n - \ubar),\ww}| \rightarrow 0
     \qquad \text{and} \qquad
     \revC{|\IOprod{((\uu_n - \ubar)\cdot \nabla)\uu_n,\ww}| \rightarrow 0.}
   \end{equation*}
 }
   With an application of triangle inequality we have thus shown \eqref{eq:NonlinDomainConvergence}.
   \revB{Hence,} for $\vv \in \V$ and $\vartheta \in C^\infty_0(\bar I)$ all terms of
   \begin{equation*}
     \IOpair{\partial_t \uu_n,\vv\vartheta}
      + \IOprod{\nabla \uu_n,\nabla \vv\vartheta}
      + \IOprod{(\uu_n\cdot \nabla) \uu_n,\vv\vartheta}
      =
      - \sum_{i=1}^L \IOprod{q_{n,i} \nabla \zeta_i,\vv\vartheta}
   \end{equation*}
   converge in such a way, that the whole equation converges to
   \begin{equation*}
     \revB{\IOpair{\partial_t \ubar,\vv\vartheta}}
     + \IOprod{\nabla \ubar,\nabla \vv\vartheta}
     + \IOprod{(\ubar\cdot \nabla) \ubar,\vv\vartheta}
      =
    - \sum_{i=1}^L \IOprod{\revB{\qbar}_{i} \nabla \zeta_i,\vv\vartheta}.
   \end{equation*}
   Thus, $\ubar$ is the solution to the Navier-Stokes equations with data $\qbar_i$, \revD{$\uu_0$}.
   Lastly, since the occuring norms in $J$ are weakly lower semicontinuous, it holds
   \begin{equation*}
      J(\revB{\qbar,\ubar}) \le \liminf_n J(\revB{q_n,\uu_n}) = \revC{\bar{J}}.
   \end{equation*}
   Hence we have shown that $(\revB{\qbar,\ubar})$ is a minimizer of \eqref{eq:minProb_3}.
   As \eqref{eq:minseq_boundedness_q} and \eqref{eq:minseq_boundedness_u} also hold in the limit, 
      the first two bounds stated in the theorem immediately follow.
   The bound for $\ubar$ in the norm of $\VV$ follows from taking the limit in 
      \eqref{eq:minimizing_sequence_bounded}.
      If $\uu_0 \in \V$, then \Cref{thm:NS_H2} shows the bound for $\ubar$ in the norm of $\W$.
\end{proof}
Problem \eqref{eq:minProb_3} is an optimal control problem in unreduced form.
To reduce it to a minimization problem posed only in terms of controls, we need to ensure, that 
every control in the admissible set possesses an associated state.
In the unreduced form, this is given by definition of the pair $(\revB{q,\uu})$ being feasible,
whereas in the reduced form, we need to consider the minimization over the following subset
\begin{equation*}
   \mathcal Q^S := \{q \in \revA{Q} : \exists \text{ weak solution } \uu \revB{\in \VV} 
   \text{ to } \eqref{eq:stateequation}\},
\end{equation*}
which by \Cref{prop:open_dataset_H2} is an open subset of $\revA{Q}$.
Note however, that in general $\mathcal Q^S$ is not a linear subspace of $\revA{Q}$.
Due to the uniqueness result, shown in \Cref{thm:uniqueness_stateeq},
on $\mathcal Q^S$ we can introduce a well defined
control-to-state mapping
\begin{equation*}
  S\colon \mathcal Q^S \to \W,
  q \mapsto \uu \text{ solving \eqref{eq:stateequation}}.
\end{equation*}
The box constraints further require the controls to be in the following convex subset of $Q$
\begin{equation*}
   \Qbox := \{ q \in Q : q_a \le q \le q_b \ \text{a.e. in } I\}.
\end{equation*}
The set of admissible controls is thus given as 
\begin{equation*}
   \Qad := \mathcal Q^S \cap \Qbox.
\end{equation*}
\revC{Under \Cref{ass:feasible_point}, $\Qad \neq \emptyset$.}
With these considerations, the reduced form of \eqref{eq:minProb_3} reads
\begin{equation}\label{eq:reduced_problem}\tag{$P_{red}$}
   \min_{q \in \Qad} \quad j(q) := J(q,S(q)) 
   = \frac{1}{4} \|S(q)-\uu_d\|_{L^4(I\times \Omega)}^4 + 
   \frac{\alpha}{2} \|q\|_{L^2(I)}^2.
\end{equation}
\subsection{First Order Optimality Conditions}
\revA{In order to derive first order optimality conditions, we first recall some differentiability results on 
   the control to state mapping $S$ as well as the reduced objective function $j$.
   Note that since $\mathcal Q^S \subset Q$ is open, we can take derivatives in any direction $\dq \in Q$.
}
\begin{theorem}\label{thm:state_derivatives}
  The control to state mapping $S$ is of class $C^\infty$, and for $q \in \mathcal Q^S$, $\uu=S(q)$,
  and $\dq \in \revA{Q}$, the derivative $\du = S'(q)(\dq)$ is given as the solution of 
  \begin{equation}\label{eq:first_derivatives}
        \left\lbrace
   \begin{aligned}
     \Oprod{\partial_t \du,\vv} + \Oprod{\nabla \du,\nabla \vv} 
     + \Oprod{(\uu\cdot \nabla)\du,\vv}
     + \Oprod{(\du\cdot \nabla)\uu,\vv} &=
     - \Oprod{\sum_{i=1}^L \dq_i \nabla \zeta_i,\vv} \quad \textforall \vv \in \V,  \alev t \in I,\\
   \du(0) &= 0.\\
   \end{aligned}
 \right.
\end{equation}
For $\rq \in \revA{Q}$, and $\ru = S'(q)(\rq)$,
the second derivative $\su = S''(q)(\dq,\rq)$ is given as the solution of
\begin{equation}\label{eq:second_derivatives}
        \left\lbrace
   \begin{aligned}
     \Oprod{\partial_t \su,\vv} + \Oprod{\nabla \su,\nabla v} 
     + \Oprod{(\uu\cdot \nabla)\su,\vv} \hspace{2cm} &\\
     + \Oprod{(\su\cdot \nabla)\uu,\vv} 
     + \Oprod{(\du\cdot \nabla) \ru,\vv} 
     + \Oprod{(\ru\cdot \nabla)\du,\vv} &=
     0 \quad \textforall \vv \in \V,  \alev t \in I,\\
   \su(0) &= 0.\\
   \end{aligned}
 \right.
\end{equation}
The first and second derivatives of the control to state mapping possess the regularities
$\du,\ru,\su \in \W$, defined in \eqref{eq:space_W_definition}.
\end{theorem}
\begin{proof}
  \revD{Arguing similarly to the proof of \Cref{thm:linearized_NS}, one can show that the systems
    \eqref{eq:first_derivatives} and \eqref{eq:second_derivatives} respectively are wellposed.
    With this, an invokation of the implicit function theorem then shows the differentiability 
    of $S$. Such an argument is carried out in detail in 
   \cite[Theorem 4.2.1, Corollary 4.2.2]{casas_optimal_1998}, which directly can be transferred to 
    the setting of this paper.}
  The regularities of the first derivates is argued in the proof of 
  \cite[Theorem 4.1]{benes_solutions_2016}.
  To argue the regularity of $\su$, one can show 
  $\du(\nabla \cdot \ru), \ru (\nabla \cdot \du) \in L^2(I;L^2(\Omega)^2)$,
  as in \cite[Theorem 3.8]{vexler_error_2024}.
  Bootstrapping then yields the same equation as for the first derivatives, which concludes the proof.
\end{proof}
\revD{Note that the results of \Cref{prop:open_dataset_H2} and \Cref{thm:state_derivatives} are 
  strongly related, as they both are a consequence of the wellposedness of the linearized equations
  together with the implicit function 
  theorem. We have chosen to present \Cref{thm:state_derivatives} here, to have a more compact 
  presentation of the optimal control theory section.}
Following a formal Lagrange approach, it is straightforward to deduce the adjoint equation.
The adjoint equation for this problem is given by: Find $\zz \in \VV$ such that
\begin{equation}\label{eq:adjoint_L4}
  \left\{
\begin{aligned}
   \revB{\Opair{-\partial_t \zz, \vv}} + \Oprod{\nabla \zz,\nabla \vv} 
  + \Oprod{(\uu\cdot \nabla)\vv + (\vv \cdot \nabla)\uu,\zz}
   &= \Oprod{|\uu-\uu_d|^2(\uu-\uu_d),\vv} \quad \textforall \vv \in \V, \alev t \in I,\\
   \zz(T) &= 0.\\
\end{aligned}
\right.
\end{equation}
\begin{theorem}\label{thm:adjoint_solvability}
   Let $\uu_d \in L^6(I;L^{3+\varepsilon}(\Omega)^2)$ for some $\varepsilon > 0$ and let 
   $\uu \in \W$ solve the state equation \eqref{eq:NS_fullData}.
   Then the adjoint equation \eqref{eq:adjoint_L4} possesses a \revA{unique} solution  
    $\zz \in \VV$, which satisfies the bound
  \begin{equation*}
    \|\partial_t \zz\|_{L^2(I;\V^*)} + \|\zz \|_{L^2(I;\V)} + \|\zz\|_{L^\infty(I;\Hspace)}
    \le C \|\uu - \uu_d\|_{L^6(I;L^{3+\varepsilon}(\Om))}^3.
  \end{equation*}
\end{theorem}
\begin{proof}
  We begin by showing that the right hand side $|\uu-\uu_d|^2(\uu-\uu_d) \in L^2(I;\V^*)$.
  Due to Hölder's inequality, and Sobolev imbedding $\V \hookrightarrow L^p(\Om)$ for any
  $1\le p<\infty$, there holds \revD{for every $\vv \in \V$ and } for any $\varepsilon > 0$
  \begin{equation*}
    \Oprod{|\uu-\uu_d|^2(\uu-\uu_d),\vv}
    \le C\|\uu-\uu_d\|_{L^{3+\varepsilon}(\Om)}^3 \|\nabla \vv\|_{L^2(\Om)}
    \quad \text{a.e. in } I.
  \end{equation*}
  Thus 
  \begin{equation*}
    \||\uu-\uu_d|^2(\uu-\uu_d)\|_{\V^*} \le C \|\uu-\uu_d\|_{L^{3+\varepsilon}(\Om)}^3 
    \le C \left(\|\uu\|_{L^{3+\varepsilon}(\Om)}^3 +\|\uu_d\|_{L^{3+\varepsilon}(\Om)}^3\right).
  \end{equation*}
  The regularity $\uu \in C(\bar I;\V)$ together with Sobolev imbedding 
  $\V \hookrightarrow L^{3+\varepsilon}(\Om)$, and the assumption on $\uu_d$
  shows, that the right-hand side of the above estimate is square integrable in time,
  which shows $|\uu-\uu_d|^2(\uu-\uu_d) \in L^2(I;\V^*)$.
  We now show the existence and uniqueness of a solution $\zz \in \VV$.
  Both follow from the stability estimate, where existence is obtained by, e.g., a Galerkin procedure and 
  uniqueness by considering the equation for the difference between two solutions $\zz_1-\zz_2$.
  To show the stability estimate, we test \eqref{eq:adjoint_L4} with its solution.
  \revB{Using \eqref{eq:trilinear_l4_est} and Young's inequality yields}
   \begin{equation*}
     \Oprod{(\zz\cdot \nabla)\uu,\zz)} \le C \|\zz\|_{L^4(\Omega)}^2\|\nabla \uu\|_{L^2(\Omega)}
     \le C \|\zz\|_{L^2(\Omega)}\|\nabla \zz \|_{L^2(\Omega)}\|\nabla \uu\|_{L^2(\Omega)}
     \le \revC{\delta} \|\nabla \zz \|_{L^2(\Omega)}^2 
     + C_{\revC{\delta}} \|\zz\|_{L^2(\Omega)}^2\|\nabla \uu\|_{L^2(\Omega)}^2,
   \end{equation*}
   \revB{where $\revC{\delta} > 0$ can be made arbitrarily small.
   By \eqref{eq:trilinear_boundary_term}, there holds }
   $
     \Oprod{(\uu \cdot \nabla) \zz,\zz} = \frac{1}{2} (\uu\cdot \nn, \zz\cdot \zz)_{\partial \Omega},
     $
   which with \Cref{thm:trace_interpolation_2} and Young's inequality, we can bound by
   \begin{align*}
      \frac{1}{2} \dOprod{\uu \cdot \nn, \zz \cdot \zz}
      &\le 
     C \|\uu\|_{L^3(\partial \Omega)}\|\zz\|_{L^3(\dO)}^2 \le 
     C \|\nabla \uu\|_{L^2(\Omega)}^{\frac{2}{3}}\|\uu\|_{L^2(\Omega)}^{\frac{1}{3}}
     \|\nabla \zz\|_{L^2(\Omega)}^{\frac{4}{3}}\|\zz\|_{L^2(\Omega)}^{\frac{2}{3}}\\
      & \le \revC{\delta} \|\nabla \zz\|_{L^2(\Omega)}^2 
         + C_{\revC{\delta}} \|\nabla \uu\|_{L^2(\Omega)}^2\|\uu\|_{L^2(\Omega)} \|\zz\|_{L^2(\Omega)}^2.
   \end{align*}
   Choosing $\revC{\delta}$ small enough, allows us to absorb the $\revC{\delta} \|\nabla \zz\|_{L^2(\Om)}^2$ terms,
   and conclude the proof by a standard Gronwall argument.
\end{proof}
  \begin{theorem}\label{thm:optimality_condition}
    The reduced objective function $j\colon\mathcal Q^S \to \R$ is of class $C^\infty$, and its 
   first derivative \revA{in any direction $\dq \in Q$} satisfies the representation
    \begin{equation*}
       j'(q)(\dq) = \alpha \Iprod{q,\dq} \revB{-} \sum_{i=1}^L \Iprod{\dq_i, \Oprod{\nabla \zeta_i,\zz}},
    \end{equation*}
    with $\uu = S(q)$ and $\zz$ satisfying \eqref{eq:adjoint_L4}.
  \end{theorem}
  \begin{proof}
    By differentiating, we obtain
    \begin{equation*}
       j'(q)(\dq) = \partial_q J(q,S(q))(\dq) + \partial_u J(q,S(q))\cdot S'(q)(\dq)
      = \alpha \Iprod{q,\dq} + \IOprod{|\uu-\uu_d|^2(\uu-\uu_d),S'(q)(\dq)}.
    \end{equation*}
    With the definitions of the adjoint state $\zz$ and $S'(q)(\dq)$, \revD{as well as 
    integration by parts in time,} the claimed form immediately follows.
  \end{proof}
  It is straightforward to check that a local solution $\bar q$ to \revB{\eqref{eq:reduced_problem}} satisfies
  \begin{equation}\label{eq:first_order_opt_cond}
   j'(\bar q)(q- \bar q) \ge 0 \quad \textforall q \in \revA{\Qbox}.
  \end{equation}
  From this we deduce that 
  \begin{align*}
     \bar q_i(t) = q_{a,i} &\Rightarrow \alpha \bar q_i(t) \revB{-} \Oprod{\nabla \zeta_i,\zbar(t)} \ge 0
     \revA{\quad \text{for a.e. } t \in I,}\\
     q_{a,i} < \bar q_i(t) < q_{b,i} &\Rightarrow \alpha \bar q_i(t) \revB{-} \Oprod{\nabla \zeta_i, \zbar(t)} = 0
     \revA{\quad \text{for a.e. } t \in I,}\\
     \bar q_i(t) = q_{b,i} &\Rightarrow \alpha \bar q_i(t) \revB{-} \Oprod{\nabla \zeta_i,\zbar(t)} \le 0
     \revA{\quad \text{for a.e. } t \in I.}
  \end{align*}
  Conversely \revA{for almost every $t \in I$, there holds}
  \begin{equation}\label{eq:bang_bang_structure}
     \alpha \bar q_i(t) \revB{-} \Oprod{\nabla \zeta_i, \zbar(t)} < 0 \Rightarrow \bar q_i(t) = q_{b,i}
     \qquad \text{and} \qquad
     \alpha \bar q_i(t) \revB{-} \Oprod{\nabla \zeta_i, \zbar(t)} > 0 \Rightarrow \bar q_i(t) = q_{a,i}.
  \end{equation}
  All in all, we have just shown, that \revC{an} optimal control $\qbar$ satisfies the representation
  \begin{equation}\label{eq:qbar_projection_formula}
    \bar q_i = P_{[q_{a,i},q_{b,i}]} \left(\frac{1}{\alpha} \Oprod{\nabla \zeta_i,\zbar}\right) 
    = P_{[q_{a,i},q_{b,i}]}\left(\frac{1}{\alpha} \Gprodj{\zbar , \nn}{\Gamma_{N,i}}\right) \quad i=1,...,L,
  \end{equation}
  where $P_{[c,d]}(y) := \max\{\min\{y,d\},c\}$ denotes the projection of $y$ onto the interval $[c,d]$.
  Note that using the construction of \Cref{lemm:boundary_to_rhs},
  we can write $\qbar$ equivalently in two different ways depending on $\zbar$,
  where the more intuitive representation might be the boundary integral one, whereas the domain integral 
  \revB{form} is the one allowing us to deduce more regularity of $\qbar$.
  \revB{Let us summarize the first order optimality conditions, which we have derived so far.}
  \begin{theorem}\label{thm:adjoint_regularity}
  Let $\qbar \in \Qad$ be \revA{a local minimum of} \eqref{eq:reduced_problem}. 
  Then there exist $\ubar = S(\qbar)$ and $\zbar$ such that 
  \begin{itemize}
     \item $\ubar \in \W$ solves the state equation
       \begin{equation*}
        \left\lbrace
   \begin{aligned}
     \Oprod{\partial_t \ubar,\vv} + \Oprod{\nabla \ubar,\nabla \vv} + \Oprod{(\ubar\cdot \nabla)\ubar,\vv} &=
     - \Oprod{\sum_{i=1}^L \qbar_i \nabla \zeta_i,\vv} \quad \textforall \vv \in \V,  \alev t \in I,\\
   \ubar(0) &= \uu_0,\\
   \end{aligned}
 \right.
       \end{equation*}
   \item $\zbar \in \VV$ solves the adjoint equation
          \begin{equation*}
           \left\{
         \begin{aligned}
            \revB{\Opair{-\partial_t \zbar, \vv}} + \Oprod{\nabla \zbar,\nabla \vv} 
           + \Oprod{(\ubar\cdot \nabla)\vv + (\vv \cdot \nabla)\ubar,\zbar}
            &= \Oprod{|\ubar-\uu_d|^2(\ubar-\uu_d),\vv} \quad \textforall \vv \in \V, \alev t \in I,\\
            \zbar(T) &= 0.\\
         \end{aligned}
         \right.
          \end{equation*}
    \item $\qbar \in H^{\frac{1}{2}}(I)^L \cap C(\bar I)^L$ satisfies the projection formula
       \begin{equation*}
          \qbar_i = P_{[q_{a,i},q_{b,i}]}\left( \frac{1}{\alpha} \Oprod{\nabla \zeta_i,\zbar}\right)
     \qquad i=1,...,L.
       \end{equation*}
  \end{itemize}
\revD{Moreover, $\ubar$ posesses an associated pressure $p_{\bar \uu} \in L^2(I;H^1(\Om))$
  in the sense of \Cref{thm:NS_H2}.}
\end{theorem}
\begin{proof}
  The regularity of \revD{$(\ubar, p_{\bar \uu})$} was shown in \Cref{thm:NS_H2}, the one of $\zbar$ in \Cref{thm:adjoint_solvability}.
  Regarding the expression for $\qbar$, note that 
  \begin{equation*}
    \zbar \in L^2(I;\V) \cap H^1(I;\V^*) \Rightarrow \zbar \in C(\bar I;\Hspace) \cap H^{\frac{1}{2}}(I;\Hspace). 
  \end{equation*}
  Together with the explicit representation of $\qbar$ in terms of $\zbar$, shown in 
  \eqref{eq:qbar_projection_formula}, \revA{and the fact, that the projection onto the 
     set of feasible controls is stable in $H^{\frac12}(I)^L$, see e.g. 
  \cite[Lemma 3.3]{kunisch_constrained_2007}, }
  this concludes the proof.
\end{proof}
\revD{Note that in the above theorem, we could also have introduced an associated pressure to
  $\zbar$. As the optimal control $\qbar$ only depends on $\zbar$, for our application an 
  adjoint pressure has no direct significance.
  Hence, in the following we shall only discuss the primal pressure.}
\subsection{Higher regularity of optimal variables}
\revA{The regularity results, obtained in \Cref{thm:adjoint_regularity}, were all deduced by assuming only
the minimum regularity of $\qbar$ given by the feasibility condition $\qbar \in \Qad \subset L^2(I)^L$ and 
applying the optimality conditions. As in the third result of \Cref{thm:adjoint_regularity}, we obtain an 
improved regularity of $\qbar$, we can iterate this process, to obtain higher regularity of 
$(\revD{\qbar,\ubar, p_{\ubar},\zbar})$.
The main difficulty here will lie in the discussion of the adjoint state $\zbar$.}
Let us recall some regularity results for auxiliary problems.
The first two are regularity results for a stationary Stokes problem with mixed boundary conditions 
and nonzero divergence. In the following we use the notation
\begin{equation*}
   \widehat{\V} := \{\vv \in H^1(\Om)^2: \vv|_{\Gamma_D} = \oo\}.
\end{equation*}
   Then the weak formulation of 
  \begin{equation}\label{eq:stat_stokes_inhom_divergence}
    \left\lbrace
    \begin{aligned}
       -\Delta \ww + \nabla p_\ww &= \ff && \textin \Omega,\\
       \nabla \cdot \ww &= \rho && \textin \Omega,\\
       \ww &= \oo && \texton \Gamma_D,\\
       -\partial_\nn \ww + p_\ww \nn &= \oo  &&\texton \Gamma_N
    \end{aligned}
      \right.
  \end{equation}
  \revC{for $f \in \widehat{\V}^*$ and $\rho \in L^2(\Om)$} 
  reads: Find $(\ww,p_\ww) \in \widehat{\V} \times L^2(\Om)$, such that 
  \begin{equation}\label{eq:weak_stat_stokes_inhom_divergence}
    \Oprod{\nabla \ww,\nabla \vv} - \Oprod{\nabla \cdot \vv, p_\ww} + \Oprod{\nabla \cdot \ww,p_\vv} = 
    \Opair{\ff,\vv} + \Oprod{\rho,p_\vv}
    \quad \textforall (\vv,p_\vv) \in \widehat{\V} \times L^2(\Om).
  \end{equation}
\begin{lemma}\label{lemm:stat_stokes_inhom_divergence}
   Let $\ff \in \widehat{\V}^*$ and $\rho \in L^2(\Om)$.
   \revB{Then \eqref{eq:weak_stat_stokes_inhom_divergence} has}
  a unique solution $(\ww,p_\ww) \in \widehat{V} \times L^2(\Om)$, and there exists a constant 
     $C>0$ such that 
  \begin{equation*}
     \|\ww\|_{H^1(\Om)} + \|p_\ww\|_{L^2(\Om)} \le 
     C\left(\|\ff\|_{\widehat{\V}^*} + \|\rho\|_{L^2(\Om)}\right).
  \end{equation*}
  If moreover, $\ff \in L^2(\Om)^2$ and $\rho \in H^1(\Om)$, it holds 
  $(\ww,p_\ww) \in \left[\revB{\widehat{\V} \cap H^2(\Om)^2}\right] \times H^1(\Om)$ and
  there exists a constant $C>0$ such that 
  \begin{equation*}
     \|\ww\|_{H^2(\Om)} + \|\nabla p_\ww\|_{L^2(\Om)} \le 
     C\left(\|\ff\|_{L^2(\Om)} + \|\rho\|_{H^1(\Om)}\right).
  \end{equation*}
\end{lemma}
\begin{proof}
   The unique solvability is a standard result, and can be found in, e.g.,
   \cite[Theorem 4.2.3]{boffi_mixed_2013}. The higher regularity result can be shown by a localization 
   argument, which is carried out in \cite[Appendix A.3]{lasiecka_boundary_2018} for a rectangular domain, but 
   due to \Cref{ass:domain} the result holds in our setting as well.
\end{proof}
In the case, that the right hand side $\rho$ of the divergence condition is already the divergence of 
some known vector field, by a duality argument, we obtain a bound for the $L^2$ norm of the velocity field
of the stationary problem.
\begin{theorem}\label{thm:stationary_stokes_inhom_divergence_l2_est}
   Let $\EE \in H^1_0(\Om)^2$ and \revA{$\ff \in \widehat{\V}^*$.}
   Then the velocity component $\ww \in \widehat{\V}$ of the weak solution to
   \revB{\eqref{eq:weak_stat_stokes_inhom_divergence}} with $\rho = \nabla \cdot \EE$ satisfies
  \begin{equation*}
     \|\ww\|_{L^2(\Om)} \le C \left(\|\EE\|_{L^2(\Om)} + \revA{\|\ff\|_{\widehat{\V}^{-2}}}\right),
  \end{equation*}
  \revA{where we use the notation $\widehat{\V}^{-2} := (\widehat{\V}\cap H^2(\Om)^2)^*$.}
\end{theorem}
\begin{proof}
  We use the above lemma in a dual formulation: given $\ppsi \in L^2(\Om)^2$, 
  let $(\zz^\ppsi,p_\zz^\ppsi) \in \widehat{\V}\times L^2(\Om)$ satisfy
  \begin{equation*}
     \Oprod{\nabla \vv,\nabla \zz^\ppsi} - \Oprod{\nabla \cdot \zz^\ppsi, p_\vv} 
     + \Oprod{\nabla \cdot \vv,p_\zz^\ppsi} = 
    \Oprod{\ppsi,\vv}
    \quad \textforall (\vv,p_\vv) \in \widehat{\V} \times L^2(\Om).
  \end{equation*}
  Then it holds 
  \begin{equation*}
     \|\zz^\ppsi\|_{H^2(\Om)} + \|\nabla p_\zz^\ppsi\|_{L^2(\Om)} \le \revB{C}\|\ppsi\|_{L^2(\Om)}.
  \end{equation*}
  Thus \revB{choosing $\ppsi = \ww$}, using the dual and primal equations,
  integration by parts, $\EE|_{\partial \Om} = 0$ 
  and the \revA{above estimate of $\zz^\ppsi$ and $\nabla p_\zz^\ppsi$ in terms of }
  the $L^2$ norm of $\ppsi$, \revB{yields}
  \revB{
  \begin{align*}
    \|\ww\|_{L^2(\Om)}^2
    & = \Oprod{\ww,\ppsi}
    = \Oprod{\nabla \ww,\nabla \zz^\ppsi} 
    - \Oprod{\nabla \cdot \zz^\ppsi, p_\ww} + \Oprod{\nabla \cdot \ww,p_\zz^\ppsi}
    = \Opair{\ff,\zz^\ppsi} + \Oprod{\rho,p_\zz^\ppsi}\\
    & = \Opair{\ff,\zz^\ppsi} + \Oprod{\nabla \cdot \EE,p_\zz^\ppsi}
    =\Opair{\ff,\zz^\ppsi} -\Oprod{\EE, \nabla p_\zz^\ppsi}
    \le C \left(\|\EE\|_{L^2(\Om)} + \|\ff\|_{\widehat{\V}^{-2}}\right)\|\ppsi\|_{L^2(\Om)}.
  \end{align*}
  Dividing by $\|\ppsi\|_{L^2(\Om)} = \|\ww\|_{L^2(\Om)}$ concludes the proof.
  }
\end{proof}
\revA{Due to the above results, the mapping $H^1_0(\Om)^2 \ni \EE \mapsto \ww \in L^2(\Om)^2$, where 
   $\ww$ solves \eqref{eq:stat_stokes_inhom_divergence} with $\rho = \nabla \cdot \EE$ and $\ff = 0$,
   is a well defined, linear operator.
   Moreover, by density of $H^1_0(\Om) \subset L^2(\Om)$ and the stability estimate
   of \Cref{thm:stationary_stokes_inhom_divergence_l2_est}, this mapping extends to a 
   continuous linear operator
   \begin{equation}\label{eq:definition_operator_T}
      R\colon L^2(\Om)^2 \to L^2(\Om)^2, \text{ satisfying } \|R(\EE)\|_{L^2(\Om)} \le C \|\EE\|_{L^2(\Om)} 
      \quad \text{for all } \EE \in L^2(\Om)^2.
   \end{equation}
   As for $\EE \in H^1_0(\Om)^2$, $R(\EE)$ coincides with the composition of the divergence operator and 
   the solution operator of \eqref{eq:stat_stokes_inhom_divergence}, by \Cref{lemm:stat_stokes_inhom_divergence}
   there hold the bounds
   \begin{equation}\label{eq:regularity_operator_T}
      \|R(\EE)\|_{H^1(\Om)} \le C \|\EE\|_{H^1(\Om)} \revC{\text{ for } \EE \in H^1_0(\Om)^2}
      \ \text{ and } \
       \|R(\EE)\|_{H^2(\Om)} \le C \|\EE\|_{H^2(\Om)} \revC{\text{ for } \EE \in (H^2(\Om) \cap H^1_0(\Om))^2}.
   \end{equation}
}
Let us now recall some instationary results. Since we have already seen that $\qbar \in \revB{C(\bar I)^L}$,
by maximal parabolic regularity, $\ubar$ will also turn out to be more regular in time. 
\revD{For the instationary Stokes equations, we have the following result, expressed in terms of 
the stationary Stokes operator $A$ introduced in \eqref{eq:def_Stokes_operator}.}
\begin{theorem}\label{thm:max_par_reg}
  Let $s \in (1,\infty)$, $\ff \in L^s(I;L^2(\Om)^2)$ and 
  $\uu_0 \in \V_s := (\Hspace,\V\cap H^2(\Om)^2)_{1-1/s,s}$.
   Then the weak solution 
   $\ww \in L^2(I;\V)$ of 
   \begin{equation*}
      \Opair{\partial_t \ww,\vv} + \Oprod{\nabla \ww,\nabla \vv} = \Oprod{\ff,\vv} \quad
      \revD{\textforall \vv \in \V, \alev t \in I}
   \end{equation*}
   and $\ww(0) = \uu_0$
   satisfies $\partial_t \ww \in L^s(I;L^2(\Om)^2)$, \revA{$A\ww \in L^s(I;L^2(\Om)^2)$} and there holds
   \begin{equation*}
      \|\partial_t \ww\|_{L^s(I;L^2(\Om))} + \revA{\|A\ww \|_{L^s(I;L^2(\Om))}}
      \le C \left(\|\ff\|_{L^s(I;L^2(\Om))} + \|\uu_0\|_{\V_s}\right).
   \end{equation*}
   Furthermore, due to \Cref{ass:domain}, there holds $\ww \in L^s(I;H^2(\Om)^2)$, with a bound
   \begin{equation*}
      \|\ww \|_{L^s(I;H^2(\Om))}\le C \left(\|\ff\|_{L^s(I;L^2(\Om))} + \|\uu_0\|_{\V_s}\right).
   \end{equation*}
\end{theorem}
\begin{proof}
  For homogeneous initial conditions, this result is standard. For the correct space for 
  inhomogeneous initial conditions defined by real interpolation, we refer to 
  \cite[Lemma 2.4]{denk_introduction_2021}.
  \revA{The $H^2(\Om)$ regularity is a consequence of \Cref{thm:mixed_Stokes_h2}.}
\end{proof}
Since the controls in our setting are purely time dependent, in \Cref{lemm:boundary_to_rhs}
we were able to transform the 
state equation subject to inhomogeneous boundary data to one with homogeneous boundary data by 
modifying the right hand side of the equation. We will see that for the adjoint equation this is not possible.
Thus, we \revC{require} the following extension theorem of Neumann boundary data.
To this end, recall the \revA{definition of the} space 
\begin{equation*}
   H^{\frac{1}{2}}_{00}((a,b)) := \{ \phi \in H^{\frac{1}{2}}((a,b)) : \ext_0(\phi) \in H^{\frac{1}{2}}(\R)\},
\end{equation*}
where $\ext_0(\phi)$ denotes the extension of $\phi$ to $\R$ by zero on $\R\backslash (a,b)$.
The spaces $H^{\frac{1}{2}}_{00}(\GNi)$ are then defined by affinely transforming $\GNi$ to an interval in $\R$. 
\revA{By \cite[Chapter 1, Theorem 12.3]{lions_nonhomogeneous_1972_vol1}, they can equivalently be written 
as $H^{\frac12}_{00}(\GNi) = [L^2(\GNi),H^1_0(\GNi)]_{\frac12}$.}
\begin{theorem}\label{thm:instat_Neumann_extension}
   Let $\gg_i \in H^{\frac{1}{4}}(I;L^2(\GNi)^2) \cap L^2(I;H^{\frac{1}{2}}_{00}(\GNi)^2)$ 
   for $i=1,...,L$.
   Then there exists a Neumann extension $\EE_\gg$ satisfying
   \begin{equation*}
      \EE_\gg \in L^2(I;H^2(\Om)^2\cap H^1_0(\Om)^2) \cap H^1(I;L^2(\Om)^2)
      \quad \text{and} \quad
      \partial_\nn \EE_\gg |_{\GNi} = \gg_i,
   \end{equation*}
   and there holds a bound
   \begin{equation*}
      \|\partial_t \EE_\gg\|_{L^2(I;L^2(\Om))} + \|\EE_\gg\|_{L^2(I;H^2(\Om))}
      \le C \sum_{i=1}^L \|\gg_i\|_{H^{\frac{1}{4}}(I;L^2(\GNi)^2) \cap
         L^2(I;H^{\frac{1}{2}}_{00}(\GNi)^2)}.
   \end{equation*}
\end{theorem}
\revA{Before we turn to proving the theorem, note that its result states $\EE_\gg \in H^1_0(\Om)^2$ for 
  a.a. $t \in I$, i.e. it allows us to specify homogeneous Dirichlet data and inhomogeneous Neumann 
data simultaneously on all $\GNi$.}
\begin{proof}
Away from the corners of $\Omega$ such an extension result is standard, \revA{and can be found e.g. in 
\cite[Chapter 4, Theorem 2.3]{lions_nonhomogeneous_1972_vol2}.}
Close to a corner, 
we follow an argument of \cite{lions_nonhomogeneous_1972_vol1,lions_nonhomogeneous_1972_vol2}.
The following scalar considerations can be applied to the individual components of $\gg_i$.
Consider the sector $\Xi = \R^+ \times \R^+ \revB{\subset \R^2}$. 
\revB{We wish to specify only homogeneous Dirichlet boundary data on $\{(x,y) \in \partial \Xi: y =0\}$
   whereas we want to set combined inhomogeneous Neumann and homogeneous Dirichlet data on 
   $\{(x,y) \in \partial\Xi: x=0\}$.}
For the sake of readability, in the remainder of the proof we index all Lebesque and Sobolev spaces 
$L^2_\xi, H^k_\xi$ by the variables $\xi \in \{t,x,y\}$ that they are associated with.
From \cite[Theorem  4.1]{seeley_interpolation_1972}, we obtain
$[H^2_y(\R^+) \cap H^1_{0,y}(\R^+),L^2_y(\R^+)]_{\frac{3}{4}} = H^{\frac{1}{2}}_{00,y}(\R^+)$.
For related results on interpolation spaces with homogeneous boundary conditions, see also
\cite[Theorem 8.1]{grisvard_caracterisation_1967}, \cite[Section 4.3.3]{triebel_interpolation_1978}
and \cite{grubb_regularity_2016}.
We introduce the spaces
\begin{equation*}
   X := H^1_t\left(I;L^2_y(\R^+)\right)\cap L^2_t\left(I;H^2_y(\R^+)\cap H^1_{0,y}(\R^+)\right) \qquad
   Y := L^2_t\left(I;L^2_y(\R^+)\right).
\end{equation*}
Thus, \revA{due to \cite[Chapter 1, Equation (9.24)]{lions_nonhomogeneous_1972_vol1}}, we have 
$
\left[X,Y\right]_{\frac{3}{4}}
  = H^{\frac{1}{4}}_t(I;L^2_y(\R^+)) \cap L^2_t(I;H^{\frac{1}{2}}_{00,y}(\R^+))
  $.
By \cite[Ch. 1, Thm. 3.2]{lions_nonhomogeneous_1972_vol1}, for every 
$g \in \left[X,Y\right]_{\frac{3}{4}}$, there is a 
\begin{equation}\label{eq:reordered_space_neumann_extension}
  E_g \in L^2_x(\R^+;X) 
  \cap H^2_x(\R^+;Y),
\end{equation}
satisfying \revB{$\partial_x E_g|_{\{x=0\}} = g$ and $E_g|_{\{x=0\}}=0$.}
\revA{Due to the intermediate derivative theorem \eqref{eq:intermediate_derivative_theorem} there further holds
  \begin{equation*}
    E_g \in H^1_x(\R^+;[X,Y]_{\frac12}) \hookrightarrow H^1_x(\R^+;L^2_t(I;H^1_y(\R^+))).
  \end{equation*}
  Thus, by reordering the derivatives, we have just proven
}
\begin{equation*}
  E_g\in L^2_t\left(I;H^2(\Xi) \cap H^1_{0}(\Xi)\right) \cap H^1_t(I;L^2(\Xi)),
\end{equation*}
which shows the theorem.
\end{proof}
Note that the above result provides an extension of the mixed Dirichlet-Neumann data, but in general 
$\nabla \cdot \EE_\gg \neq 0$. Nevertheless, using the result of the stationary Stokes problem with 
nonzero divergence, this extension result allows us to solve the instationary Stokes problem 
with inhomogeneous mixed boundary data.
\begin{theorem}\label{thm:stokes_inhom_neumann_data}
  Let $\ff \in L^2(I;L^2(\Om)^2)$ and 
  $\gg_i \in H^{\frac{1}{4}}(I;L^2(\GNi)^2) \cap L^2(I;H^{\frac{1}{2}}_{00}(\GNi)^2)$
  for $i = 1,...,L$.
  Then there exists a \revD{unique} weak solution 
  $(\ww,p_\ww) \in \W \times L^2(I;H^1(\Om))$ of
  \begin{equation}\label{eq:thm_Neumann_H2_weak_form}
    \Oprod{\partial_t \ww,\vv} + \Oprod{\nabla \ww,\nabla \vv} - \Oprod{p_\ww, \nabla \cdot \vv} 
    + \Oprod{\nabla \cdot \ww,p_\vv}
    = \Oprod{\ff,\vv} + \sum_{i=1}^L \Gprodj{\gg_i,\vv}{N,i}
  \end{equation}
  for all \revD{$(\vv,p_\vv) \in \widehat{\V}\times L^2(\Om)$, a.e. $t \in I$, and $\ww(0) = \oo$,}
satisfying
\begin{equation*}
  \|\partial_t \ww\|_{L^2(I;L^2(\Om))} + \|\ww\|_{L^2(I;H^2(\Om))}
  \revD{+ \|p_{\ww}\|_{L^2(I;H^1(\Om))}}
  \le C \left(\|\ff\|_{L^2(I;L^2(\Om))} +
  \sum_{i=1}^L \|\gg_i\|_{H^{\frac{1}{4}}(I;L^2(\GNi)) \cap L^2(I;H^{\frac{1}{2}}_{00}(\GNi))}\right).
\end{equation*}
\end{theorem}
\begin{proof}
  \revD{Unique solvability for this problem was already shown in 
    \Cref{prop:instat_stokes_solvability}. For low regularity of the data, existence
    of a distributional pressure can be shown with a similar argument as in the proof of
    \Cref{prop:instat_stokes_solvability}. Hence we only need to show the improved 
  regularity.}
   By \Cref{thm:instat_Neumann_extension}, there exists 
   $\EE_\gg \in L^2\left(I;\revC{(H^2(\Om)\cap H^1_0(\Om))^2}\right) \cap H^1(I;L^2(\Om)^2)$, such that 
  $\partial_\nn \EE_\gg|_{\GNi} = \gg_i$ and there holds a bound
  \begin{equation*}
  \|\partial_t \EE_\gg\|_{L^2(I;L^2(\Om))} + \|\EE_\gg\|_{L^2(I;H^2(\Om))}
  \le C \left(
  \sum_{i=1}^L \|\gg_i\|_{H^{\frac{1}{4}}(I;L^2(\GNi)) \cap L^2(I;H^{\frac{1}{2}}_{00}(\GNi))}\right).
  \end{equation*}
  Furthermore, \revA{using integration by parts}, we obtain 
  \begin{equation}\label{eq:Neumann_extension_partial_integration}
    \Oprod{\nabla \EE_\gg,\nabla \vv} = \Oprod{- \Delta \EE_\gg,\vv} 
    + \sum_{i=1}^L \Gprodj{\partial_\nn \EE_\gg,\vv}{N,i}
    \quad \textforall \vv \in \revD{\widehat{\V}, \alev t \in I}.
  \end{equation}
  Since at this point, $\EE_\gg$ might not have zero divergence, 
  let us introduce for a.e. $t \in I$ the function $\ww_1(t) = R(\EE_\gg(t))$ with $R$ defined as in 
   \eqref{eq:definition_operator_T}.
   As $\EE_\gg \in L^2\left(I;\revC{(H^2(\Om)\cap H^1_0(\Om))^2}\right)$,
   due to \eqref{eq:regularity_operator_T} there holds
  \begin{equation*}
    \|\ww_1\|_{L^2(I;H^2(\Om))} \le C \|\EE_\gg \|_{L^2(I;H^2(\Om))},
  \end{equation*}
  and $\ww_1(t) \revC{\in H^2(\Om)^2 \cap \widehat{\V}}$,
  together with some $p_{\ww_1}(t) \revC{\in H^1(\Om)}$, satisfy for a.e. $t \in I$ the equation
  \begin{equation}\label{eq:semistationary_problem}
     \Oprod{\nabla \ww_1(t),\nabla \vv} + \Oprod{\nabla \cdot \ww_1(t),p_\vv} 
     - \Oprod{\nabla \cdot \vv,p_{\ww_1}(t)}
     = \Oprod{- \nabla \cdot \EE_\gg \revA{(t)},p_\vv}
    \quad \textforall (\vv,p_\vv) \in \widehat{\V}\times L^2(\Om).
  \end{equation}
  Thus, by construction, it holds $\nabla \cdot (\EE_\gg(t) + \ww_1(t)) = 0$ for a.e. $t \in I$.
  As $R$ is independent of $t$, there holds 
  \begin{equation*}
     \partial_t \ww_1 = \partial_t R(\EE_\gg) = R(\partial_t \EE_\gg),
  \end{equation*}
  and due to $\partial_t \EE_\gg \in L^2(I;L^2(\Om)^2)$ and \eqref{eq:definition_operator_T}, there holds 
  \begin{equation*}
    \|\partial_t \ww_1\|_{L^2(I;L^2(\Om))}
    = \|R(\partial_t \EE_\gg)\|_{L^2(I;L^2(\Om))}
    \le C \|\partial_t \EE_\gg\|_{L^2(I;L^2(\Om))}.
  \end{equation*}
  Moreover, again using $\EE_\gg \in L^2(I;H^2(\Om)^2)$ and \Cref{lemm:stat_stokes_inhom_divergence}, 
  there holds 
  \begin{equation*}
    \|p_{\ww_1}\|_{L^2(I;H^1(\Om))} \le C \|\EE_\gg\|_{L^2(I;H^2(\Om))}.
  \end{equation*}
  Setting $\ww_2 := \ww - \EE_\gg - \ww_1$ and $p_{\ww_2} := p_\ww - p_{\ww_1}$, we obtain after
  subtracting
  \eqref{eq:Neumann_extension_partial_integration} and \eqref{eq:semistationary_problem} 
  from \eqref{eq:thm_Neumann_H2_weak_form} \revD{for a.e. $t \in I$}:
  \begin{equation*}
    \Oprod{\partial_t \ww,\vv} + \Oprod{\nabla \ww_2,\nabla \vv} 
    + \Oprod{\nabla \cdot (\ww \!-\! \ww_1),p_\vv}
    - \Oprod{p_{\ww_2},\nabla \cdot \vv}
    = \Oprod{\ff,\vv} 
    + \Oprod{\nabla \cdot \EE_\gg,p_\vv}
    - \Oprod{- \Delta \EE_\gg,\vv}. 
  \end{equation*}
  Here, we have used $\partial_\nn \EE_\gg|_{\GNi} = \gg_i$.
  Subtracting $\Oprod{\nabla \cdot \EE_\gg,p_\vv}$ and $\Oprod{\partial_t (\EE_\gg + \ww_1),\vv}$ 
  on both sides yields
  \begin{equation}\label{eq:wtwo_PDE}
    \Oprod{\partial_t \ww_2,\vv} + \Oprod{\nabla \ww_2,\nabla \vv} + \Oprod{\nabla \cdot \ww_2,p_\vv}
    - \Oprod{p_{\ww_2},\nabla \cdot \vv}
    = \Oprod{\tilde \ff, \vv},
  \end{equation}
  \revD{for all $(\vv,p_\vv) \in \widehat{\V}\times L^2(\Om)$, a.e. $t \in I$,}
  where we have set $\tilde \ff = \ff - \partial_t (\EE_\gg + \ww_1) + \Delta \EE_\gg 
  \in L^2(I\times \Om)^2$.
  \revD{From timewise trace estimate from 
   \cite[Chapter 1, Theorem 3.1]{lions_nonhomogeneous_1972_vol1}, we obtain
  \begin{equation*}
     \|\EE_\gg(0)\|_{H^1(\Om)} 
     \le C \left(\|\EE_\gg\|_{L^2(I;H^2(\Om))} + \|\partial_t \EE_\gg\|_{L^2(I;L^2(\Om))}\right).
  \end{equation*}
  With this, \eqref{eq:regularity_operator_T} and since by construction it holds
  $\nabla \cdot (\EE_\gg(0)+\ww_1(0)) = 0$, we have
     \begin{equation}\label{eq:wtwo_initCond}
      \ww_2(0) = \ww(0) - \EE_\gg(0) - \ww_1(0) = -(\EE_\gg(0)+\ww_1(0)) \in \V,
   \end{equation}
   together with a bound
  \begin{equation*}
     \|\ww_2(0)\|_\V \le C \|\EE_\gg(0)\|_{H^1(\Om)} 
     \le C \left(\|\EE_\gg\|_{L^2(I;H^2(\Om))} + \|\partial_t \EE_\gg\|_{L^2(I;L^2(\Om))}\right).
  \end{equation*}
  Putting \eqref{eq:wtwo_PDE} and \eqref{eq:wtwo_initCond} together, $(\ww_2,p_{\ww_2})$ satisfy
  an instationary Stokes equation with homogeneous boundary data,
  which by \Cref{prop:instat_stokes_solvability} yields
  $\ww_2 \in L^2(I;H^2(\Om)^2\cap \V) \cap H^1(I;L^2(\Om)^2)$, $p_{\ww_2} \in L^2(I;H^1(\Om))$ and the bound
  \begin{equation*}
    \|\partial_t \ww_2\|_{L^2(I;L^2(\Om))}+ \|\ww_2\|_{L^2(I;H^2(\Om))} 
    + \|p_{\ww_2}\|_{L^2(I;H^1(\Om))}
    \le C \left(\|\tilde \ff \|_{L^2(I;L^2(\Om))} 
         + \|\EE_\gg\|_{L^2(I;H^2(\Om))} + \|\partial_t \EE_\gg\|_{L^2(I;L^2(\Om))}\right).
  \end{equation*}
  }
  Inserting the previous estimates of $\EE_\gg, \ww_1$, in order to estimate the right hand side $\tilde \ff$
  yields the result.
\end{proof}
With this regularity result for inhomogeneous mixed boundary conditions established, we can now turn
  towards discussing higher regularity of the adjoint state $\zbar$. 
  As we will not be able to directly deduce the full regularity in one step, let us first recall the following
  technical lemma, allowing us to characterize interpolation spaces, between spaces that will be of interest 
  later.
   \begin{lemma}\label{lemm:boundary_condition_interpolation_space}
      For any $\theta \in [\frac12,1)$ \revB{and $i=1,...,L$}, the following
      \revC{interpolation result holds}
      \begin{equation*}
         \left[L^2(I;\Hhalfoodual),
            H^{\frac{1}{4}}(I;L^2(\GNi)) \cap L^2(I;H^{\frac{1}{2}}_{00}(\GNi))\right]_\theta
         =H^{\frac{\theta}{4}}(I;H^{\frac{\theta}{2} -\frac{1}{2}}(\GNi)) \cap
            L^2(I;H^{\theta-\frac{1}{2}}(\GNi)).
      \end{equation*}
   \end{lemma}
   \begin{proof}
      From \cite[Chapter 1, Proposition 2.1]{lions_nonhomogeneous_1972_vol1} we obtain that 
      $L^2(\GNi) = \left[\Hhalfoodual,H^{\frac{1}{2}}_{00}(\GNi)\right]_{\frac{1}{2}}$.
      As taking derivatives with respect to $t$ and $x$ is commutative,
      by \cite[Chapter 1, Theorem 13.1]{lions_nonhomogeneous_1972_vol1}, we obtain
     \begin{align*}
        \left[L^2(I;\Hhalfoodual),H^{\frac{1}{4}}(I;L^2(\GNi)) 
       \cap L^2(I;H^{\frac{1}{2}}_{00}(\GNi))\right]_\theta \hspace{5cm}\\
       = H^{\frac{\theta}{4}}\left(I;[\Hhalfoodual,L^2(\GNi)]_{1-\theta}\right) 
       \cap L^2\left(I;[\revB{\Hhalfoodual},H^{\frac{1}{2}}_{00}(\GNi)]_\theta\right).
     \end{align*}
     As a result of \cite[Chapter 1, Theorem 12.3]{lions_nonhomogeneous_1972_vol1} we have,
     \revD{due to $H^{-1}(\GNi) = (H^1_0(\GNi))^*$, the identity}
     $\Hhalfoodual = [H^{-1}(\GNi),L^2(\GNi)]_{\frac{1}{2}}$. 
     In combination with the reiteration theorem \eqref{eq:reiteration_theorem} this gives
     \begin{equation*}
      \left[\Hhalfoodual,L^2(\GNi)\right]_{1-\theta} 
      = H^{\frac{\theta}{2} - \frac{1}{2}}(\GNi).
     \end{equation*}
     Similarly using several times the reiteration theorem \eqref{eq:reiteration_theorem}
     we obtain for $\theta \ge \frac{1}{2}$
     \begin{align*}
        \left[\Hhalfoodual,H^{\frac{1}{2}}_{00}(\GNi)\right]_\theta
       & = \left[[H^{-1}(\GNi),L^2(\GNi)]_{\frac{1}{2}},[H^1_0(\GNi),L^2(\GNi)]_{\frac{1}{2}}\right]_\theta\\
       & = \left[\left[H^{-1}(\GNi),\left[H^{-1}(\GNi),H^1_0(\GNi)\right]_{\frac{1}{2}}\right]_{\frac{1}{2}},
       \left[H^1_0(\GNi),\left[H^{-1}(\GNi),H^1_0(\GNi)\right]_{\frac{1}{2}}\right]_{\frac{1}{2}}\right]_\theta \\
       & = \left[\left[H^{-1}(\GNi),H^1_0(\GNi)\right]_{\frac{1}{4}},
          \left[H^{-1}(\GNi),H^1_0(\GNi)\right]_{\frac{3}{4}}\right]_\theta\\
       & = \left[H^{-1}(\GNi),H^1_0(\GNi)\right]_{(1-\theta) \frac{1}{4} + \theta \frac{3}{4}}\\
       & = \left[H^{-1}(\GNi),H^1_0(\GNi)\right]_{\frac{\theta}{2} + \frac{1}{4}}\\
       & = \left[\left[H^{-1}(\GNi),H^1_0(\GNi)\right]_{\frac{1}{2}},H^1_0(\GNi)\right]_{\theta - \frac{1}{2}}\\
       & = \left[L^2(\GNi),H^1_0(\GNi)\right]_{\theta - \frac12} 
       =H^{\theta - \frac{1}{2}}(\GNi),
     \end{align*}
     where in the last step, we have used the characterization of the interpolation space between $L^2$ and 
     $H^1_0$ from \cite[Chapter 1, Theorem 12.6]{lions_nonhomogeneous_1972_vol1}.
     The condition $\theta \in [\frac12,1)$ was needed in the step to the second to last line.
%
   \end{proof}
   The next two technical lemmas gives us regularity results for multiplication of two functions in some 
   Bochner spaces, as well as instationary trace results.
  \begin{lemma}\label{lemm:space_time_multiplication}
    Let $\omega \subset \R$ be open and let $p_1,p_2 \in [1,\infty]$ satisfy $1/p_1 + 1/p_2=1/2$ and let 
    $s_1,s_2,s \in [0,1]$ satisfy $s \le \min\{s_1,s_2\}$ and $s_1 + s_2 > s + 1/2$.
    Then for any $u \in H^{s_1}(I;L^{p_1}(\omega))$ and $w \in H^{s_2}(I;L^{p_2}(\omega))$, there holds 
    $u\cdot w \in H^s(I;L^2(\omega))$ with a bound
    \begin{equation*}
      \|u\cdot w\|_{H^s(I;L^2(\omega))} 
      \le C \|u\|_{H^{s_1}(I;L^{p_1}(\omega))}  \|w\|_{H^{s_2}(I;L^{p_2}(\omega))}.
    \end{equation*}
    Moreover, for any
    $u \in L^{p_1}(I;H^{s_1}(\omega))$ and $w \in L^{p_2}(I;H^{s_2}(\omega))$, there holds
    $u\cdot w \in L^2(I;H^s(\omega))$ with a bound
    \begin{equation*}
      \|u\cdot w\|_{L^2(I;H^s(\omega))} 
      \le C \|u\|_{L^{p_1}(I;H^{s_1}(\omega))}  \|w\|_{L^{p_2}(I;H^{s_2}(\omega))}.
    \end{equation*}
  \end{lemma}
  \begin{proof}
    We only show one of the claims, the other follows by switching the roles of $I$ and $\omega$.
    By Hölder's inequality and the assumptions on $p_1,p_2$ there holds 
    \begin{equation*}
      \|u\cdot v\|_{L^2(\omega)} \le \|u\|_{L^{p_1}(\omega)}\|v\|_{L^{p_2}(\omega)}.
    \end{equation*}
    By assumptions on the regularities of $u,v$, the norms on the right hand sides are functions in 
    $H^{s_i}(I)$, $i=1,2$.
    The result then is a direct consequence of the multiplication theorem in Sobolev spaces 
    \cite[Theorem A.1]{behzadan_multiplication_2021}.
    Note that the second result can be obtained with the same argument and the same conditions on $s_i$, 
    as $\omega \subset \R$ was assumed to be one-dimensional.
  \end{proof}
  \begin{lemma}\label{lemm:space_time_trace}
 The spacial trace operator $\tr(\cdot)$ is a well defined, continuous, linear operator between 
 the following spaces
\begin{align*}
  \VV
  & \to H^{\gamma}(I;H^{\frac12-2\gamma}(\dO)^2),
  &&\text{if } \gamma\in \left[0,\frac14\right),\\
  \W
  & \to H^{\gamma}(I;H^{\frac32-2\gamma}(\dO)^2),
  && \text{if } \gamma \in \left(\frac{1}{4},\frac34\right),\\
  L^2(I;H^{1+\theta}(\Om)^2) \cap H^{\theta}(I;H^{1-\theta}(\Om)^2) & \to
  H^{\gamma}(I;H^{\frac12+\theta-2\gamma}(\dO)^2),
  &&\text{if } 
  \gamma \in [0,\theta) \cap \left(\frac{\theta}{2}-\frac{1}{4},\frac{\theta}{2}+\frac{1}{4}\right),
\end{align*}
where in the last statement, $\theta$ is some fixed parameter, satisfying $\theta \in (0,1)$.
  \end{lemma}
  \begin{proof}
    Let first $\ww \in \VV = L^2(I;\V) \cap H^1(I;\V^*)$.
    Note that for $\gamma=0$, this is just a direct application
    of the trace theorem to $\ww \in L^2(I;\V)$. As $[\V,\V^*]_{\frac{1}{2}} = \Hspace$,
    the intermediate derivative theorem \eqref{eq:intermediate_derivative_theorem} gives
    $\ww \in H^{\frac{1}{2}}(I;\Hspace)$. Thus $\ww \in L^2(I;H^1(\Om)^2)\cap H^{\frac{1}{2}}(I;L^2(\Om)^2)$.
    Applying the intermediate derivative theorem \eqref{eq:intermediate_derivative_theorem} again, yields 
    $\ww \in H^{\frac{\beta}{2}}(I;H^{1-\beta}(\Om)^2)$ for any $\beta \in (0,1)$. 
    If we introduce $\gamma := \frac{\beta}{2} \in (0,\frac12)$, in case $1-2\gamma \in (\frac12,\frac32)$,
    the trace theorem \eqref{eq:trace_general} yields $\tr(\ww) \in H^{\gamma}(I;H^{1/2 - 2\gamma}(\dO)^2)$.
    Thus the result holds for all $\gamma \in [0,\frac12)\cap (-\frac14,\frac14) = [0,\frac14)$.\\
    Let now 
    $\ww \in \W = L^2(I;\V\cap H^2(\Om)^2) \cap H^1(I;\Hspace) \subset L^2(I;H^2(\Om)^2)\cap H^1(I;L^2(\Om)^2)$,
    the intermediate derivative theorem \eqref{eq:intermediate_derivative_theorem} 
    gives $\ww \in H^\gamma(I;H^{2-2\gamma}(\Om)^2)$ for any 
    $\gamma \in (0,1)$. Thus the trace theorem \eqref{eq:trace_general} gives 
    $\tr(\ww)\in H^\gamma(I;H^{\frac32-2\gamma}(\revB{\partial\Om})^2)$ whenever $\gamma \in (\frac14,\frac34)$.\\
    Lastly, let
    $\ww \in L^2(I;H^{1+\theta}(\Om)^2) \cap H^{\theta}(I;H^{1-\theta}(\Om))$,
    for some fixed $\theta \in (0,1)$.
    We can define $\sigma := \frac{\gamma}{\theta}$. It holds $\sigma \in (0,1)$ iff $\gamma \in (0,\theta)$.
    The intermediate derivative theorem \eqref{eq:intermediate_derivative_theorem} yields 
    $\ww \in H^{\sigma\theta}(I;H^{\sigma(1-\theta)+(1-\sigma)(1+\theta)}(\Om)^2)$
    Thus with the definition of $\sigma$, there holds
    $\ww \in H^\gamma(I;H^{1+\theta-2\gamma}(\Om)^2)$.
    Whenever $1+\theta-2\gamma \in (\frac12,\frac32)$, or equivalently
    $\gamma\in(\frac{\theta}{2}-\frac14,\frac{\theta}{2}+\frac14)$, \eqref{eq:trace_general} yields
    $\tr(\ww) \in H^\gamma(I;H^{\frac12+\theta-2\gamma}(\dO)^2)$.
    Combining this condition with $\gamma \in (0,\theta)$ concludes the proof.
  \end{proof}
Finally we are able to show improved regularity results for the optimal variables.
\begin{theorem}\label{thm:improved_regularity}
   Let $\uu_d \in L^6(\revB{I\times \Om})^2$ and let $(\bar q,\bar \uu,\bar \zz)$ satisfy the
   first order optimality conditions of \Cref{thm:adjoint_regularity},
   \revD{and let $ p_{\ubar}$ be the pressure associated with $\ubar$.}
   Then there holds 
   \begin{align*}
      \bar \uu & \in L^{\revD{r}}(I;\V \cap H^2(\Om)^2) \cap W^{1,{\revD{r}}}(I;L^2(\Om)^2) 
               & \text{for any } \revD{r \in [1,\infty)},\\
       p_{\bar \uu} & \in \revD{L^r(I;H^1(\Om))} & \text{for any } r \in [1,\infty),\\
      \bar \zz & \in \W, &\\
      \bar q & \in H^1(I)^L. &
   \end{align*}
\end{theorem}
\begin{proof}
   We first show the regularity of $\bar \uu$ with a bootstrapping argument. 
   \revC{For any $\varepsilon>0$,} it holds
   \begin{equation*}
      \|(\bar \uu \cdot \nabla )\bar \uu\|_{L^2(\Om)} 
      \le \|\bar \uu\|_{L^{\frac{4+2\varepsilon}{\varepsilon}}(\Om)}\|\nabla \bar \uu\|_{L^{(2+\varepsilon)}(\Om)}
      \le \|\nabla \bar \uu\|_{L^{2}(\Om)}\|\nabla \bar \uu\|_{L^{2}(\Om)}^{\frac{2}{2+\varepsilon}}
      \|\nabla \bar \uu\|_{H^1(\Om)}^{1-\frac{2}{2+\varepsilon}}
   \end{equation*}
   where we have used Sobolev imbedding theorem and 
   \revB{the $L^p$ interpolation estimate from \eqref{eq:Lp_interpolation_estimate} applied to $\nabla \ubar$}.
   As $\varepsilon$ can be arbitrarily small, $\bar \uu \in L^2(I;H^2(\Om)^2)$ and 
   \revB{$\nabla \bar \uu \in C(\bar I;L^2(\Om)^{2\times 2})$}, this shows 
   $(\bar \uu\cdot \nabla \bar \uu)\in L^{\revD{r}}(I;L^2(\Om)^2)$ for any $\revD{r \in [1,\infty)}$.
   By \Cref{thm:adjoint_regularity}, we have $\bar q \in \revB{C(\bar I)^L} \hookrightarrow L^{\revD{r}}(I)^L$
   for any 
   $\revD{r \in [1,\infty)}$. By the maximal parabolic regularity result of \Cref{thm:max_par_reg}, applied to 
   \eqref{eq:stateequation}, we immediately get the proposed regularity of $\bar \uu$.
   \revD{Here the case $p=1$ follows from the case $p>1$ and boundedness of $I$. 
     A bootstrapping argument as in the proof of \Cref{thm:NS_H2} yields the regularity of 
   $ p_{\bar \uu}$.}
   The improved regularity of $\qbar$ \revB{will immediately follow} as before from the regularity of $\zbar$,
   \revB{see the proof of \Cref{thm:adjoint_regularity}}.
   \revB{Hence, } let us now turn toward the regularity of the adjoint state.
   With integration by parts, the adjoint equation \eqref{eq:adjoint_L4} is equivalent to:
   Find $\zbar \in \VV$ such that $\zbar(T) = 0$ and for all $\vv \in \V$ and a.e.
  $t \in I$, there holds
\begin{equation}
   \revB{\Opair{-\partial_t \bar \zz, \vv}} + \Oprod{\nabla \bar \zz,\nabla \vv} 
      = \Oprod{|\bar \uu-\uu_d|^2(\bar \uu-\uu_d),\vv} + \Oprod{(\bar \uu\cdot \nabla)\bar \zz,\vv}  
      - \Oprod{(\nabla \bar \uu) \bar \zz,\vv} 
      - \sum_{i=1}^L \Gprodj{\bar \uu\cdot \nn,\bar \zz \cdot \vv}{N,i}.
\end{equation}
Our goal is to apply \Cref{thm:stokes_inhom_neumann_data}, to which end we need to show
\begin{enumerate}[(I)]
   \item $\tilde \ff := |\ubar-\uu_d|^2(\ubar-\uu_d) + (\ubar\cdot \nabla) \zbar 
      - (\nabla \ubar)\zbar \in L^2(I;L^2(\Om)^2)$, and \label{proof:I}
   \item $\gg_i := (\ubar\cdot \nn)\zbar |_{\GNi} \in H^{\frac{1}{4}}(I;L^2(\GNi)^2)
      \cap L^2(I;H^{\frac{1}{2}}_{00}(\GNi)^2)$ for $i=1,...,L$. \label{proof:II}
\end{enumerate}
To show (\ref{proof:I}), first note that the assumption $\uu_d \in \revB{L^6(I\times\Om)^2}$ and
$\ubar \in C(\bar I;\V) \hookrightarrow \revB{L^6(I\times\Om)^2}$ yields 
$|\ubar-\uu_d|^2(\ubar-\uu_d) \in L^2(I\times \Om)^2$.
Further, following the proof of \cite[Theorem 3.10]{vexler_error_2024}, with the available regularity of 
$\bar \uu$, we obtain $\ubar \in L^\infty(I \times \Om)^2$.
Moreover, using the embeddings from \revB{\eqref{eq:L4L4_imbeddings}, 
the regularities $\ubar \in \W, \zbar \in \VV$} yield
\begin{align*}
   \nabla \ubar \in \revB{L^4(I\times\Om)^2} \land \zbar \in \revB{L^4(I\times\Om)^2}
   & \Rightarrow (\nabla \ubar) \zbar \in \revB{L^2(I\times\Om)^2}\\
\ubar \in L^\infty(I\times \Om)^2 \land \nabla \zbar \in \revB{L^2(I\times\Om)^2}
   & \Rightarrow (\ubar\cdot \nabla) \zbar \in \revB{L^2(I\times\Om)^2}.
\end{align*}
This shows (\ref{proof:I}). Let us turn towards (\ref{proof:II}).
   It turns out, that the proposed regularity of $\gg_i$ does not immediately follow from 
   $\zbar \in L^2(I;\V) \cap H^1(I;\V^*)$, which we have shown in \Cref{thm:adjoint_solvability},
   but we rather have to iterate a bootstrapping argument, first showing an intermediate result for $\gg_i$.
   From \Cref{prop:instat_stokes_solvability} we obtain that if 
      $\gg_i \in L^2(I;(H^{\frac{1}{2}}_{00}(\GNi))^*)$, there holds 
   \begin{equation}\label{eq:adjoint_regularity_low}
   \|\zbar\|_{L^2(I;\V)} + \|\zbar\|_{H^1(I;\V^*)} 
   \le C \left(\|\tilde \ff\|_{L^2(I;L^2(\Om))} 
   +  \sum_{i=1}^L \|\gg_i\|_{L^2(I;(H^{\frac{1}{2}}_{00}(\GNi))^*)}\right).
   \end{equation}
   In fact, as for any $\vv \in \V$, there holds $\tr(\vv) \in H^{\frac{1}{2}}(\partial \Om)$ and
   $\tr(\vv)|_{\Gamma_D} = \oo$. This yields $\vv|_{\GNi} \in H^{\frac{1}{2}}_{00}(\GNi)$ with a bound
   $\|\vv\|_{H^{\frac{1}{2}}_{00}(\GNi)} \le C\|\vv\|_{H^{\frac{1}{2}}(\partial \Om)}$.
   Overall, we thus have
   \begin{equation*}
      \Gprodj{\gg_i,\vv}{N,i}
      \le \|\gg_i\|_{(H^{\frac{1}{2}}_{00}(\GNi))^*}\|\tr(\vv)\|_{H^{\frac{1}{2}}_{00}(\GNi)}
      \le C\|\gg_i\|_{(H^{\frac{1}{2}}_{00}(\GNi))^*}\|\tr(\vv)\|_{H^{\frac{1}{2}}(\partial \Om)}
      \le C\|\gg_i\|_{(H^{\frac{1}{2}}_{00}(\GNi))^*}\|\vv\|_{\V},
   \end{equation*}
   which shows $\Gprodj{\gg_i,\cdot}{N,i} \in L^2(I;\V^*)$, if 
   $\gg_i \in L^2(I;(H^{\frac{1}{2}}_{00}(\GNi))^*)$.
For more regular $\gg_i$, we have shown in \Cref{thm:stokes_inhom_neumann_data} the following
   \begin{align}
     \|\zbar\|_{L^2(I;\V \cap H^2(\Om)^2)} + \|\zbar\|_{H^1(I;L^2(\Om))} 
       & \le C \left(\|\tilde \ff\|_{L^2(I;L^2(\Om))} 
       + \sum_{i=1}^L \|\gg_i\|_{H^{\frac{1}{4}}(I;L^2(\GNi)) \cap L^2(I;H^{\frac{1}{2}}_{00}(\GNi))}\right).
     \label{eq:adjoint_regularity_high}
   \end{align}
   \revB{By interpolating between the regularities $\zbar \in L^2(I;H^1(\Om)^2)$ and 
   $\zbar \in L^2(I;H^2(\Om)^2) \cap H^1(I;L^2(\Om))$ 
   obtained in \eqref{eq:adjoint_regularity_low} and \eqref{eq:adjoint_regularity_high},
   with the help of \eqref{eq:space_time_intersection_interpolation} and the characterization of 
   the interpolation space for the $\gg_i$ from \Cref{lemm:boundary_condition_interpolation_space},
   we obtain}
\begin{equation}\label{eq:interpolation}
  \|\zbar\|_{L^2(I;H^{1+\theta}(\Om))} + \|\zbar\|_{H^\theta(I;H^{1-\theta}(\Om))} 
  \le C \left(\|\tilde \ff \|_{L^2(I;L^2(\Om))} + 
  \sum_i \|\gg_i\|_{H^{\frac{\theta}{4}}(I;H^{\frac{\theta}{2} -\frac{1}{2}}(\GNi)) \cap L^2(I;H^{\theta-\frac{1}{2}}(\GNi))}\right),
\end{equation}
whenever $\theta \in [\frac12,1)$ and $\gg_i$ has the appropriate regularity. 
In a first step, we will show, that this holds for some $\theta$ arbitrarily close to $1$.
As $\zbar \in \VV$, for sufficiently small $\varepsilon > 0$ and with 
  $\gamma_1=\frac14-\frac\varepsilon2$,
the first result of \Cref{lemm:space_time_trace} implies
\begin{equation}\label{eq:ztrace}
  \zbar|_{\GNi} \in L^2(I;H^{\frac{1}{2}}(\GNi)^2) \cap 
  H^{\frac{1}{4} - \frac{\varepsilon}{2}}(I;H^{\varepsilon}(\GNi)^2).
\end{equation}
Moreover, $\ubar \in \W$, thus 
the second result of \Cref{lemm:space_time_trace} yields with the choice $\gamma_2=\frac{1+\varepsilon}{2}$
\begin{equation}\label{eq:utrace}
  \ubar|_{\GNi} \in H^{\frac{1+\varepsilon}{2}}(I;H^{\frac{1}{2} -\varepsilon}(\GNi)^2).
\end{equation}
As $H^\varepsilon(\GNi) \hookrightarrow L^{\frac{2}{1-2\varepsilon}}(\GNi)$ and 
$H^{\frac{1}{2} -\varepsilon}(\GNi) \hookrightarrow L^{\frac{2}{2 \varepsilon}}(\GNi)$ where 
$\frac{1-2\varepsilon}{2}+\frac{2\varepsilon}{2} = \frac{1}{2}$, an application of 
\Cref{lemm:space_time_multiplication}
together with the time regularities stated in \eqref{eq:ztrace},\eqref{eq:utrace}
yields for arbitrarily small $\sigma>0$, that
\begin{equation}\label{eq:uz_regularity_time}
  (\ubar\cdot \nn) \zbar|_{\GNi} \in H^{\frac{1}{4} -\sigma}(I;L^2(\GNi)^2).
\end{equation}
From \eqref{eq:utrace}, we further have
$\ubar|_{\GNi} \in L^\infty(I;H^{\frac{1}{2} -\varepsilon}(\GNi)^2)$,
which together with \eqref{eq:ztrace} and \Cref{lemm:space_time_multiplication} yields
for arbitrarily small $\alpha > 0$, that 
\begin{equation}\label{eq:uz_regularity_space}
  (\ubar\cdot \nn) \zbar|_{\GNi} \in L^2(I;H^{\frac{1}{2} - \varepsilon - \alpha}(\GNi)^2).
\end{equation}
To write the regularities stated in \eqref{eq:uz_regularity_time},\eqref{eq:uz_regularity_space} in terms 
  of interpolation spaces, we search for $\varepsilon,\sigma, \alpha > 0$, $\theta \in [0,1]$ such that 
\begin{equation*}
  \left[\theta - \frac{1}{2} = \frac{1}{2} - \varepsilon - \alpha \land \frac{\theta}{4} = \frac{1}{4} - \sigma\right]
  \Leftrightarrow
  \left[\theta = 1- \varepsilon - \alpha \land \theta = 1 - 4\sigma\right].
\end{equation*}
Thus, for arbitrarily small $\sigma > 0$, after choosing $\varepsilon + \alpha = 4\sigma$ and 
$\theta = 1 - 4\sigma$, we 
have just shown in \eqref{eq:uz_regularity_time},\eqref{eq:uz_regularity_space}, that 
  \begin{equation*}
  \|(\ubar \cdot \nn)\zbar\|_{H^{\frac\theta4}(I;H^{\frac\theta2-\frac12}(\GNi))\cap
  L^2(I;H^{\theta-\frac12}(\GNi))} < \infty.
  \end{equation*}
  Using \eqref{eq:interpolation}, this shows, that 
  $\zbar \in H^{\theta}(I;H^{1-\theta}(\Om)^2)\cap L^2(I;H^{1+\theta}(\Om)^2)$.
  We now proceed from this intermediate regularity of $\zbar$ to the full 
  $L^2(I;\V\cap H^2(\Om)^2)\cap H^1(I;L^2(\Om)^2)$.
  For $\sigma >0$ close to $0$, it holds 
  $\gamma_3 = \frac14 \in [0,\theta)\cap (\frac{\theta}{2}-\frac14,\frac{\theta}{2}+\frac14)$.
  Thus an application of the third result in \Cref{lemm:space_time_trace} to $\zbar$ with $\gamma_3$ yields
  \begin{equation*}
     \zbar|_{\GNi} \in H^{\frac14}(I;H^{1-4\sigma}(\GNi)^2) \hookrightarrow H^{\frac14}(I;L^\infty(\GNi)^2).
  \end{equation*}
  This, together with \eqref{eq:utrace} and \Cref{lemm:space_time_multiplication} yields
   \begin{equation*}
     (\ubar\cdot \nn) \zbar|_{\GNi} \in H^{\frac{1}{4}}(I;L^2(\GNi)^2).
   \end{equation*}
   Lastly, another application of \Cref{lemm:space_time_trace} to $\zbar$ with 
   $\gamma_4 = \frac{1}{2} - 2\sigma$ yields
  \begin{equation}\label{eq:ztrace_almost_optimal}
  \zbar|_{\GNi} \in H^{\frac12 - 2 \sigma}(I;H^{\frac12}(\GNi)^2).
  \end{equation}
  By the same argument, with which we obtained \eqref{eq:utrace}, we can also show that 
  $\ubar|_{\GNi} \in H^{\frac{1-\delta}{2}}(I;H^{\frac{1}{2} +\delta}(\GNi)^2)$ for small $\delta > 0$.
  This regularity of $\ubar$, combined with \eqref{eq:ztrace_almost_optimal}, Sobolev imbedding in time, and 
  \Cref{lemm:space_time_multiplication} yields
\begin{equation*}
  (\ubar \cdot \nn) \zbar |_{\GNi} \in L^2(I;H^{\frac{1}{2}}(\GNi)^2).
\end{equation*}
Let us finally comment on the compatibility condition $H^{\frac{1}{2}}_{00}(\GNi)$.
In general we have shown
$\tr((\ubar \cdot \nn) \zbar) \in L^2(I;H^{\frac{1}{2}}(\partial \Om)^2)$.
As $\tr( \ubar)|_{\Gamma_{D}} = \tr( \zbar)|_{\Gamma_{D}}=\oo$, this shows 
$( \ubar \cdot \nn) \zbar |_{\Gamma_{N,i}} \in L^2(I;H^{\frac{1}{2}}_{00}(\GNi)^2)$,
which concludes the proof.
\end{proof}

Note that while this improved regularity of the adjoint state is an interesting result 
in itself for the continuous optimal control problem, 
it is essential in deriving error estimates for the discretized optimal control problem.
While a derivation of error estimates is beyond the scope of this work, the improved 
regularity of the adjoint state, can be viewed as one important step towards this goal.
The second order conditions shown in the next section are another main ingredient.

\subsection{Second order conditions}
  Following \cite{bonnans_optimal_1999}, in order to discuss second order optimality conditions, we define the 
  critical cone, as the subset of tangent cone to $\Qad$, in which no first order increase w.r.t the objective 
  function $j$ occurs. Thus let us define 
  \begin{align*}
    T_{\Qad}(q) &:= \{\dq \in \revA{Q} : (\dq_i (t) \le 0 \text{ if } q_i(t) = q_{b,i} \text{ a.e. in } I )
        \land
     (\dq_i(t) \ge 0 \text{ if } q_i(t) = q_{a,i} \text{ a.e. in } I) \ i=1,...,L \},\\
        \mathcal C(q) &:= \{\dq \in T_{\Qad}(q) : j'(q)(\dq) \le 0\}.
  \end{align*}
  If the first order necessary optimality conditions \eqref{eq:first_order_opt_cond} hold,
  we can characterize the critical cone \revA{$\mathcal C(\qbar)$} in terms of pointwise conditions.
  \revB{ Due to \eqref{eq:bang_bang_structure}, if $\qbar$ satisfies the first order optimality conditions,
     and $\dq \in T_{\Qad}(\qbar)$,
     \begin{align*}
        \text{for a.e. } t \in I^i_-:=  \{s \in I: \alpha \bar q_i(s) \revB{-} \Gprodj{\zbar(s),\nn}{N,i} < 0\}
        \text{ there holds } \qbar_i(t) = q_{b,i} \text{ and thus } \dq_i(t) \le 0,\\
        \text{for a.e. } t \in I^i_+ := \{s \in I: \alpha \qbar_i(s) \revB{-} \Gprodj{\zbar(s),\nn}{N,i} > 0\}
        \text{ there holds } \qbar_i(t) = q_{a,i} \text{ and thus } \dq_i(t) \ge 0.
     \end{align*}
     If any of the above inequalities for $\dq_i$ were to hold 
     strictly on a subset of $I^i_{\pm}$ of positive measure, this would imply 
     \begin{equation*}
         \int_I (\alpha \bar q_i \revB{-} \Gprodj{\zbar,\nn}{N,i})\dq_i \dt > 0,
     \end{equation*}
     and thus $j'(\qbar)(\dq) > 0$.
  }
  Hence for $\qbar$ satisfying the first order necessary optimality
  conditions, and $\dq \in T_{\Qad}(\qbar)$, the following conditions are equivalent:
  \begin{equation*}
     j'(\qbar)(\dq) \le 0 \Leftrightarrow \text{for a.e. $t \in I$, with } 
     \alpha \bar q_i \revB{-} \Gprodj{\zbar,\nn}{N,i} \neq 0, \text{ it holds } \dq_i = 0.
  \end{equation*}
  We just proved the implication from left to right, the converse direction is trivial.
  Thus for any $\qbar \in \Qad$ satisfying the first order necessary optimality conditions, the 
  critical cone has the following characterization
   \begin{equation*}
      \mathcal C(\qbar) = \left\{\dq \in Q \text{ s.t.}
      \revB{\text{ for }i=1,...,L :}
      \begin{cases}{}
         \dq_i \ge 0 & \text{ if } \qbar_i(t) = q_{a,i},\\
         \dq_i \le 0 & \text{ if } \qbar_i(t) = q_{b,i},\\
         \dq_i = 0  & \text{ if } \alpha \qbar_i(t) \revB{-} \Gprodj{\zbar(t),\nn}{N,i} \neq 0,
      \end{cases}
   \right\}.
   \end{equation*}
   \revA{With this characterization of the critical cone, we can now turn towards second order conditions.}
   Using the formulas and notations of \Cref{thm:state_derivatives}
for derivatives of $S$, we can express the derivatives of $j$ as follows.
\begin{align}
j'(q)(\dq) &= \partial_q J(q,\uu) (\dq) + \partial_u J(q,\uu) (\du)\label{eq:gradient}\\
  j''(q)(\dq,\rhq) &= \partial_{qq} J(q,\uu)( \dq, \rhq) + \partial_{uq} J(q,\uu) (\dq,\rhu)
  \label{eq:hessian}\\
                   &\quad +\partial_{qu} J(q,\uu)(\du,\rhq) 
                   + \partial_{uu} J(q,\uu)(\du,\rhu) + \partial_u J(q,\uu)(\su) \notag
\end{align}
In order to show the quadratic growth condition, we need to verify several assumptions on $j$, 
we summarize our findings in the following Lemma.
\begin{lemma}\label{lemm:objective}
   Let $q_n \to q$ and $\dq_n \rightharpoonup \dq$ in $\revA{Q}$, such that $q \in \mathcal Q^S$.
   Let further $\uu = S(q)$, \revB{and $\uu_n = S(q_n)$ which are well defined
      for large enough $n$ due to the openness of $\mathcal Q^S$. Consider the first and second 
      derivatives }
   $\du_n = S'(q)(\dq_n)$ and $\du = S'(q)(\dq)$ as well as 
   $\delta^2 \uu_n = S''(q)(\dq_n,\dq_n)$ and $\delta^2 \uu = S''(q)(\dq,\dq)$.
   Then there exist subsequences $\{q_n\}, \{\dq_n\}$ denoted by the same index, such that 
    \begin{align}
      \du_n \to  \du \text{ in } L^2(I;\V) \cap L^\infty(I;\Hspace)\label{eq:grad_strong_convergence}\\
      \delta^2 \uu_n \to  \delta^2 \uu \text{ in } L^2(I;\V) \cap L^\infty(I;\Hspace) \label{eq:hess_strong_convergence}
    \end{align}
    and there hold
    \begin{align}
      j'(q)(\dq) &=  \lim_{n\to \infty} j'(q_n)(\dq_n) \label{eq:convergence_dir_derivative} \\
      j''(q)(\dq,\dq) &\le \liminf_{n\to \infty} j''(q_n)(\dq_n,\dq_n). \label{eq:lower_semicont_hessian}
    \end{align}
    Moreover, \revB{if} $\dq = 0$, there exists $\gamma > 0$ such that
    \begin{equation*}
       \liminf_{n\to\infty} \gamma \|\dq_n\|_{L^2(I;L^2(\Omega)}^2 \le \liminf_{n\to\infty} j''(q_n)(\dq_n,\dq_n).
    \end{equation*}
\end{lemma}
\begin{proof}
The claims \eqref{eq:grad_strong_convergence} and \eqref{eq:hess_strong_convergence} are consequences of
the compact embedding of \eqref{eq:compact_imbedding}.
To show \eqref{eq:convergence_dir_derivative}, we insert $\pm j'(q)(\dq_n)$, apply triangle inequality 
and obtain
  \begin{equation*}
    |j'(q)(\dq) - j'(q_n)(\dq_n)| \le |j'(q)(\dq - \dq_n)| + |j'(q)(\dq_n) - j'(q_n)(\dq_n)|,
  \end{equation*}
  where the first term converges due to the weak convergence of $\dq_n$ and the second due to 
  the strong convergence of $q_n$ and the fact, that the derivative is continuous in $q$.
  In similar fashion, we can show \eqref{eq:lower_semicont_hessian}. 
  More specifically, \revA{if we insert $q_n$, $\dq_n$ and $\rhq_n = \dq_n$ into the individual terms of the
     representation of $j''(q_n)(\dq_n,\dq_n)$ obtained from \eqref{eq:hessian}, we get}
  \begin{align*}
    \lim_{n\to \infty}
    \partial_{uq} J(q_n,\uu_n) (\dq_n,\du_n)
    +\partial_{qu} J(q_n,\uu_n)(\du_n,\dq_n) 
    + \partial_{uu} J(q_n,\uu_n)(\du_n,\du_n)
    + \partial_u J(q_n,\uu_n)(\delta^2 \uu_n)\\
    = \partial_{uq} J(q,\uu) (\dq,\du)
    +\partial_{qu} J(q,\uu)(\du,\dq) 
    + \partial_{uu} J(q,\uu)(\du,\du)
    + \partial_u J(q,\uu)(\delta^2 \uu).
  \end{align*}
  Hence, for all
  but the first term on the right hand side of \eqref{eq:hessian}, we obtain convergence to the 
  limiting term, as only one weakly converging sequence is present.
  To show \eqref{eq:lower_semicont_hessian}, it thus suffices to show
  \begin{equation*}
    \partial_{qq}J(q,\uu)(\dq,\dq) \le \liminf_{n\to \infty} \partial_{qq} J(q_n,\uu_n)(\dq_n,\dq_n).
  \end{equation*}
  Here we cannot show equality due to the presence of two weakly convergent sequences.
  By writing out the corresponding derivatives of the objective functions, this is nothing else than
  \begin{equation*}
     \alpha \|\dq\|_{L^2(I)}^2 = \alpha \Iprod{\dq,\dq} 
     \le \liminf_{n\to \infty} \alpha \Iprod{\dq_n,\dq_n} = \liminf_{n\to\infty} \alpha \|\dq_n\|_{L^2(I)}^2,
  \end{equation*}
  which holds true due to the weak lower semicontinuity of the norm. Thus, we have shown
  \eqref{eq:lower_semicont_hessian}.
  In the same fashion, if 
  $\dq = 0$ then all but the first term on the right hand side of \eqref{eq:hessian} converge to $0$.
  What remains is 
  \begin{equation*}
    \liminf_{n\to\infty} j''(q_n)(\dq_n,\dq_n) = \liminf_{n\to\infty} \alpha \ltwoltwonorm{\dq_n}^2,
  \end{equation*}
  i.e., we can choose $\gamma = \alpha$. This shows the last claim of the lemma.
\end{proof}
With these technical properties fulfilled, we are able to show the following

\begin{theorem}
   Let $\bar q \in \Qad$ be a local solution to \eqref{eq:reduced_problem}, then 
   it satisfies \revB{$j''(\bar q)(\dq,\dq) \ge 0$ for all $\dq \in \mathcal C(\bar q)$.}
   Furthermore, if $\bar q \in \Qad$ satisfies
   \begin{align*}
      \revB{
      j'(\bar q)(\dq) \ge 0 \quad \textforall \dq \in T_{\Qad}(\qbar) \quad \text{and} \quad
      j''(\bar q)(\dq,\dq) > 0 \quad \textforall \dq \in \mathcal C (\qbar) \backslash \{0\},
   }
   \end{align*}
   then $\bar q$ is a local solution to \eqref{eq:reduced_problem}, more specifically,
   there exist $\epsilon,\delta > 0$ such that 
   \begin{equation*}
      j(q) \ge j(\qbar) + \delta \|q- \qbar\|_{L^2(I)}^2 
      \qquad \revB{\text{for all }} q \in \revB{\Qbox} \cap B_\epsilon (\bar q).
   \end{equation*}
\end{theorem}
\begin{proof}
   This is a direct consequence of the technical results shown in \Cref{lemm:objective} and
      the general second order optimality conditions shown in \cite[Theorem 2.3]{casas_general_2012}.
      Due to a more accessible formulation of the assumptions,
      we have chosen to work with its generalization \cite[Theorem 3.2]{hoppe_optimal_2022}, 
      which is also applicable to problems with state constraints. 
      As we have continous differentiability of $S$ as a function on $\mathcal{Q}^S$,
      we can avoid the two-norm-discrepancy.
\end{proof}

%% file: numericalResults_pt2.tex
In this concluding section, we wish to present some numerical results obtained with our approach.
Let us first discuss the solution of only the state equation for some given data.
  We will see, that the issue of non-existence of solutions is not a purely theoretical phenomenon,
  but actually happens in fairly simple domains.
  In the construction of our examples, we rely on the results presented in 
  \cite{lanzendorfer_multiple_2020}, which focussed on stationary problems, but contained some instationary problems as well.
  They considered the equations on a ring segment with \eqref{eq:DN} conditions specified on the 
  inner and outer radius. As we wish to use a geometry fitting \Cref{ass:domain} and would like to highlight,
  that the blowup of solutions is not related to this special domain, we construct the following domain:
  \begin{equation*}
    \Omega = \{\xx \in \R^2: \xx_1 \in [0,L] \land \xx_2 \in [-\phi(\xx_1),\phi(\xx_1)]\},
  \end{equation*}
  where $\phi\colon \R\to\R$ is the following cubic function depending on the three 
  parameters $L,r,R>0$
  \begin{equation*}
    \phi(s) = \dfrac{2(r-R)}{L^3}s^3 + \dfrac{3(R-r)}{L^2}s^2 + r.
  \end{equation*}
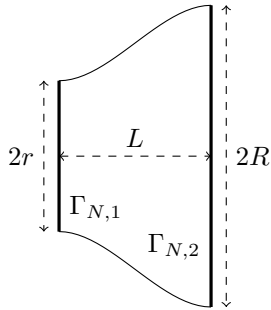
\begin{figure}[H]
  \centering
  \begin{tikzpicture}
    \def\Lo{2}
    \def\Lu{2}
    \def\Ro{2}
    \def\ro{1}
    \def\Ru{2}
    \def\ru{1}
    \def\Ao{3}
    \def\Au{4}
    \draw[domain=0:\Lo] plot(\x,{2*(\ro-\Ro)/(\Lo^3)*\x*\x*\x + 3*(\Ro-\ro)/(\Lo^2)*\x*\x + \ro});
    \draw[domain=\Lu:0] plot(\x,{-(2*(\ru-\Ru)/(\Lu^3)*\x*\x*\x+ 3*(\Ru-\ru)/(\Lu^2)*\x*\x + \ru)});
    \draw (0,\ro) -- (0,-\ru);
    \draw (\Lo,\Ro) -- (\Lu,-\Ru);
    \draw[<->,dashed] (-.2,\ro) -- (-.2,0) node[left] {$2r$} -- (-.2,-\ru);
    \draw[<->,dashed] (\Lo+.2,\Ro) -- (\Lo+.2,0) node[right] {$2R$} -- (\Lo+.2,-\Ru);
    \draw[<->,dashed] (0,0) -- (\Lo/2,0) node[above] {$L$} -- (\Lo,0);
    \draw[line width=1.2pt] (0,\ro)--(0,-\ru) node[above right] {$\Gamma_{N,1}$};
    \draw[line width=1.2pt] (\Lo,\Ro)--(\Lo,-\Ru) node[above=.8, left] {$\Gamma_{N,2}$};
  \end{tikzpicture}
  \caption{Domain $\Omega$ depending on the parameters $r,R,L$.}
\end{figure}
It is straightforward to verify, that the function 
\begin{equation*}
  \ww = \begin{pmatrix}
    \ww_1 \\ \ww_2
  \end{pmatrix}
  = \begin{pmatrix}
    \phi^{-1}(\xx_1)\\ - \xx_2 \cdot \phi'(\xx_1)\cdot \phi(\xx_1)^{-2}
  \end{pmatrix}
\end{equation*}
satisfies $\ww \in \Hspace$. In order to better visualize the results obtained, we define the 
following flowrate functional 
\begin{equation*}
  Q(\vv) = -\int_{\Gamma_{N,1}} \vv \cdot \nn \, ds,
\end{equation*}
measuring the amount of flow through $\Gamma_{N,1}$, and thus for divergence free functions also through
$\Gamma_{N,2}$. Whenever $\vv$ is time dependent, the flowrate is also time dependent.
\subsection{Blowup in finite time}\label{subsect:blowup_state}
Our first subgoal is to identify, for which data blowup happens.
Empirically, we determined that if either
$\uu_0$ is too large (e.g., $\uu_0 = 15 \ww$) or the difference $q_1 -q_2$ is too large (e.g., $50$)
then blowup happens sooner or later. We report our findings in \Cref{fig:blowup}.
Heuristically, this blowup can be explained as follows. Due to the difference in crossection and the 
divergence constraint, the flow can either slow down and spread evenly over the wider crossection, or 
more fluid can be drawn in at the sides of the left boundary to ensure conservation of mass.
As the latter phenomenon amplifies itself,
at some point the solution blows up, which results in the timestepping scheme to stop prematurely.
%
%
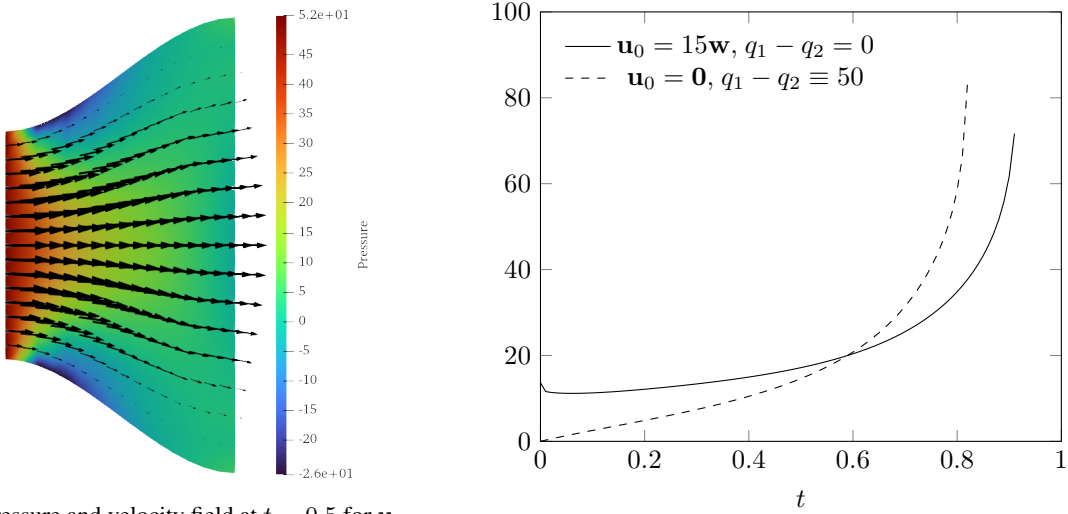
\begin{figure}[H]
  \begin{subfigure}{.38\linewidth}
    \centering
    \includegraphics[width=.8\linewidth]{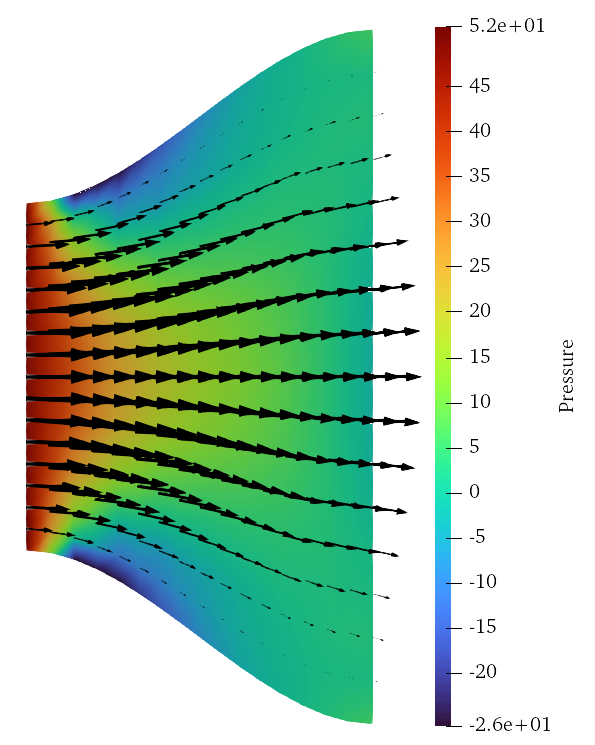}
    \caption{Pressure and velocity field at $t=0.5$ for $\uu_0 = \oo$ and $q_1 - q_2 \equiv 50$}
  \end{subfigure}
  \hspace{5mm}
  \begin{subfigure}{.5\linewidth}
  \begin{tikzpicture}[]
    \begin{axis}[legend pos=north west, legend style={draw=none},
                  xmin=0, xmax=1, xlabel={$t$},ymin=0,ymax=100]
      \addplot[black] table[x expr=\thisrow{"Time"}/100,y="avg(U001)",col sep=comma]{plots/U0_blowup.csv};
      \addplot[black,dashed] table[x expr=\thisrow{"Time"}/100,y="avg(U001)",col sep=comma]{plots/Pdiff_blowup.csv};
      \legend{{$\uu_0 = 15 \ww {, \ } q_1 - q_2 = 0$},{$\uu_0 = \oo {, \ } q_1 - q_2 \equiv 50$}}
    \end{axis}
  \end{tikzpicture}
  \caption{Flowrate $Q(\uu)$ for two different choices of data.}
  \label{subfig:blowup}
\end{subfigure}
\caption{Two reasons for blowup of numerical solutions: large initial data or large boundary data.}
\label{fig:blowup}
\end{figure}
\subsection{Solving the optimal control problem}
We now want to highlight, that our proposed optimal control problem is able to deal with 
this problem of blowup of solutions quite well. To this end, we slightly modify our problem formulation 
to use the objective function 
\begin{equation*}
  \tilde J(q,\uu) = \frac14 \|\uu-\uu_d\|_{L^4(I\times \Om)}^4 + \frac\alpha2 \|q - q_d\|_{L^2(I)}^2,
\end{equation*}
where $q_d \in L^2(I)^L$ is some desired control. 
\revD{Note that due to this change of structure, in the optimality conditions, only the projection
  formula \eqref{eq:qbar_projection_formula} has to be adapted, and the regularity of $\qbar$ does
  not solely depend on the regularity of $\zbar$, but also on $q_d$.
  If $q_d \in (C(\bar I) \cap H^{\frac12}(I))^L$,
  the same regularities of $\ubar$ and $\zbar$ as shown in \Cref{thm:improved_regularity} hold.
}
We then set $\uu_0 = 15\ww$ and $q_{d,1} = 50, q_{d,2} = 0$.
Further, we choose $q_a = (-\infty,-\infty)^T$ and $q_b = (+\infty,+\infty)^T$,
\revC{and set $\uu_d$ to a multiple of $\ww$.}
With these choices, we would like to achieve an optimal flow $\ubar$, which is still well defined for all 
$t \in I$, but is very close to the blowup regime.
Note that due to the choice of initial condition, and the results presented in the previous subsection,
we cannot initialize our control variable by $q_1 = q_2 \equiv 0$.
We thus choose $q_1^0 \equiv 0$ and $q_2^0 \equiv 50$, i.e., we apply an opposite pressure gradient, 
slowing down the instable initial condition.
The resulting optimal controls for this example for different choices of $\uu_d$ are reported in 
\Cref{fig:opt_control_result}.
Notice how the optimal controls mainly have three phases: at first, the initial flow is either 
abruptly stopped by applying an opposing pressure drop (\Cref{subfig:ud10})
or abruptly sped up (\Cref{subfig:ud20}) to match the target velocity $\uu_d$.
Afterwards, the flow is kept stably close to $\uu_d$, where an opposing pressure drop needs to be applied
in \Cref{subfig:ud20}. 
Only shortly before the end of the time interval, $\qbar_i$ tend towards $\qbar_{d,i}$, as a blowup of 
$\ubar$ after the end of $I$ is not prevented by our formulation.
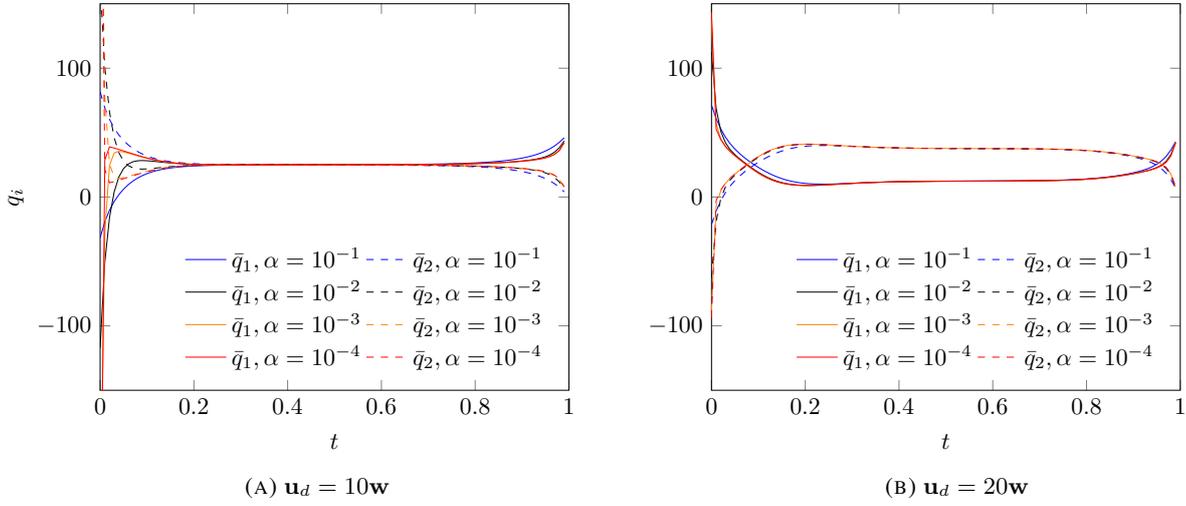
\begin{figure}[H]
  \begin{subfigure}{.49\linewidth}
  \begin{tikzpicture}[scale=.90]
    \begin{axis}[legend pos=south east,legend style={draw=none}, legend columns=2,
                  xmin=0, xmax=1, ylabel={$q_i$}, xlabel={$t$}, ymin=-150, ymax=150]
      \addplot[blue] table[x expr=\coordindex/100,y expr=\thisrow{Q0},col sep=space]{plots/optControl_ud10_alph_e-1.csv};
      \addplot[blue,dashed] table[x expr=\coordindex/100,y expr=\thisrow{Q1},col sep=space]{plots/optControl_ud10_alph_e-1.csv};
      \addplot[black] table[x expr=\coordindex/100,y expr=\thisrow{Q0},col sep=space]{plots/optControl_ud10_alph_e-2.csv};
      \addplot[black,dashed] table[x expr=\coordindex/100,y expr=\thisrow{Q1},col sep=space]{plots/optControl_ud10_alph_e-2.csv};
      \addplot[orange] table[x expr=\coordindex/100,y expr=\thisrow{Q0},col sep=space]{plots/optControl_ud10_alph_e-3.csv};
      \addplot[orange,dashed] table[x expr=\coordindex/100,y expr=\thisrow{Q1},col sep=space]{plots/optControl_ud10_alph_e-3.csv};
      \addplot[red] table[x expr=\coordindex/100,y expr=\thisrow{Q0},col sep=space]{plots/optControl_ud10.csv};
      \addplot[red,dashed] table[x expr=\coordindex/100,y expr=\thisrow{Q1},col sep=space]{plots/optControl_ud10.csv};
      \legend{{$\qbar_1,\alpha=10^{-1}$},{$\qbar_2,\alpha=10^{-1}$},{$\qbar_1,\alpha=10^{-2}$},{$\qbar_2,\alpha=10^{-2}$},{$\qbar_1,\alpha=10^{-3}$},{$\qbar_2,\alpha=10^{-3}$},{$\qbar_1,\alpha=10^{-4}$},{$\qbar_2,\alpha=10^{-4}$}}
    \end{axis}
  \end{tikzpicture}
    \caption{$\uu_d = 10\ww$}
    \label{subfig:ud10}
  \end{subfigure}
  \begin{subfigure}{.49\linewidth}
  \begin{tikzpicture}[scale=.90]
    \begin{axis}[legend pos=south east, legend style={draw=none},legend columns=2,
                  xmin=0, xmax=1, ymin=-150, ymax=150, xlabel={$t$}]
      \addplot[blue] table[x expr=\coordindex/100,y expr=\thisrow{Q0},col sep=space]{plots/opt_Control_ud20_alph_e-1.csv};
      \addplot[blue,dashed] table[x expr=\coordindex/100,y expr=\thisrow{Q1},col sep=space]{plots/opt_Control_ud20_alph_e-1.csv};
      \addplot[black] table[x expr=\coordindex/100,y expr=\thisrow{Q0},col sep=space]{plots/optControl_ud20_alph_e-2.csv};
      \addplot[black,dashed] table[x expr=\coordindex/100,y expr=\thisrow{Q1},col sep=space]{plots/optControl_ud20_alph_e-2.csv};
      \addplot[orange] table[x expr=\coordindex/100,y expr=\thisrow{Q0},col sep=space]{plots/optControl_ud20_alph_e-3.csv};
      \addplot[orange,dashed] table[x expr=\coordindex/100,y expr=\thisrow{Q1},col sep=space]{plots/optControl_ud20_alph_e-3.csv};
      \addplot[red] table[x expr=\coordindex/100,y expr=\thisrow{Q0},col sep=space]{plots/optControl_ud20.csv};
      \addplot[red,dashed] table[x expr=\coordindex/100,y expr=\thisrow{Q1},col sep=space]{plots/optControl_ud20.csv};
      \legend{{$\qbar_1,\alpha=10^{-1}$},{$\qbar_2,\alpha=10^{-1}$},{$\qbar_1,\alpha=10^{-2}$},{$\qbar_2,\alpha=10^{-2}$},{$\qbar_1,\alpha=10^{-3}$},{$\qbar_2,\alpha=10^{-3}$},{$\qbar_1,\alpha=10^{-4}$},{$\qbar_2,\alpha=10^{-4}$}}
    \end{axis}
  \end{tikzpicture}
  \caption{$\uu_d = 20\ww$}
  \label{subfig:ud20}
\end{subfigure}
\caption{Optimal Controls for different choices of $\uu_d$ and $\alpha$. 
}
\label{fig:opt_control_result}
\end{figure}
\subsection{Prevention of Blowup}
We would further like to highlight that our method is able to recover from near blowup.
\revC{By comparing with \Cref{subfig:blowup}, and manually tuning parameters,} we determined a function that
closely approximates the flowrate of the solution $\tilde \uu$
to the case $\uu_0 = \oo$ $q_1-q_2 \equiv 50$ presented in \Cref{subsect:blowup_state}.
We let the flowrate of $\uu_d$ follow this function until $t=0.82$, and set it constant afterwards.
Thus we set $\uu_d = \zeta(t) \ww$, where 
\begin{equation*}
  \zeta\colon [0,1] \to \R, \
  \zeta(t) = 5\left(\frac{1}{0.9-\min(t,0.82)}-\frac{1}{0.9}\right)+15\cdot\min(t,0.82).
\end{equation*}
We empirically determine the regularization parameter $\alpha = 10$, which compared to the previous examples
has to be chosen rather large, as the norm of the tracking type functional and thus the value of $J$ in 
this example is in the order of magnitude of $10^5$.
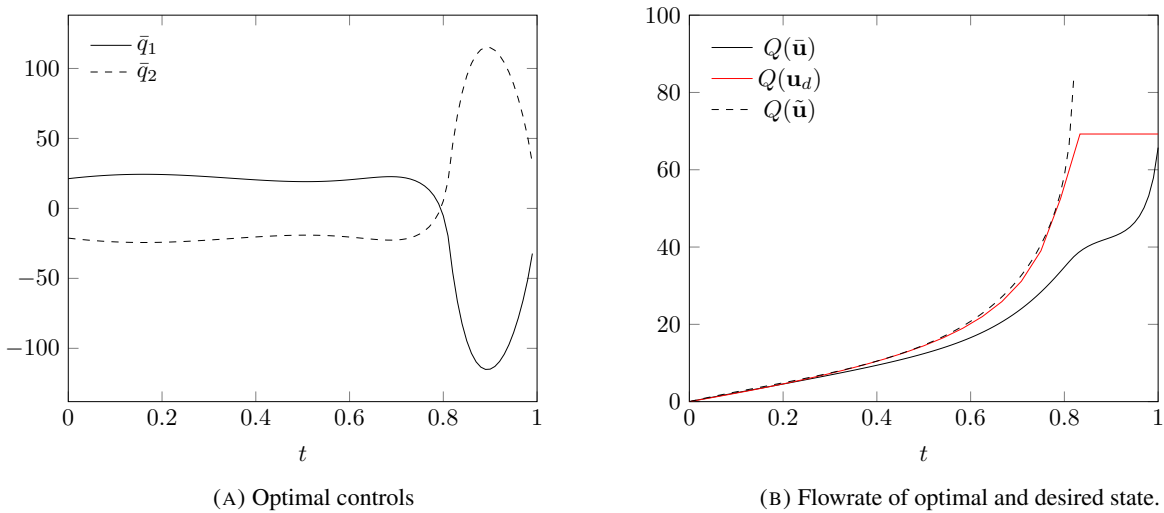
\begin{figure}[H]
  \begin{subfigure}{.49\linewidth}
  \begin{tikzpicture}[scale=.90]
    \begin{axis}[legend pos=north west, legend style={draw=none},legend columns=1,
                  xmin=0, xmax=1, xlabel={$t$}]
      \addplot[black] table[x expr=\coordindex/100,y expr=\thisrowno{0},col sep=space]{plots/control.00034.csv};
      \addplot[black,dashed] table[x expr=\coordindex/100,y expr=\thisrowno{1},col sep=space]{plots/control.00034.csv};
      \legend{{$\qbar_1$},{$\qbar_2$}}
    \end{axis}
  \end{tikzpicture}
  \caption{Optimal controls}
\end{subfigure}
  \begin{subfigure}{.5\linewidth}
  \begin{tikzpicture}[scale=.90]
    \begin{axis}[legend pos=north west, legend style={draw=none},
                  xmin=0, xmax=1, xlabel={$t$},ymin=0,ymax=100]
      \addplot[black] table[x expr=\thisrow{"Time"}/100,y="avg(U001)",col sep=comma]{plots/flowrate_blowup_alpha_10.csv};
      \addplot[red,domain=0:1] {5*(1/(.9-min(x,.82))-1/.9)+15*min(x,.82)};
      \addplot[black,dashed] table[x expr=\thisrow{"Time"}/100,y="avg(U001)",col sep=comma]{plots/Pdiff_blowup.csv};
      \legend{{$Q(\ubar)$},{$Q(\uu_d)$},{$Q(\tilde \uu)$}}
    \end{axis}
  \end{tikzpicture}
  \caption{Flowrate of optimal and desired state.}
\end{subfigure}
\caption{Results of an optimal control problem with data close to blowup.}
\end{figure}

\subsection{Bidirectional flow}
Lastly, we want to highlight, that since we did not rely on a boundary condition that assumes pure 
outflow, like \eqref{eq:DDN}, all open boundaries can either serve as inflow or outflow boundaries.
We thus set $\uu_d = 50\sin(2 \pi t) \ww$ and $\alpha=10^{-1}$.
The optimal control and the flowrates of $\uu_d$ as well as $\ubar$ can be taken from \Cref{fig:bidirectional}.
\begin{figure}[H]
  \begin{subfigure}{.49\linewidth}
  \begin{tikzpicture}[scale=1]
    \begin{axis}[legend pos=north west, legend style={draw=none},legend columns=1,
                  xmin=0, xmax=1, xlabel={$t$}]
      \addplot[black] table[x expr=\coordindex/100,y expr=\thisrowno{0},col sep=space]{plots/optControl_ud_sinus.csv};
      \addplot[black,dashed] table[x expr=\coordindex/100,y expr=\thisrowno{1},col sep=space]{plots/optControl_ud_sinus.csv};
      \legend{{$\qbar_1$},{$\qbar_2$}}
    \end{axis}
  \end{tikzpicture}
  \caption{Optimal controls}
\end{subfigure}
  \begin{subfigure}{.5\linewidth}
  \begin{tikzpicture}[]
    \begin{axis}[legend pos=north west, legend style={draw=none},
                  xmin=0, xmax=1, xlabel={$t$},ymin=-100,ymax=100]
      \addplot[black] table[x expr=\thisrow{"Time"}/100,y="avg(U001)",col sep=comma]{plots/flowrate_sinus_ud.csv};
      \addplot[red,domain=0:1] {50*sin(360*x)};
      \legend{{$Q(\ubar)$},{$Q(\uu_d)$}}
    \end{axis}
  \end{tikzpicture}
  \caption{Flowrate of the optimal vs desired state.}
\end{subfigure}
\caption{Solution of the optimal control problem with desired state that features a 
  sinusoidal dependence of time.}
\label{fig:bidirectional}
\end{figure}
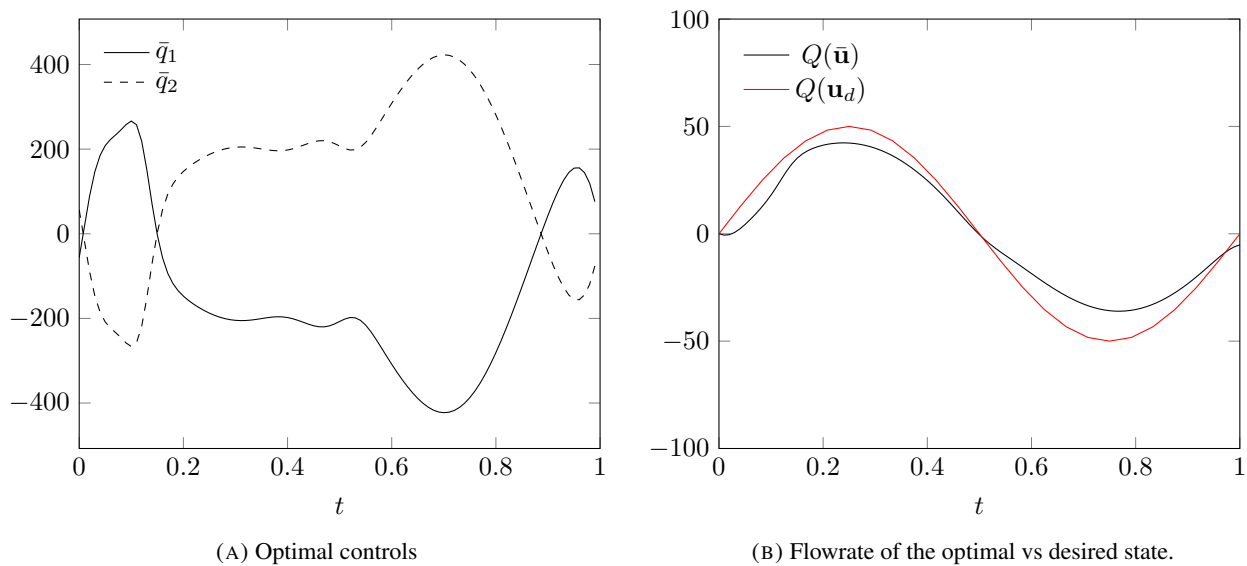
The examples presented above are chosen to highlight, that using \eqref{eq:DN} conditions does not only 
lead to issues in special domains but rather simple ones, and that our approach nonetheless is able to 
produce a control that compromises well between keeping the solution stable, but also steering them 
as close as possible to the desired data. 

%% file: appendix.tex
\appendix
\newpage

%% file: appendix_space_equivalence_proof.tex
\label{sect:app_proof_space_equivalence}
  We first discuss the $\Hspace$ spaces.
  \revA{As $\mathcal V \subset \Hspace$, $\Ht$ by definition is a closed subspace of $\Hspace$.}
  To show, that these spaces are actually equal, we show, that any linear functional on $\Hspace$, that vanishes on
  $\Ht$ also vanishes on $\Hspace$.
  Thus let $l \in \Hspace^*$ vanish on $\Ht$, i.e., it holds \revA{especially}
   \begin{equation*}
     \langle l, \vv \rangle = 0 \qquad \forall \vv \in \mathcal V.
   \end{equation*}
   As $l$ is a functional on $\Hspace \subset L^2(\Omega)^2$, we 
   can extend it to a functional on the whole $L^2(\Omega)^2$. Since it vanishes especially 
   for all testfunctions $\vv \in \mathcal V$, that are $\oo$ on the whole boundary,
   by \cite[Chapter I, Proposition 1.2]{Temam1977}, there exits $p \in H^1(\Omega)$, such that 
   \begin{equation*}
     \langle l, \vv\rangle =  \Oprod{\nabla p, \vv} \qquad \forall \vv \in L^2(\Omega)^2.
   \end{equation*}
   Hence it holds 
   \begin{equation*}
     0 = \Oprod{\nabla p, \vv} = - \Oprod{p,\nabla \cdot \vv} + \int_{\partial \Omega} p (\vv \cdot \nn) ds
     = \int_{\Gamma_N} p(\vv \cdot \nn) ds \qquad
     \forall \vv \in \mathcal V.
   \end{equation*}
   As so far, the representation of $l$ as $\nabla p$, only made use of testfunctions, that are $0$
   on the whole boundary,
   we now can show with the same arguments as in \cite{Heywood1996}, that in fact the pressure is 
   constant on $\Gamma_N$. For any two points $\xx_1 \neq \xx_2$, $\xx_{i} \in \Gamma_N$ let $\ttau_i$ and $\nn_i$ 
   for $i=1,2$
   denote a unit tangential vector and the outer normal vector on $\Gamma_N$. Fix $\varepsilon > 0$
   small enough, such that $\Sigma_i := \xx_i + (-\varepsilon,\varepsilon)\cdot \ttau_i \subset \Gamma_N$ $i=1,2$
   and $\Sigma_1 \cap \Sigma_2 = \emptyset$.
   Define $p_i: (-\varepsilon,\varepsilon) \to \R$, $p_i(t) = p(\xx_i + t \cdot \ttau_i)$, $i=1,2$.
   Now consider an arbitrary $\xi \in C^{\infty}_0 ((-\varepsilon,\varepsilon))$,
   and construct $v \in C^\infty(\bar \Omega)^2$ in such a way, that 
   \begin{equation*}
     \nabla \cdot \vv = 0 \ \text{in } \Omega, \quad \vv|_{\partial \Omega}(\xx) =
     \begin{cases}
       -\xi(t) \nn_1 & \text{if } \Sigma_1 \ni \xx = \xx_1 + t \ttau_1 \\
       \xi(t) \nn_2 & \text{if } \Sigma_2 \ni \xx = \xx_2 + t \ttau_2 \\
       0 & \text{else}.
     \end{cases}
   \end{equation*}
   For the proof of the existence of such a flux carrier, one can consider 
   \cite[Exercise III.3.5]{galdi_introduction_2011}.
   Thus one obtains
   \begin{equation*}
     0 = \int_{\Gamma_N} p(\vv \cdot \nn) ds 
     = \int_{\Sigma_1} p (\vv \cdot \nn) ds + \int_{\Sigma_2} p (\vv \cdot \nn) ds
     = \int_{-\epsilon}^\varepsilon p_1(t) \xi(t) dt - \int_{-\varepsilon}^\varepsilon p_2(t) \xi(t) dt
     = \int_{-\epsilon}^\varepsilon (p_1 - p_2)(t) \xi(t) dt
   \end{equation*}
   for all $\xi \in C^{\infty}_0 ((-\varepsilon,\varepsilon))$. Thus $p_1 = p_2$ on $(-\varepsilon,\varepsilon)$
   and thus $p|_{\Sigma_1} = p|_{\Sigma_2}$ (identifying the two boundary segments in the right orientation).
   As $\xx_1,\xx_2$ were arbitrary, this shows that $p|_{\Gamma_N} \equiv P$,
   for some constant $P$. If we define $\tilde p = p - P$, it holds $\tilde p \in H^1(\Omega)$ and
   $\tilde p|_{\Gamma_N} = 0$. Thus
   \begin{equation*}
     \langle l , \ww\rangle = (\nabla p,\ww) = (\nabla \tilde p,\ww) = - (\tilde p, \nabla \cdot \ww)
     + \int_{\Gamma_N} \tilde p (\ww\cdot \nn) ds + \int_{\Gamma_D} \tilde p (\ww \cdot \nn) ds = 0.
   \end{equation*}
     for all $\ww \in \Hspace$. Hence $l$ vanishes on the space $\Hspace$, proving $\Hspace = \Ht$.\\
     We now turn towards the spaces $\V$. Again it is easy to show that $\Vt \subset \V$. For the converse 
     direction, we need to show that every $\vv \in \V$ can be approximated by functions from $\mathcal V$.
     We sketch intermediate steps of a proof mentioned in 
     \cite[Remark 2.3]{nguyen_boundary_2015}, where the idea is to reduce the equivalence of spaces
     for the mixed boundary condition setting, to the one for pure Dirichlet boundary conditions.
     This requires a careful extension of the domain, i.e. definition of a domain 
     $\overline{\Omega^E} = \overline{\Omega} \cup \overline{\Omega^C}$.
     A sketch of such a domain can be seen in \Cref{fig:domain_extension}. Note that in order to 
     be able to extend any solenoidal function in $\V$ to a solenoidal function on the larger domain, 
     all Neumann boundaries have to be connected to each other.
     By \Cref{ass:domain}, each connected component of $\Gamma_{D,i}$ has a small neighborhood 
     $\mathcal U_i$, such that $\mathcal U_i \cap \Gamma_{D,j} = \emptyset$ for $j\neq i$.
     Each connected component of $\Gamma_{D,i}$ can be augmented to a closed, not-self intersecting 
     curve $\gamma_i$, such that $\gamma_i \subset \mathcal U_i$ and $\gamma_i \cup \bar \Omega = \Gamma_{D,i}$.
     Let us denote the domain enclosed by $\gamma_i$ by $\omega_i$.
     Define $R$ large enough, such that $\Omega \cup \bigcup_i \gamma_i \subset B_R(0)$.
     Then a possible extension is given by $\Omega^E := B_R(0) \backslash \bigcup_i \bar \omega_i$,
     and we define $\Omega^C := \interior(\Omega^E \backslash \Omega)$.
     We can split the boundary of $\Omega^C$ into three parts
     $\partial \Omega^C = \Gamma_N \cup \Gamma_C \cup \Gamma_B$, where $\Gamma_N$ is the original Neumann
     boundary of $\Omega$, $\Gamma^C = \bigcup (\gamma_i \backslash \Gamma_{D,i})$ and 
     $\Gamma_B = \partial B_R(0)$.
     On $\Omega^C$, we solve the following Stokes problem subject to pure Dirichlet conditions,
     see \cite[Theorem IV.1.1]{galdi_introduction_2011}
      \begin{equation*}
         \left\lbrace 
         \begin{aligned}
           -\Delta \ww + \nabla p_\ww &= \oo \quad \text{in } \Omega^C,\\
           \nabla \cdot \ww &= 0 \quad \text{in } \Omega^C,\\
           \ww &= \vv \quad \text{on } \Gamma_N,\\
           \ww &= \oo \quad \text{on } \Gamma_C \cup \Gamma_B.
         \end{aligned}
         \right.
      \end{equation*}
      Using integration by parts, and using that the traces of $\ww$ and $\vv$ coincide on $\Gamma_N$,
      it is straightforward to show, that the composite function
      \begin{equation*}
         \uu = \begin{cases}{}
            \vv & \xx \in \Omega\\
            \ww & \xx \in \Omega^C
         \end{cases}
      \end{equation*}
      is in $H^1_0(\Omega^E)^2$ and $\nabla \cdot \uu = 0$.
      As due to \cite[Chapter I, Theorem 1.6]{Temam1977}, there exists a sequence of functions
      $\{\uu_n\} \subset C^\infty_0(\Omega^E)$, such that $\nabla \cdot \uu_n = 0$ for all $n \in \N$ and 
      $\|\nabla (\uu-\uu_n)\|_{L^2(\Omega^E)} \xrightarrow{n\to \infty} 0$.
      This shows that $\|\nabla(\vv - \uu_n)|_{\Omega}\|_{L^2(\Omega)} \xrightarrow{n\to \infty} 0$,
      concluding the proof.

\begin{figure}[H]
  \centering
  \begin{tikzpicture}
  \node (image) at (0,0) {\includegraphics[width=7cm]{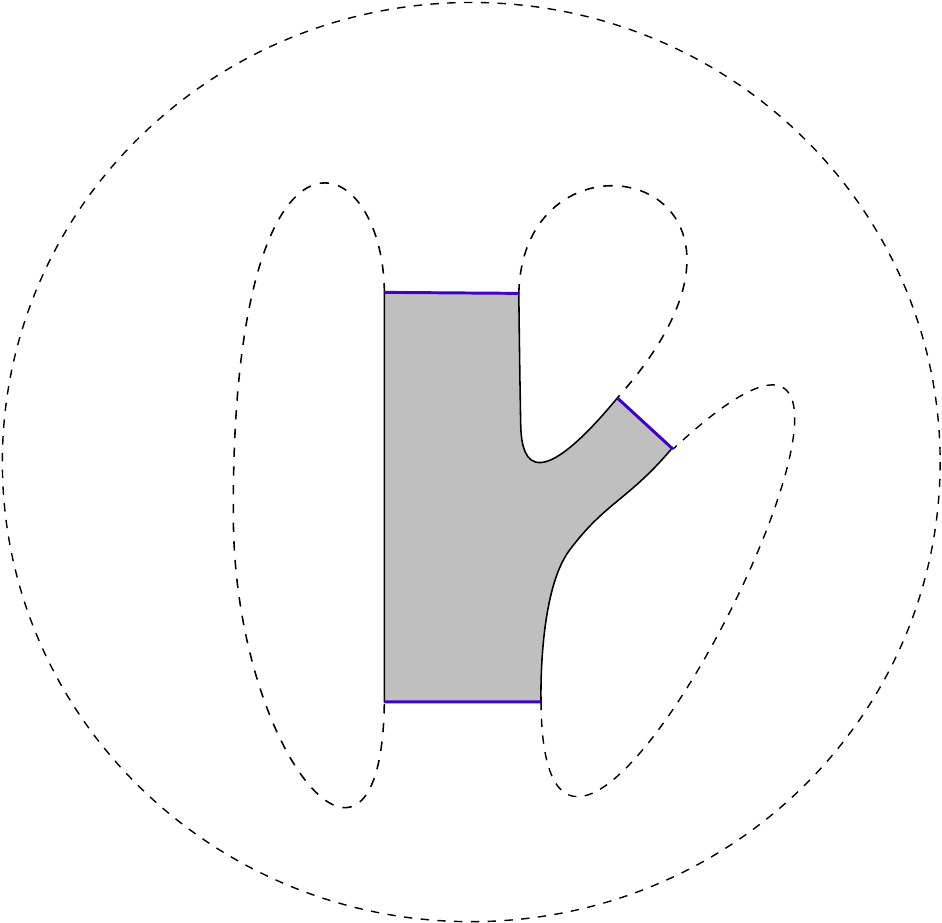}};
  \node at (-.1,-.2) {$\Omega$};
  \node[blue] at (-.1,-2.1) {\small $\Gamma_{N,1}$};
  \node[blue] at (-.1,1.5) {\small$\Gamma_{N,2}$};
  \node[blue] at (1.5,0.5) {\small $\Gamma_{N,3}$};
  \node at (-1.0,-.2) {\small $\Gamma_{D,1}$};
  \node at (1.1,-.9) {\small $\Gamma_{D,2}$};
  \node at (.75,.7) {\small $\Gamma_{D,3}$};
  \node[anchor=south west] at (-3.5,-3.0) {$B_R(0)$};
  \end{tikzpicture}
  \caption{Extension of $\Omega$, which admits an extension of divergence free vector fields
    from $\Omega$ to 
  the larger domain.}
  \label{fig:domain_extension}
\end{figure}

%% file: NS_Pressure_OC.bbl
\begin{thebibliography}{10}

\bibitem{adams_cone_1977}
{\sc R.~Adams and J.~Fournier}, {\em Cone conditions and properties of {Sobolev} spaces}, Journal of Mathematical Analysis and Applications, 61 (1977), pp.~713--734.

\bibitem{amrouche_decomposition_1994}
{\sc C.~Amrouche and V.~Girault}, {\em Decomposition of vector spaces and application to the {Stokes} problem in arbitrary dimension}, Czechoslovak Mathematical Journal, 44 (1994), pp.~109--140.

\bibitem{amrouche_stokes_2011}
{\sc C.~Amrouche and N.~E.~H. Seloula}, {\em Stokes equations and elliptic systems with nonstandard boundary conditions}, Comptes Rendus Mathematique, 349 (2011), pp.~703--708.

\bibitem{arndt_existence_2020}
{\sc R.~Arndt, A.~N. Ceretani, and C.~N. Rautenberg}, {\em On existence and uniqueness of solutions to a {Boussinesq} system with nonlinear and mixed boundary conditions}, Journal of Mathematical Analysis and Applications, 490 (2020), p.~124201.

\bibitem{banasiak_mixed_1989}
{\sc J.~Banasiak and G.~Roach}, {\em On mixed boundary value problems of {Dirichlet} oblique-derivative type in plane domains with piecewise differentiable boundary}, Journal of Differential Equations, 79 (1989), pp.~111--131.

\bibitem{behzadan_multiplication_2021}
{\sc A.~Behzadan and M.~Holst}, {\em Multiplication in {Sobolev} spaces, revisited}, Arkiv för Matematik, 59 (2021), pp.~275--306.

\bibitem{benes_mixed_2011}
{\sc M.~Beneš}, {\em Mixed initial-boundary value problem for the three-dimensional {Navier-Stokes} equations in polyhedral domains}, in Conference Publications 2011, {AIMS} Press.

\bibitem{benes_strong_2011}
\leavevmode\vrule height 2pt depth -1.6pt width 23pt, {\em Strong solutions to non-stationary channel flows of heat-conducting viscous incompressible fluids with dissipative heating}, Acta Applicandae Mathematicae, 116, pp.~237--254.

\bibitem{benes_solutions_2012}
{\sc M.~Beneš and P.~Kučera}, {\em Solutions of the {Navier–Stokes} equations with various types of boundary conditions}, Archiv der Mathematik, 98 (2012), pp.~487--497.

\bibitem{benes_solutions_2016}
\leavevmode\vrule height 2pt depth -1.6pt width 23pt, {\em Solutions to the {Navier–Stokes} equations with mixed boundary conditions in two‐dimensional bounded domains}, Mathematische Nachrichten, 289 (2016), pp.~194--212.

\bibitem{benes_local_2021}
{\sc M.~Beneš, P.~Kučera, and P.~Vacková}, {\em Local in time existence of solution of the {Navier-Stokes} equations with various types of boundary conditions}, Journal of Elliptic and Parabolic Equations, 7 (2021), pp.~297--310.

\bibitem{boffi_mixed_2013}
{\sc D.~Boffi, F.~Brezzi, and M.~Fortin}, {\em Mixed {Finite} {Element} {Methods} and {Applications}}, vol.~44 of Springer {Series} in {Computational} {Mathematics}, Springer Berlin Heidelberg, Berlin, Heidelberg, 2013.

\bibitem{bonnans_optimal_1999}
{\sc J.~F. Bonnans and H.~Zidani}, {\em Optimal {Control} {Problems} with {Partially} {Polyhedric} {Constraints}}, SIAM Journal on Control and Optimization, 37 (1999), pp.~1726--1741.

\bibitem{boyer_mathematical_2013}
{\sc F.~Boyer and P.~Fabrie}, {\em Mathematical {Tools} for the {Study} of the {Incompressible} {Navier}-{Stokes} {Equations} and {Related} {Models}}, vol.~183 of Applied {Mathematical} {Sciences}, Springer New York, New York, NY, 2013.

\bibitem{braack_directional_2014}
{\sc M.~Braack and P.~Mucha}, {\em Directional {Do}-{Nothing} {Condition} for the {Navier}-{Stokes} {Equations}}, Journal of Computational Mathematics, 32 (2014), pp.~507--521.

\bibitem{Brenner2008}
{\sc S.~C. Brenner and L.~Scott}, {\em The Mathematical Theory of Finite Element Methods}, Texts in Applied Mathematics, Springer New York, 2008.

\bibitem{brezis_gagliardo_2018}
{\sc H.~Brezis and P.~Mironescu}, {\em {Gagliardo-Nirenberg} inequalities and non-inequalities: The full story}, Annales de l'Institut Henri Poincar\'e C, Analyse non lin\'eaire, 25 (2018), pp.~1355--1376.

\bibitem{brown_mixed_2009}
{\sc R.~Brown, I.~Mitrea, M.~Mitrea, and M.~Wright}, {\em Mixed boundary value problems for the {Stokes} system}, Transactions of the American Mathematical Society, 362, pp.~1211--1230.

\bibitem{casas_optimal_1998}
{\sc E.~Casas}, {\em Chapter 4. {An} {Optimal} {Control} {Problem} {Governed} by the {Evolution} {Navier}-{Stokes} {Equations}}, in Optimal {Control} of {Viscous} {Flow}, S.~S. Sritharan, ed., Society for Industrial and Applied Mathematics, Jan. 1998, pp.~79--95.

\bibitem{casas_discontinuous_2012}
{\sc E.~Casas and K.~Chrysafinos}, {\em A {Discontinuous} {Galerkin} {Time}-{Stepping} {Scheme} for the {Velocity} {Tracking} {Problem}}, SIAM Journal on Numerical Analysis, 50 (2012), pp.~2281--2306.

\bibitem{casas_analysis_2016}
\leavevmode\vrule height 2pt depth -1.6pt width 23pt, {\em Analysis of the velocity tracking control problem for the 3d evolutionary {Navier}--{Stokes} equations}, {SIAM} Journal on Control and Optimization, 54 (2016), pp.~99--128.

\bibitem{casas_velocity_2024}
\leavevmode\vrule height 2pt depth -1.6pt width 23pt, {\em The velocity tracking problem for {Navier}–{Stokes} equations with pointwise-integral control constraints in time-space}, Optimization,  (2024), pp.~1--33.

\bibitem{casas_general_2012}
{\sc E.~Casas and F.~Tröltzsch}, {\em A general theorem on error estimates with application to a quasilinear elliptic optimal control problem}, Computational Optimization and Applications, 53 (2012), pp.~173--206.

\bibitem{ceretani_boussinesq_2019}
{\sc A.~N. Ceretani and C.~N. Rautenberg}, {\em The {Boussinesq} system with mixed non-smooth boundary conditions and do-nothing boundary flow}, Zeitschrift für angewandte Mathematik und Physik, 70, p.~14.

\bibitem{choi_optimal_2022}
{\sc J.~Choi, H.~Dong, and Z.~Li}, {\em Optimal {Regularity} of {Mixed} {Dirichlet}-{Conormal} {Boundary} {Value} {Problems} for {Parabolic} {Operators}}, SIAM Journal on Mathematical Analysis, 54 (2022), pp.~1393--1427.

\bibitem{demengel_functional_2012}
{\sc F.~Demengel and G.~Demengel}, {\em Functional Spaces for the Theory of Elliptic Partial Differential Equations}, Universitext, Springer London, 2012.

\bibitem{denk_introduction_2021}
{\sc R.~Denk}, {\em An {Introduction} to {Maximal} {Regularity} for {Parabolic} {Evolution} {Equations}}, in Nonlinear {Partial} {Differential} {Equations} for {Future} {Applications}: {Sendai}, {Japan}, {July} 10–28 and {October} 2–6, 2017, vol.~346 of Springer {Proceedings} in {Mathematics} \& {Statistics}, Springer Singapore, Singapore, 2021.

\bibitem{denk_optimal_2007}
{\sc R.~Denk, M.~Hieber, and J.~Prüss}, {\em Optimal {L} p -{L} q -estimates for parabolic boundary value problems with inhomogeneous data}, Mathematische Zeitschrift, 257 (2007), pp.~193--224.

\bibitem{ern_finite_2021_partI}
{\sc A.~Ern and J.-L. Guermond}, {\em Finite {Elements} {I}: {Approximation} and {Interpolation}}, vol.~72 of Texts in {Applied} {Mathematics}, Springer International Publishing, Cham, 2021.

\bibitem{ern_finite_2021}
\leavevmode\vrule height 2pt depth -1.6pt width 23pt, {\em Finite {Elements} {III}: {First}-{Order} and {Time}-{Dependent} {PDEs}}, vol.~74 of Texts in {Applied} {Mathematics}, Springer International Publishing, Cham, 2021.

\bibitem{galdi_introduction_2011}
{\sc G.~Galdi}, {\em An Introduction to the Mathematical Theory of the {Navier}-{Stokes} Equations}, Springer Monographs in Mathematics, Springer New York, 2011.

\bibitem{girault_finite_1986}
{\sc V.~Girault and P.-A. Raviart}, {\em Finite Element Methods for {Navier}-{Stokes} Equations}, vol.~5 of Springer Series in Computational Mathematics, Springer Berlin Heidelberg, 1986.

\bibitem{grisvard_caracterisation_1967}
{\sc P.~Grisvard}, {\em Caractérisation de quelques espaces d'interpolation}, Archive for Rational Mechanics and Analysis, 25 (1967), pp.~40--63.

\bibitem{grisvard_elliptic_2011}
\leavevmode\vrule height 2pt depth -1.6pt width 23pt, {\em Elliptic problems in nonsmooth domains}, no.~69 in Classics in applied mathematics, Society for Industrial and Applied Mathematics, Philadelphia, siam ed~ed., 2011.
\newblock OCLC: ocn746618649.

\bibitem{grubb_regularity_2016}
{\sc G.~Grubb}, {\em Regularity of spectral fractional {Dirichlet} and {Neumann} problems}, Mathematische Nachrichten, 289 (2016), pp.~831--844.

\bibitem{guerra_optimal_2024}
{\sc T.~Guerra, I.~Marín-Gayte, and J.~Tiago}, {\em An optimal boundary control problem related to the time dependent {Navier}-{Stokes} equations}, July 2024.
\newblock arXiv:2407.12561 [math].

\bibitem{guerra_existence_2015}
{\sc T.~Guerra, A.~Sequeira, and J.~Tiago}, {\em Existence of optimal boundary control for the {Navier–Stokes} equations with mixed boundary conditions}, Portugaliae Mathematica, 72, pp.~267--283.

\bibitem{guerra_optimal_2014}
{\sc T.~Guerra, J.~Tiago, and A.~Sequeira}, {\em Optimal control in blood flow simulations}, International Journal of Non-Linear Mechanics, 64, pp.~57--69.

\bibitem{gunzburger_analysis_1991}
{\sc M.~D. Gunzburger, L.~S. Hou, and T.~P. Svobodny}, {\em Analysis and finite element approximation of optimal control problems for the stationary {Navier}-{Stokes} equations with {Dirichlet} controls}, ESAIM: Mathematical Modelling and Numerical Analysis, 25 (1991), pp.~711--748.

\bibitem{hallerdintelmann_maximal_2009}
{\sc R.~Haller-Dintelmann and J.~Rehberg}, {\em Maximal parabolic regularity for divergence operators including mixed boundary conditions}, Journal of Differential Equations, 247 (2009), pp.~1354--1396.

\bibitem{Heywood1996}
{\sc J.~G. Heywood, R.~Rannacher, and S.~Turek}, {\em Artificial boundaries and flux and pressure conditions for the incompressible {Navier}-{Stokes} equations}, International Journal for Numerical Methods in Fluids, 22 (1996), pp.~325--352.

\bibitem{heywood_artificial_1996}
{\sc J.~G. Heywood, R.~Rannacher, and S.~Turek}, {\em {ARTIFICIAL} {BOUNDARIES} {AND} {FLUX} {AND} {PRESSURE} {CONDITIONS} {FOR} {THE} {INCOMPRESSIBLE} {NAVIER}–{STOKES} {EQUATIONS}}, International Journal for Numerical Methods in Fluids, 22 (1996), pp.~325--352.

\bibitem{hieber_quasilinear_2008}
{\sc M.~Hieber and J.~Rehberg}, {\em Quasilinear {Parabolic} {Systems} with {Mixed} {Boundary} {Conditions} on {Nonsmooth} {Domains}}, SIAM Journal on Mathematical Analysis, 40 (2008), pp.~292--305.

\bibitem{hintermuller_differentiability_2023}
{\sc M.~Hintermüller and A.~Kröner}, {\em Differentiability properties for boundary control of fluid-structure interactions of linear elasticity with {Navier-Stokes} equations with mixed-boundary conditions in a channel}, Applied Mathematics \& Optimization, 87, p.~15.

\bibitem{hinze_second_2004}
{\sc M.~Hinze and K.~Kunisch}, {\em Second order methods for boundary control of the instationary {Navier}-{Stokes} system}, Zeitschrift für angewandte Mathematik und Mechanik, 84 (2004), pp.~171--187.

\bibitem{hoppe_optimal_2022}
{\sc F.~Hoppe and I.~Neitzel}, {\em Optimal {Control} of {Quasilinear} {Parabolic} {PDEs} with {State}-{Constraints}}, SIAM Journal on Control and Optimization, 60 (2022), pp.~330--354.

\bibitem{john_finite_2016}
{\sc V.~John}, {\em Finite {Element} {Methods} for {Incompressible} {Flow} {Problems}}, vol.~51 of Springer {Series} in {Computational} {Mathematics}, Springer International Publishing, Cham, 2016.

\bibitem{kato_extension_2000}
{\sc T.~Kato, M.~Mitrea, G.~Ponce, and M.~Taylor}, {\em Extension and {Representation} of {Divergence}-free {Vector} {Fields} on {Bounded} {Domains}}, Mathematical Research Letters, 7 (2000), pp.~643--650.

\bibitem{kim_non_stationary_2017}
{\sc T.~Kim and D.~Cao}, {\em Non-stationary {Navier-Stokes} equations with mixed boundary conditions}, Journal of Mathematical Sciences The University of Tokio, 24 (2017), pp.~159--194.

\bibitem{kim_existence_2018}
\leavevmode\vrule height 2pt depth -1.6pt width 23pt, {\em Existence of solution to parabolic equations with mixed boundary condition on non-cylindrical domains}, Journal of Differential Equations, 265 (2018), pp.~2648--2670.

\bibitem{kim_equations_2021}
\leavevmode\vrule height 2pt depth -1.6pt width 23pt, {\em Equations of {Motion} for {Incompressible} {Viscous} {Fluids}: {With} {Mixed} {Boundary} {Conditions}}, Advances in {Mathematical} {Fluid} {Mechanics}, Springer International Publishing, Cham, 2021.

\bibitem{kozlov_neumann_2023}
{\sc V.~Kozlov and J.~Rossmann}, {\em On the {Neumann} problem for the nonstationary {Stokes} system in angles and cones}, Mathematische Nachrichten, 296, pp.~1504--1533.

\bibitem{kracmar_modeling_2018}
{\sc S.~Kračmar and J.~Neustupa}, {\em Modeling of the unsteady flow through a channel with an artificial outflow condition by the {Navier–Stokes} variational inequality}, Mathematische Nachrichten, 291 (2018), pp.~1801--1814.

\bibitem{kunisch_constrained_2007}
{\sc K.~Kunisch and B.~Vexler}, {\em Constrained {Dirichlet} {Boundary} {Control} in \${L}{\textasciicircum}2\$ for a {Class} of {Evolution} {Equations}}, SIAM Journal on Control and Optimization, 46 (2007), pp.~1726--1753.
\newblock Publisher: Society for Industrial and Applied Mathematics.

\bibitem{kucera_basic_2009}
{\sc P.~Kučera}, {\em Basic properties of solution of the non-steady {Navier–Stokes} equations with mixed boundary conditions in a bounded domain}, {ANNALI} {DELL}'{UNIVERSITA}' {DI} {FERRARA}, 55, pp.~289--308.

\bibitem{kucera_local_1998}
{\sc P.~Kučera and Z.~Skalak}, {\em Local solutions to the {Navier–Stokes} equations with mixed boundary conditions}, Acta Applicandae Mathematicae, 54, pp.~275--288.

\bibitem{lanzendorfer_multiple_2020}
{\sc M.~Lanzendörfer and J.~Hron}, {\em On {Multiple} {Solutions} to the {Steady} {Flow} of {Incompressible} {Fluids} {Subject} to {Do}-nothing or {Constant} {Traction} {Boundary} {Conditions} on {Artificial} {Boundaries}}, Journal of Mathematical Fluid Mechanics, 22 (2020), p.~11.

\bibitem{lanzendorfer_pressure_2011}
{\sc M.~Lanzendörfer and J.~Stebel}, {\em On pressure boundary conditions for steady flows of incompressible fluids with pressure and shear rate dependent viscosities}, Applications of Mathematics, 56 (2011), pp.~265--285.

\bibitem{lasiecka_boundary_2018}
{\sc I.~Lasiecka, K.~Szulc, and A.~Żochowski}, {\em Boundary control of small solutions to fluid–structure interactions arising in coupling of elasticity with {Navier–Stokes} equation under mixed boundary conditions}, Nonlinear Analysis: Real World Applications, 44, pp.~54--85.

\bibitem{lindemulder_maximal_2020}
{\sc N.~Lindemulder}, {\em Maximal regularity with weights for parabolic problems with inhomogeneous boundary conditions}, Journal of Evolution Equations, 20 (2020), pp.~59--108.

\bibitem{lions_nonhomogeneous_1972_vol1}
{\sc J.~L. Lions and E.~Magenes}, {\em Non-{Homogeneous} {Boundary} {Value} {Problems} and {Applications}, {Volume} 1}, vol.~1, Springer Berlin Heidelberg, Berlin, Heidelberg, 1972.

\bibitem{lions_nonhomogeneous_1972_vol2}
\leavevmode\vrule height 2pt depth -1.6pt width 23pt, {\em Non-{Homogeneous} {Boundary} {Value} {Problems} and {Applications}, {Volume} 2}, vol.~2, Springer Berlin Heidelberg, Berlin, Heidelberg, 1972.

\bibitem{marschall_trace_1987}
{\sc J.~Marschall}, {\em The trace of {Sobolev-Slobodeckij} spaces on {Lipschitz} domains}, Manuscripta Mathematica, 58, pp.~47--65.

\bibitem{mazya_l_2007}
{\sc V.~Maz'ya and J.~Rossmann}, {\em \textit{{L}} $_{\textrm{ \textit{p} }}$ estimates of solutions tomixed boundary value problems for the {Stokes} system in polyhedral domains}, Mathematische Nachrichten, 280 (2007), pp.~751--793.

\bibitem{mazia_elliptic_2010}
{\sc V.~G. Maz'ya and J.~Rossmann}, {\em Elliptic equations in polyhedral domains}, no.~v. 162 in Mathematical surveys and monographs, American Mathematical Society, Providence, R.I, 2010.
\newblock OCLC: ocn499129712.

\bibitem{mitrea_stokes_2011}
{\sc M.~Mitrea, S.~Monniaux, and M.~Wright}, {\em The {Stokes} operator with {Neumann} boundary conditions in {Lipschitz} domains}, Journal of Mathematical Sciences, 176, pp.~409--457.

\bibitem{mitrea_boundary_2012}
{\sc M.~Mitrea and M.~Wright}, {\em Boundary value problems for the {Stokes} system in arbitrary {Lipschitz} domains}, no.~344 in Astérisque, Soc. Math. de France, Paris, 2012.

\bibitem{nguyen_boundary_2015}
{\sc P.~A. Nguyen and J.-P. Raymond}, {\em Boundary {Stabilization} of the {Navier}--{Stokes} {Equations} in the {Case} of {Mixed} {Boundary} {Conditions}}, SIAM Journal on Control and Optimization, 53 (2015), pp.~3006--3039.

\bibitem{beirao_da_veiga_navierstokes_2024}
{\sc P.~Nogueira and A.~L. Silvestre}, {\em Navier–{Stokes} {Equations} with {Regularized} {Directional} {Boundary} {Condition}}, in Nonlinear {Differential} {Equations} and {Applications}, H.~Beirão Da~Veiga, F.~Minhós, N.~Van~Goethem, and L.~Sanchez~Rodrigues, eds., vol.~7, Springer International Publishing, Cham, 2024, pp.~197--220.
\newblock Series Title: CIM Series in Mathematical Sciences.

\bibitem{nolte_inverse_2022}
{\sc D.~Nolte and C.~Bertoglio}, {\em Inverse problems in blood flow modeling: {A} review}, International Journal for Numerical Methods in Biomedical Engineering, 38 (2022), p.~e3613.
\newblock \_eprint: https://onlinelibrary.wiley.com/doi/pdf/10.1002/cnm.3613.

\bibitem{quarteroni_coupling_2001}
{\sc A.~Quarteroni, S.~Ragni, and A.~Veneziani}, {\em Coupling between lumped and distributed models for blood flow problems}, Computing and Visualization in Science, 4 (2001), pp.~111--124.

\bibitem{rathore_projection_based_2024}
{\sc S.~Rathore, P.~C. Africa, F.~Ballarin, F.~Pichi, M.~Girfoglio, and G.~Rozza}, {\em Projection-based {Reduced} {Order} {Modelling} for {Unsteady} {Parametrized} {Optimal} {Control} {Problems} in {3D} {Cardiovascular} {Flows}}, Oct. 2024.
\newblock arXiv:2410.20828 [math].

\bibitem{cialdea_mixed_2009}
{\sc J.~Rossmann}, {\em Mixed boundary value problems for {Stokes} and {Navier-Stokes} systems in polyhedral domains}, in Analysis, Partial Differential Equations and Applications, A.~Cialdea, P.~E. Ricci, and F.~Lanzara, eds., Birkhäuser Basel, pp.~269--280.

\bibitem{rosch_regularity_2005}
{\sc A.~Rösch and D.~Wachsmuth}, {\em Regularity of the {Adjoint} {State} for the {Instationary} {Navier}-{Stokes} {Equations}}, Zeitschrift für Analysis und ihre Anwendungen, 24 (2005), pp.~103--116.

\bibitem{savare_parabolic_1997}
{\sc G.~Savaré}, {\em Parabolic problems with mixed variable lateral conditions: {An} abstract approach}, Journal de Mathématiques Pures et Appliquées, 76 (1997), pp.~321--351.

\bibitem{seeley_interpolation_1972}
{\sc R.~Seeley}, {\em Interpolation in ${L}^p$ with boundary condition}, Studia Mathematica, 44 (1972), pp.~47--60.

\bibitem{mondaini_boundary_2018}
{\sc A.~Sequeira, J.~Tiago, and T.~Guerra}, {\em Boundary control problems in hemodynamics}, in Trends in Biomathematics: Modeling, Optimization and Computational Problems, R.~P. Mondaini, ed., Springer International Publishing, pp.~27--48.

\bibitem{simon_compact_1986}
{\sc J.~Simon}, {\em Compact {Sets} in the {Space} ${L}^p(0,{T};{B})$}, Annali di Matematica pura ed applicata, 146 (1986), pp.~65--96.

\bibitem{simon_convective_2022}
{\sc J.~S.~H. Simon and H.~Notsu}, {\em A convective boundary condition for the {Navier}–{Stokes} equations}, Applied Mathematics Letters, 128 (2022), p.~107876.

\bibitem{sohr_navier_stokes_2001}
{\sc H.~Sohr}, {\em The {Navier}-{Stokes} {Equations}}, Springer Basel, Basel, 2001.

\bibitem{stankovic_4d_2014}
{\sc Z.~Stankovic, B.~D. Allen, J.~Garcia, K.~B. Jarvis, and M.~Markl}, {\em {4D} flow imaging with {MRI}}, Cardiovascular Diagnosis and Therapy, 4 (2014), pp.~17392--17192.
\newblock Number: 2 Publisher: AME publishing company.

\bibitem{Temam1977}
{\sc R.~Temam}, {\em {Navier}-{Stokes} Equations: Theory and Numerical Analysis}, North-Holland Publ. Co., Amsterdam, 1977.

\bibitem{tiago_velocity_2017}
{\sc J.~Tiago, T.~Guerra, and A.~Sequeira}, {\em A velocity tracking approach for the data assimilation problem in blood flow simulations}, International Journal for Numerical Methods in Biomedical Engineering, 33, p.~e2856.

\bibitem{triebel_interpolation_1978}
{\sc H.~Triebel}, {\em Interpolation {Theory}, {Function} {Spaces}, {Differential} {Operators}}, North-Holland Publishing Company, Amsterdam - New York - Oxford, 1978.

\bibitem{triebel_theory_1992}
\leavevmode\vrule height 2pt depth -1.6pt width 23pt, {\em Theory of Function Spaces {II}}, Springer Basel, 1992.

\bibitem{vexler_error_2024}
{\sc B.~Vexler and J.~Wagner}, {\em Error estimates for finite element discretizations of the instationary {Navier}–{Stokes} equations}, ESAIM: Mathematical Modelling and Numerical Analysis, 58 (2024), pp.~457--488.

\end{thebibliography}
